\RequirePackage{lineno}
\pdfoutput=1
\documentclass[aip,cha,reprint,nofootinbib,floatfix,longbibliography]{revtex4-1}
\usepackage[margin=1in]{geometry}
\usepackage{xcolor}
\usepackage[sort&compress]{natbib} % nicer bibliography

\usepackage[final]{hyperref}
\hypersetup{
pdftitle={Applied Koopmanism},
pdfauthor={Budisic, Mohr, Mezic},
bookmarksopen=True,
bookmarksnumbered=True
}
\usepackage[open,numbered]{bookmark}
\usepackage{url}

\usepackage{amsmath}
\usepackage{amsthm}
\usepackage{amssymb}
\usepackage{mathtools}
\usepackage[mathscr]{eucal}

\newenvironment{DIFnomarkup}{}{}

\newcommand*\patchAmsMathEnvironmentForLineno[1]{%
  \expandafter\let\csname old#1\expandafter\endcsname\csname #1\endcsname
  \expandafter\let\csname oldend#1\expandafter\endcsname\csname end#1\endcsname
  \renewenvironment{#1}%
     {\linenomath\csname old#1\endcsname}%
     {\csname oldend#1\endcsname\endlinenomath}}% 
\newcommand*\patchBothAmsMathEnvironmentsForLineno[1]{%
  \patchAmsMathEnvironmentForLineno{#1}%
  \patchAmsMathEnvironmentForLineno{#1*}}%
\AtBeginDocument{%
\patchBothAmsMathEnvironmentsForLineno{equation}%
\patchBothAmsMathEnvironmentsForLineno{align}%
\patchBothAmsMathEnvironmentsForLineno{flalign}%
\patchBothAmsMathEnvironmentsForLineno{alignat}%
\patchBothAmsMathEnvironmentsForLineno{gather}%
\patchBothAmsMathEnvironmentsForLineno{multline}%
}

\usepackage{graphicx}
\usepackage{color,psfrag}
\graphicspath{ {./img/} }
\usepackage{xcolor}
\usepackage[hypcap]{caption}
\usepackage{subcaption}
\usepackage{epstopdf}

\usepackage{xkvltxp}
\usepackage[status=final]{fixme}
\fxusetheme{colorsig}

\usepackage{paralist}

\newtheorem{thm}{Theorem} 
 
\newtheorem{prop}[thm]{Proposition} 
\theoremstyle{definition}
\newtheorem{example}[thm]{Example} 
\newtheorem{defn}[thm]{Definition} 
\newtheorem{rem}[thm]{Remark}

\newcommand{\field}[1]{ \ensuremath{\mathbb{#1}}}
\newcommand{\T}{\ensuremath{\field{T}}}
\newcommand{\C}{ \ensuremath{\field{C}}}
\newcommand{\R}{ \ensuremath{\field{R}}}
\newcommand{\Q}{ \ensuremath{\field{Q}}}
\newcommand{\N}{ \ensuremath{\field{N}}}
\newcommand{\Z}{ \ensuremath{\field{Z}}}

\newcommand{\eps}{\epsilon}

\newcommand{\mc}{\mathcal}
\newcommand{\msf}{\mathsf}
\newcommand{\mbf}{\boldsymbol} % requires amsmath, amssymb packages
\newcommand{\mbb}{\mathbb}

\newcommand{\norm}[1]{ \ensuremath{\left\lVert{#1}\right\rVert}}
\newcommand{\avg}[1]{ \ensuremath{\left\langle{#1}\right\rangle}}
\newcommand{\abs}[1]{ \ensuremath{\left\lvert{#1}\right\rvert}}
\newcommand{\inner}[2]{\ensuremath{\left\langle #1, #2 \right\rangle}}

\DeclareMathOperator{\linspan}{span}

\newcommand{\favg}[2][\omega]{ \ensuremath{\tilde {#2}^{(#1)} }}

\begin{document}
\clearpage
\title{Applied Koopmanism}
\thanks{We dedicate this paper to the memory of Jerrold E.\ (Jerry) Marsden whose support of our research was invaluable.}
\author{Marko \surname{Budi\v{s}i\'c}}
\author{Ryan \surname{Mohr}}
\author{Igor \surname{Mezi\'c}}
\affiliation{Department of Mechanical Engineering, University of California, Santa Barbara}
\date{Nov 29, 2012}
\begin{abstract}
  % 500 word limit, no paragraphs, no math expressions, no citations or footnotes
  A majority of methods from dynamical systems analysis, especially those in applied settings, rely on Poincar\'e's geometric picture that focuses on ``dynamics of states''. While this picture has fueled our field for a century, it has shown difficulties in handling high-dimensional, ill-described, and uncertain systems, which are more and more common in engineered systems design and analysis of ``big data'' measurements.
  This overview article presents an alternative framework for dynamical systems, based on the ``dynamics of observables'' picture. The central object is the Koopman operator: an infinite-dimensional, linear operator that is nonetheless capable of capturing the full nonlinear dynamics. The first goal of this paper is to make it clear how methods that appeared in different papers and contexts all relate to each other through spectral properties of the Koopman operator. The second goal is to present these methods in a concise manner in an effort to make the framework accessible to researchers who would like to apply them, but also, expand and improve them. Finally, we aim to provide a road map through the literature where each of the topics was described in detail.
  We describe three main concepts: Koopman mode analysis, Koopman eigenquotients, and continuous indicators of ergodicity. For each concept we provide a summary of theoretical concepts required to define and study them, numerical methods that have been developed for their analysis, and, when possible, applications that made use of them.
  The Koopman framework is showing potential for crossing over from academic and theoretical use to industrial practice. Therefore, the paper  highlights its strengths,  in applied and numerical contexts. Additionally, we point out areas where an additional research push is needed before the approach is adopted as an off-the-shelf framework for analysis and design.
\end{abstract}

\maketitle
\begin{quotation} 
%\begin{anfxerror}[author=MB]{Lead paragraph} 
%Lead paragraph:Short non-jargon summary of the entire submission
%  \end{anfxerror}
A majority of methods from dynamical systems analysis, especially those in applied settings, rely on Poincar\'e's geometric picture that focuses on ``dynamics of states''. While this picture has fueled our field for a century, it has shown difficulties in handling high-dimensional, ill-described, and uncertain systems, which are more and more common in engineered systems design and analysis of ``big data'' measurements.
  This overview article presents an alternative framework for dynamical systems, based on the ``dynamics of observables'' picture.
We present an overview of several approaches to studying dynamical systems using the Koopman operator, which holds promise to resolve these issues. The dynamics are analyzed by looking at evolutions of functions on the state space, rather than directly at state space trajectories.  The evolution can be understood by expanding the function into a basis of eigenfunctions of the Koopman operator. The first approach is based on the Koopman modes, which generalize linear mode analysis from linear systems to nonlinear systems, while preserving global nonlinear features of the system, unlike, e.g., linearizations based on Taylor- and Fourier- expansions. The second approach identifies coherent structures in flows. An equivalence relation between points in the state space can be defined using spectral properties of the Koopman operator, where equivalent points correspond to initial conditions that behave statistically the same with respect to any observable. The third approach we present introduces continuous quantifications of ergodicity and mixing, concepts existing in ergodic theory that are traditionally treated as binary notions.  Throughout the paper, we highlight examples from the literature using each of these concepts.  Examples are taken from diverse areas such as fluid mechanics, fluid mixing, energy efficiency of buildings, power systems, and Unmanned Aerial Vehicle path-planning for search-and-rescue. A common trait of all the methods is that they do not require access to an analytical model of the system; the spectral properties of the Koopman operator can be constructed from measured or simulated data.
\end{quotation}

\section{Introduction}\label{sec:introduction}
Currently, dynamical systems analysis and design primarily uses the geometric picture, as put forth by Poincar\'e in his work on the three body problem.  Much of the framework is built around notions from differential geometry, trajectories and invariant manifolds.  Such an approach has met with success in a variety of settings and, at this point, one hardly needs to justify the use of geometric theory when working on a particular problem.  

However, the geometric viewpoint is ill-suited to many of the situations that are of interest in real systems.  For example, for systems possessing hyperbolic regimes, the unstable manifolds give rise to locally exponentially divergent trajectories.
Any noise or uncertainty in the system will lead to multiple possible trajectories for an initial condition, with the width of the set trajectories initially expanding exponentially.  In such cases, questions about the behavior of a specific trajectory are difficult to answer.

Systems with a large number of dimensions can be problematic as well, since many of the geometric arguments are only valid in a low number of dimensions, e.g., Bendixson's Criterion for determining the non-existence of periodic orbits in the plane.  In some cases, these arguments can be extended, with difficulty, to an arbitrary number of dimensions.  Even in these cases, however, a practical implementation limits them to a moderate number of dimensions.  To handle high-dimensional systems, special symmetries or other conditions are required in order to effectively reduce the dimension to a manageable size.  Furthermore, without access to explicit ODEs, even basic geometric analysis is difficult to apply.  If dynamical systems theory is to become an important field in the context of pressing problems such as ``big data'', tools need to be developed that are capable of handling high-dimensional, uncertain, and ill-described systems, as well as systems for which past time-evolution data is available, but for which no simple mathematical description can be determined.\cite{Jones:2001tj}

This article presents a viewpoint that is at the intersection of applied ergodic theory and operator theory.  These two fields can be used in applied settings to analyze and design dynamical systems, with many of the aforementioned difficulties being handled with a certain amount of elegance.  In fact, when we study dynamical systems through certain linear operators, the \textit{full nonlinear dynamics can be captured within a linear setting}.  This linear setting allows the power of spectral analysis to be brought to bear on a (nonlinear) problem without sacrificing any information as required by other linearization techniques.  Contrast this with the traditional spectral approach that only determines geometry locally in the state space.  Additionally, in theory, the operator-theoretic approach works equally well whether the original state-space is low- or high-dimensional; the same techniques apply to both cases.  The framework is also well suited to studying noisy systems because the primary object of interest is no longer the trajectory.  Finally, and perhaps most significantly, the operators involved can be constructed,  approximated, or analyzed using only simulation or experimental measurement data.  This allows a certain black-box approach to the analysis which is quite useful in real problems where the practitioner may not have full knowledge of the system's internals.

As with any technique, however, there is a tradeoff in order to gain the above advantages.  The operator-theoretic picture has no immediate connection to our physical intuition, making its meaning more difficult to comprehend.  One's viewpoint must change from considering the evolution of points in the state space to considering the evolution of functions.
Additionally, the new approach is inherently infinite-dimensional, even when the state space is finite-dimensional.  This sacrifice is what allows the full information of a nonlinear system to be contained within a linear setting.  Because of this, the implementation of any approximation is a more delicate issue.  Finally, the associated numerical techniques are underdeveloped.  Most of our approaches employ direct computations which are little more than numerical implementations of proofs.

The two main candidates for the study of systems via operators are the Koopman operator and the Perron-Frobenius operator.  In appropriate function spaces, they are duals to each other, so theoretically, there should not be any distinction in working with one as opposed to the other.  However, as mentioned previously, we must always include applied considerations.  Questions arise such as how do we construct or represent the chosen operator from the problem description and given data?  How well does a finite approximation represent the ideal theoretical picture?  What part of intuition gained is due to numerical artifacts and what is real?  

The Perron-Frobenius operator represents a ``dynamics of densities'' picture; it looks at groups of trajectories.  One can think of this as watching the evolution of a mass distribution under the action of a flow.  From a numerical perspective, construction of the operator relies on selecting a set of initial conditions and simulating forward for only a short time period, thus avoiding the compounding of numerical time-integration errors.  Due to these short bursts, transient dynamics can be captured very well. However, much attention has been focused on computing invariant densities,\cite{Dellnitz:1999tr} which are infinite-time objects, through approximating the Perron-Frobenius operator by a Markov chain.  The number of simulated initial conditions is dictated by the need to sample the region of interest well.  In high-dimensions, both short- and long-time dynamics simulations require a mesh on the entire space.  This can be true even in the case of a low-dimensional attractor.  If we have \emph{a priori} knowledge of the low-dimensional subspace the attractor lives in, then the mesh size can be restricted.  However, for an arbitrary system, this knowledge may not be initially available, thus requiring the full mesh.

On the other hand, the Koopman operator presents a picture for the ``dynamics of observables''.  The difference in viewpoints between the Perron-Frobenius and Koopman operators is similar to the Eulerian versus the Lagrangian viewpoint in fluid mechanics, with the Koopman picture corresponding to the Lagrangian viewpoint. What is meant is that measurements are made along trajectories.  For the Koopman operator, the numerical construction relies on potentially fewer initial conditions, but requires longer run-times, which is more suitable to physical experiments.  For example, when testing a jet engine, it is started from a relatively small number of initial conditions and run it over a long time rather than preparing thousands of initial conditions and running the engine for a few seconds for each initial condition.  Due to the long run-times required, the asymptotics are well-understood.  However, more research is needed to understand the transients.

To visualize high-dimensional dynamical systems, we often restrict our attention to one, or a few, two-dimensional cross-sections in the state space and look at the invariant structures intersecting that slice.  With the Perron-Frobenius operator, it is difficult to directly compute invariant densities on the slice of interest, since, in principle, it requires a computation of the invariant density for the entire state space as an intermediate step.  For the Koopman operator, invariant objects are attached to initial conditions, making it well-suited to visualizing structures on an arbitrary 2D cross-section in the state space.  Initial conditions can be easily prepared on the slice and the invariants directly computed.  In such cases, the number of initial conditions required to understand the dynamics is significantly reduced.

While operator methods, and specifically the Koopman and Perron-Frobenius operators, have much potential to deal with applied problems, these methods are all but absent from the applied and industrial settings, with due exceptions.\cite{Mezic:2000tm,Mezic:2004is,Eisenhower:2010tv}  To speed up the adoption of operator techniques in these domains, any new methodology needs to be able to leverage already existing data, instead of proposing both a new methodology \textit{and} a new way to collect data.  The ``dynamics of observables'' perspective, and specifically the Koopman operator, is used as it deals with measurements, i.e., observables, which are well-understood both theoretically and computationally.  On the other hand, the Perron-Frobenius techniques would require working with representations of densities, which are often singular, especially in well-behaved engineered systems.

In this paper, we intend to describe three concepts, all under the umbrella of ``dynamics of observables'', that show how this theory can be made useful for analysis and design.  Contributions can be split between theoretical and applied contributions.
%of the theory can be split into

\subsection*{Theoretical contributions}

1. Dynamical evolution of a system can be studied by looking at what is termed \emph{Koopman mode analysis}.  The concept is similar to normal mode analysis familiar from linear vibration theory.  Koopman mode analysis starts with a choice of a set of linearly independent observables, or equivalently a vector-valued observable. The Koopman operator \(U\) is then analyzed through its action on the subspace spanned by the chosen observables.  The observables are decomposed into projections onto the eigenspaces of $U$, and the evolution is a sum of terms composed of a product of three terms:
  \begin{inparaenum}[i)]
    \item a part that is time-dependent and is determined by the eigenvalue (or frequency) associated with the eigenspace;
    \item an eigenfunction of $U$, which is a function of the initial conditions;
    \item the vector of the coefficients of the projection of the observables onto the eigenspaces, with the coefficients only being functions of the chosen observables.
  \end{inparaenum}In this way, spectral analysis can be performed on nonlinear systems.  This analysis is also used for model reduction.\cite{Mezic:2004is,Mezic:2005ji}

2. The notions of the \emph{ergodic quotients and eigenquotients} allow the Koopman operator to be used for the extraction and analysis of invariant and periodic structures in the state space.\cite{Budisic:2009iy}  The points in the state space are grouped into invariant sets using level sets of eigenfunctions of the Koopman operator. Instead of set-theoretic framework, this approach is lifted to the analytical setting using the ergodic quotient and eigenquotient formalism. The eigenquotients are studied as subsets of particular Sobolev spaces, where their geometry gives insight into the structure of the state space, in spirit similar to analysis of Hamiltonian systems via Morse theory of the associated energy functions.

3. In the standard interpretation of ergodic theory, mixing and ergodicity are treated as binary concepts: a system is either mixing/ergodic or it is not. Both mixing and ergodicity can be formulated using spectral invariants of the Koopman operator. From a finite-time evolution of an arbitrary system, we can quantify how close its spectral invariants are to the ``ideal'' case, e.g., mixing or ergodic, and in this way formulate \emph{continuous indicators of ergodicity and mixing}.  The ergodicity defect\cite{Scott:2009wh} and the mixing norm\cite{Mathew:2005eq} are examples of continuous indicators that extend the corresponding binary notions.  Such a relaxation brings the concepts of ergodicity and mixing into an engineering context, allowing, e.g., the use of the indicators as optimization criteria. We present a unified explanation of the concepts that have previously appeared in literature.\cite{Mathew:2005eq, Mathew:2007vq, Mathew:2009et, Mathew:2011ev,Scott:2009wh,Rypina:2011ec}

\subsection*{Numerical techniques and applications}

Numerical computation of the objects in the theory uses elements from three different areas.  Fourier analysis based methods are useful for computing Koopman modes for dynamics on the attractor in addition to being essential for the construction of the eigenquotient.  A variant of the standard Arnoldi algorithm based on companion matrices is also useful for computing part of the spectrum of the Koopman operator, a basic element of Koopman mode analysis.  This variant does not require an explicit representation of the operator and only requires data, sequences of vectors coming from either simulations or experiments.

Koopman mode analysis has seen applications in fluids mechanics to extract spatial structures for the flow.\cite{Rowley:2009ez,Schmid:2010ba,Seena:2011ft,Chen:2012jh}  Koopman modes have also found applications in the analysis of coherency and instabilities for power systems\cite{Susuki:2010ei,Susuki:2011ef,Susuki:2011jq,Susuki:2012gk} and in the field of building energy efficiency where they have been used for model validation and data analysis.\cite{Eisenhower:2010tv,Georgescu:2012tt}

To compute the eigenquotients numerically, a set of observables is averaged along trajectories started at different initial conditions, obtaining a finite-dimensional representation of any eigenquotient. Such representations are analyzed with the aid of a diffusion maps algorithm, \cite{Coifman:2005bk,Coifman:2006cy} which computes a change of coordinates, \emph{the diffusion modes}, for the eigenquotient. In the limit where infinitely many initial conditions were simulated for infinite time, such coordinate change would  render any consequent analysis independent of the choice of the observables averaged during the computation. The scale-ordering of diffusion coordinates makes it  practical to obtain a low-dimensional approximation of the eigenquotients. 

The averaging along trajectories is used again to formulate continuous indicators for ergodicity and mixing, where the rate of approach of the averages along finite-time trajectories to the infinite limit is indicative of the underlying dynamics. Based on such indicators, the dynamics of the flows can be designed to match a particular statistical behavior, with applications in path-planning for vehicles\cite{Mathew:2009et} and mixing of fluids on micro-scales.\cite{Mathew:2007vq}

\section{Notation and Terminology}
\label{sec:notation}
%\fxnote*[author=]{}{The methodology that we present is at the intersection of operator theory and applied ergodic theory.  Standards texts on ergodic theory deal with measure-preserving dynamical system whose transformations are either invertible are nonsingular ($\mu(A) = 0$ implies $\mu(T^{-1}A) = 0$).  We do not require that the dynamical system preserves a measure, thus presenting a viewpoint that includes the dissipative case.}

We start off by fixing some notation and terminology.  Let us denote the state space by $M$ and define dynamics on it by the iterated map $T:M \to M$.  Note that the set $M$ can be an arbitrary set (possessing no structure) and $T$ can be an arbitrary map on this set.  Then the abstract dynamical system is specified by the couple $(M,T)$. Note that standard texts on ergodic theory study a specific case when \(M\) is a measurable space, with a $\sigma$-algebra \(\mathfrak B\), and \(T\) is \(\mathfrak B\)-measurable.  Additionally, transformation \(T\) is typically assumed to be measure-preserving, i.e.,  there exists a measure \(\mu\), \emph{the invariant measure}, such that for any \(S \in \mathfrak B\) 
\begin{equation}
  \label{eq:invariant-measure}
  \mu(S) = \mu( T^{-1} S ),
\end{equation}
with \(T^{-1}S\) understood as the pre-image of $S$. The measure \(\mu\) does not necessarily have a density function associated with it. For the formulation of the theory in this paper, we do not require the measurable framework, although when answering more specific questions, we might restrict ourselves to it, as it is the one most commonly encountered in applied dynamical systems.

We will be concerned with the behavior of observables on the state space.  To this end, we define an observable to be a function $f : M \to \C$, where $f$ is an element of some function space $\mathcal F$.  For now, it is not necessary to specify any structure for $\mathcal F$.  A concrete interpretation of an observable is that of a sensor probe for the dynamical system in question; we can access information about the system via the evolution of the observable's values.  Instead of tracking the trajectory $\{ p, T(p), T^{2}(p),\dots \}$, we now track the trace $\{ f(p), f(T(p)), f(T^{2}(p)), \dots \}$.  The description of the dynamics can then be concisely written down in the form of state and output equations, familiar to control theorists:
	\begin{equation}
	\begin{aligned}
  	p_{n+1} &= T(p_n) \\
  	v_{n} &= f(p_n).
	\end{aligned}\label{eq:state-output}
	\end{equation}
The dynamical systems community mainly focuses on state space trajectories $\{p_n\}$, while the control systems community usually studies systems that have an additional input or disturbance terms, with the functions $T$ and $f$ taking particular forms that are common in engineered systems. %Here we offer a general dynamical analysis, not relying on internal structure of either $T$ or $f$, with which we can capture full nonlinear behavior of the state equation through linear analysis of the output trajectory.

%The Koopman operator formalizes the above observable evolution.  
We define the (discrete-time) Koopman operator, $U_T : \mathcal F \to \mathcal F$, as 
	\begin{equation}
	[U_Tf](p) = f( T(p) ), \label{eq:koopman}
	\end{equation}
i.e., it is a composition, $U_T f = f \circ T$, of the observable $f$ and the iterated map $T$.  When it is obvious which transformation gives rise to the Koopman operator, we will drop the dependence on $T$ from the notation and write $U$ instead of $U_T$.  When $\mathcal F$ is a vector space, $U$ is a linear operator. 

When \(M\) is a finite set, $U$ is a finite-dimensional operator and can be represented by a matrix. However, when \(M\) is finite- or infinite-dimensional, \(U\) is generally infinite dimensional. Much of the time, we only have access to a particular collection of observables $\{ f_{1},\dots, f_{K} \} \subset \mathcal F$; these could be physically relevant observables arising naturally from the problem, or a (subset of a) function basis for $\mathcal F$.  We can extend the Koopman operator to this larger space in the natural way:  If $F = (f_1, \dots, f_K)^{\msf T} \in \mathcal F^{K}$, then $U_K : \mathcal F^{K} \to \mathcal F^{K}$ is defined as
	\begin{equation}
	[U_K F](p) := \begin{bmatrix}
	[Uf_{1}](p) \\
	\vdots \\
	[Uf_{K}](p)
	\end{bmatrix}. \label{eq:koopman-vv}
	\end{equation}
Hence $U_K = \bigotimes_{1}^{K} U$.  With an abuse of notation, we generally write $U_K$ as $U$.  The space $\mathcal F^{K}$ is the space of $\C^{K}$-valued observables on $M$.  In this context, $\C^{K}$ is referred to as the output space.  More generally, we can consider vector valued observables, $F : M \to V$, where $V$ is some vector space.  For example, when analyzing the heat equation on a periodic box $\mbb B$, the state space can be regarded as the sequence space of Fourier coefficients and an observable $F:M \to L^{2}(\mbb B,dx)$ can be regarded as mapping between a sequence of Fourier coefficients (the state space $M$) and a temperature distribution on $\mbb B$ (the real-valued space $L^2(\mbb B,dx)$). We will revisit this setup in more detail in Example \ref{ex:heat}.

The above notion of the Koopman operator was defined in the context of discrete-time dynamical systems.  Often though, working in the discrete-time setting poses an unneeded restriction; in many systems, the natural formulation of the dynamics is with respect to a continuous time variable.  The Koopman operator can be extended to deal with continuous-time dynamical systems, or even more generally, event-based dynamical systems.

Assume we have the continuous-time dynamical system $\dot{p} = T(p)$.  In this context, there is not just \emph{the} Koopman operator, but a semigroup of operators $\{ U^{t} \}_{t\in\R^{+}}$ given by a generator $U$.  We call  the semigroup $\{ U^{t} \}$ the \emph{Koopman semigroup}.  We explicitly define the action of the semigroup on the observable $f \in \mathcal F$ as
	\begin{equation}
	[U^{t}f](p) = f( \Phi^{t}(p) ).
	\end{equation}
Here $\Phi^{t}(p) \equiv \Phi(p,t)$ is the flow map that takes an initial condition $p \in M$ and maps it to the solution at time $t$ of the initial value problem (IVP) having initial condition $p(0) =p$; i.e., for a fixed $p_{0} \in M$, the trajectory $\{ \Phi(p_{0}, \cdot) \}_{t\geq 0}$ is a solution of the IVP $\dot{p} = T(p),\, p(0) = p_{0}$.  The generator of the Koopman semigroup is defined by
	\begin{equation}
	[Uf] := \lim_{t \to 0} \frac{U^{t}f - f}{t},
	\end{equation}
where the limit is taken in the strong sense.\cite{Lasota:1994vt}

The following examples describe the above concepts in certain simple, concrete cases.
\begin{example}[Cyclic group]
Let $M = \{ e, a, a^{2} \}$ be a cyclic group of order 3 $(a^3 \equiv e)$.  Define $T:M \to M$ by $T(p) = a \cdot p$.  Hence the entire state space is a periodic orbit of period 3.  Let $\mathcal F$ be the $\C$-valued functions on $M$.  Clearly, the space of observables is $\C^3$. Let $f_{1},f_{2},f_{3}$ be the indicator functions on $e,a,a^{2}$, respectively:
	\begin{equation}
	\begin{aligned}
	f_{1}(p) &= \begin{cases}
	1, & p=e,\\
	0, & p\neq e
	\end{cases}, \\
	f_{2}(p) &= \begin{cases}
	1, & p=a,\\
	0, & p\neq a
	\end{cases}, \\
	f_{3}(p) &= \begin{cases}
	1, & p=a^2,\\
	0, & p\neq a^2
	\end{cases}.
	\end{aligned}
	\end{equation}
These form a basis for $\mathcal F$.  The action of the Koopman operator on this basis is given as

	\begin{equation}
	\begin{aligned}
	&[Uf_{1}](p) = f_{1}(a\cdot p) = f_{3}(p), \\ 
	&[Uf_{2}](p) = f_{2}(a\cdot p) = f_{1}(p), \\ 
	&[Uf_{3}](p) = f_{3}(a\cdot p) = f_{2}(p).
	\end{aligned}
	\end{equation}
For an arbitrary observable $f \in \mc F$ given by $f = c_{1}f_{1} + c_{2}f_{2} + c_{3}f_{3}$, with $c_{i} \in \C$, we have that 
	\begin{equation*}
	Uf =  c_{1} f_{3} + c_{2} f_{1} + c_{3} f_{2}.  
	\end{equation*}
Then, the matrix representation of $U$ in the $\{f_1, f_2, f_3\}$ basis is given by
	\begin{equation}\label{eq:koopman-ex-group}
	U\begin{bmatrix} c_{1} \\ c_{2} \\ c_{3} \end{bmatrix}
	 = \begin{bmatrix}
	0 & 1 & 0 \\
	0 & 0 & 1 \\
	1 & 0 & 0
	\end{bmatrix}
	\begin{bmatrix}
	c_{1} \\
	c_{2} \\
	c_{3} 
	\end{bmatrix}.
	\end{equation}
In this case, \eqref{eq:koopman-ex-group} gives the full action of the Koopman operator on $\mathcal F$.
 \qed\end{example}

\begin{example}[Linear, diagonalizable systems]\label{ex:diagonalizable}
Let $M = \R^{d}$ and define $T:M \to M$ by
	\begin{equation}
	(T(\mbf x))_{i} = \mu_{i} x_{i},
	\end{equation}
where $\mbf x = (x_{1},\dots,x_{d})^{\msf T} \in M$ and $\mu_{i} \in \R$.  Let $\mc F$ be the space of $\C$-valued functions on $\R^{d}$.  Let $\{ \mbf b_{1}, \dots, \mbf b_{d} \} \subset M$ be a basis for $M$ and define $f_{i}(\mbf x) = \inner{\mbf b_{i}}{\mbf x}$, where $\inner{\cdot}{\cdot}$ is the inner product on $\R^{d}$.  The action of $U:\mc F \to \mc F$ on $f_{i}$ is
	\begin{equation}
	\begin{aligned}
	&[Uf_{i}](\mbf x) = \inner{\mbf b_{i}}{T(\mbf x)}=  [b_{i,1}, \dots, b_{i,d} ] \begin{bmatrix}
	\mu_{1} x_{1} \\
	\vdots \\
	\mu_{d} x_{d}
	\end{bmatrix} \\
	&\qquad= 
	[b_{i,1}, \dots, b_{i,d} ] \begin{bmatrix}
	\mu_{1} & 0 & \cdots & 0 \\
	0 & \mu_{2} & \cdots & 0 \\
	\vdots & \vdots & \ddots & \vdots \\
	0 & 0 & \cdots & \mu_{d} 
	\end{bmatrix}
	\begin{bmatrix} x_{1} \\ \vdots \\ x_{d} \end{bmatrix} .
	\end{aligned}
	\end{equation}
Let $\mc F^{d} = \bigotimes_{1}^{d} \mathcal F$ and define $U_{d}$ on $\mc F^{d}$ as in \eqref{eq:koopman-vv}.  Then for $F = (f_{1},\dots, f_{d})^{\msf T}$,
	\begin{align} \label{eq:koopman-rd-example-vector-valued}
	\begin{aligned}
	&[U_{d} F](\mbf x) \\
	&\quad = 
	\begin{bmatrix}
	b_{1,1} & \cdots & b_{1,d} \\
	\vdots & \ddots & \vdots \\
	b_{d,1} & \cdots & b_{d,d}
	\end{bmatrix} 
	\begin{bmatrix}
	\mu_{1} & 0 & \cdots & 0 \\
	0 & \mu_{2} & \cdots & 0 \\
	\vdots & \vdots & \ddots & \vdots \\
	0 & 0 & \cdots & \mu_{d} 
	\end{bmatrix}
	\begin{bmatrix} x_{1} \\ \vdots \\ x_{d} \end{bmatrix}.
	\end{aligned} 
	\end{align}
Note that \eqref{eq:koopman-rd-example-vector-valued} is just the action of the Koopman operator on the \emph{particular} observable $F = (f_{1},\dots,f_{d})^{\msf T}$, not the full action of the Koopman operator on the entire observable space, $\mathcal F$. The main point is that the Koopman operator is reducible, provided $T$ leaves some subspace of $\mc F$ invariant.
	
As a special case, we can take the $\mbf b_{i}$'s to be the canonical basis vectors having zeros everywhere, except for a 1 in the $i^{th}$ entry.  In this case, the functions $\{f_{i}\}$ represent the canonical projections onto the coordinates of $\mbf x$; i.e. $f_{i}(\mbf x) = x_{i}$.  Then the action of the Koopman operator with respect to these \emph{particular} observables is given as
	\begin{equation}
	[U_d F](\mbf x) = 
	\begin{bmatrix}
	\mu_{1} & 0 & \cdots & 0 \\
	0 & \mu_{2} & \cdots & 0 \\
	\vdots & \vdots & \ddots & \vdots \\
	0 & 0 & \cdots & \mu_{d} 
	\end{bmatrix} 
	\begin{bmatrix} x_{1} \\ \vdots \\ x_{d} \end{bmatrix}.
	\end{equation}
 \qed\end{example}

The next example shows that the Koopman operator formalism can easily handle state spaces that are mixtures of discrete and continuous domains.

\begin{example}[Mixed state space]
Let $\T^2 = [0,2\pi)\times[0,2\pi)$ and $G = \{0,1,2\}$. Define $M = \T \times G$ and let the dynamics be given by
	\begin{equation}
	\begin{aligned}
	I_{k+1} &= I_k + \left( \frac{s_k}{2}\right) K \sin \theta_k, \mod 2\pi \\
	\theta_{k+1} &= \theta_k + I_{k+1}, \mod 2\pi \\
	s_{k+1} &=  s_k + 1, \mod 3
	\end{aligned}
	\end{equation}
where $\mbf p_k = (I_k,\theta_k) \in \T^2$, $s_k \in G$, and $K > 0$.  The dynamics are given by the standard map cycling between the unperturbed, shear-flow case and two perturbed cases.  The ``perturbation dynamics'' are driven by a group action.

Let $\mc F_{\T^2} = \{ f : \T^2 \to \C\}$ be the set of all functions mapping the torus into the complex numbers.  We do not assume that the functions in $\mc F_{\T^2}$ have any type of regularity or algebraic properties; for the moment they are completely arbitrary.  Similarly, let $\mc F_{G} = \{ g : G \to \C\}$ be the set of all functions from the group $G$ into the complex numbers.   One possible choice for the space of observables $\mc F$ on $M$ is the set $\mc F = \{ h = g\cdot f \mid f \in \mc F_{\T^2}, g \in \mc F_{G} \}$.  Hence observables on $M$ are pointwise products of functions on $\T^2$ and $G$ and map the mixed state space into $\C$.  The Koopman operator can easily be defined as 
	\begin{equation*}
	[Uh](\mbf p_k,s_k) = g(s_{k+1})\cdot f(\mbf p_{k+1}).
	\end{equation*}
Another possible choice for $\mc F$ could be the set of all the observables that are functions of only $I$ and $\theta$.  This a a subset of the previous choice by taking $g$ to be a constant function.  This is a natural choice in the case that the ``perturbation dynamics'' (the dynamics on $G$), cannot be measured.

\qed\end{example}

%\begin{example}[Mixed state space]
%Assume $r_{k} \in \N \cup \{ 0 \} =: \N_0$ has dynamics of a renewal process
%	%%%       Equation       %%%
%	\begin{equation}\label{eq:renewal-process}
%	\begin{aligned}
%	\Pr [r_{k+1} &= r_k -1 \mid r_k > 0] = 1 \\
%	\Pr [r_{k+1} &= n \in \N \mid r_k =0] = p_n
%	\end{aligned}
%	\end{equation}
%	%%%%%%%%%%%%%%
%where $0 < p_{n}$ and $\sum p_n = 1$.  Define $M = [0,2\pi]^2 \times \N_0$ and the dynamics by
%	%%%       Equation       %%%
%	\begin{equation}
%	\begin{aligned}
%	I_{k+1} &= I_k + K\left(1 - \frac{r_k}{1+r_k}\right)\sin \theta_k, \mod 2\pi \\
%	\theta_{k+1} &= \theta_n + I_{k+1}, \mod 2\pi
%	\end{aligned}
%	\end{equation}
%	%%%%%%%%%%%%%%
%with $\{r_k\}$ following \eqref{eq:renewal-process}.  Note that $M$ is a mixture of discrete and continuous spaces and the dynamics on $M$ form a skew product system.  The dynamics are just the normal standard map with a stochastic perturbation parameter.  To visualize the state-space, take the state-space of the Standard map with different perturbation parameters and stack them on top of each other.
%
%Let $\mc F$ be the observables on $M$ defined as the set $\mc F = \{ (f, s) \mid f : [0,2\pi]^2 \to \C, s:\N_0 \to \C \}$.  Hence functions in $\mc F$ map $M$ into $\C^2$.  The Koopman operator can easily be defined as 
%	%%%       Equation       %%%
%	\begin{equation*}
%	[Uf](p_k,r_k) = (f(p_{k+1}), s(r_{k+1}))
%	\end{equation*}
%	%%%%%%%%%%%%%%
%where $p_k := (I_k,\theta_k)$.  Another choice for $\mc F$ would be all the observables that are only functions of $I$ and $\theta$.
% \qed\end{example}

\begin{example}[Partial differential equations]
\label{ex:heat}
Consider the 2D heat equation on $\mbb B = [-\frac{1}{2}, \frac{1}{2}]\times[-\frac{1}{2},\frac{1}{2}]$ with periodic boundary conditions:
	\begin{equation}
	\frac{\partial \mbf u(\mbf x,t)}{\partial t} = c^{2}\, \nabla^{2} \mbf u(\mbf x,t).
	\end{equation}
Assuming $\mbf u, \nabla^{2} \mbf u \in L^{2}(\mbb B,d\mbf x)$, $\mbf u(\mbf x,t)$ can be expanded in a trigonometric basis:
	\begin{equation}\label{eq:heat-expansion}
	\mbf u (\mbf x,t) = \sum_{\mbf j \in \Z^{2}} a_{\mbf j}(t) e^{i 2\pi \mbf j\cdot \mbf x},
	\end{equation}
where $\mbf j \cdot \mbf x$ is the dot product of $\mbf j$ and $\mbf x$.
A Galerkin projection onto this basis yields 
	\begin{equation}
	\dot{a}_{\mbf j}(t) = -4\pi^{2}c^{2} \norm{\mbf j}_{2}^{2} a_{\mbf j}(t).
	\end{equation}
Thus, we have the continuous-time, infinite-dimensional analogue of example \ref{ex:diagonalizable}. We could proceed with exhibiting the Koopman semigroup or induce a discrete-time evolution from the continuous-time flow.  In the latter case, fix a time step $h>0$ and get
	\begin{equation}\label{eq:heat-induced-discrete}
	a_{\mbf j}(t_{n+1}) = \exp(-4\pi^{2}c^{2} \norm{\mbf j}_{2}^{2}h) \, a_{\mbf j}(t_{n}), 
	\end{equation}
where $t_{n} := nh$.	

The state space, $M = \ell^{2}(\Z^{2})$, is the space of Fourier coefficients; if $\mbf a \in M$, then $\mbf a = (a_{\mbf j_{1}}, a_{\mbf j_{2}}, \dots)$, where $(\mbf j_{1},\mbf j_{2},\dots)$ is some ordering of $\Z^{2}$. Let \eqref{eq:heat-induced-discrete} define  the induced discrete-time evolution map, $T_{h} : M \to M$.  At this point, we could exactly reduce this problem to example \ref{ex:diagonalizable} by restricting our attention to a finite-dimensional subspace of $M$.   

Let the observable space, $\mathcal F$, be the $\C$-valued functions on $M$.  A family of observables on $M$, parameterized by $\mbf x\in \mbb B$, is given by \eqref{eq:heat-expansion}; namely, fixing $\mbf x\in \mbb B$, we have for any $\mbf a \in M$
	\begin{equation}\label{eq:temp-observable}
	f_{\mbf x}(\mbf a) = \sum_{\mbf j \in \Z^{2}} a_{\mbf j} e^{i 2\pi \mbf j\cdot \mbf x}.
	\end{equation}
Hence, the temperature at a point $\mbf x \in \mbb B$ is a linear observable on the space of Fourier coefficients. 

Assume the temperature can be measured at a finite number of points $(\mbf x_{1},\mbf x_{2},\dots,\mbf x_{K})$ in $\mbb B$ and take the finite collection of observables $(f_{\mbf x_{1}},\dots,f_{\mbf x_{K}})$ defined by \eqref{eq:temp-observable}.  Then the action of the Koopman operator on this set of observables is
	\begin{equation}\label{eq:heat-koopman}
	U \begin{bmatrix}
	f_{\mbf x_{1}}(\mbf a) \\
	\vdots \\
	f_{\mbf x_{K}}(\mbf a)
	\end{bmatrix}
	= \sum_{\mbf j\in \Z^{2}} \mu_{\mbf j}\,  a_{\mbf j} \begin{bmatrix}
	 e^{i 2\pi \mbf j\cdot \mbf x_{1}}  \\
	\vdots \\
	e^{i 2\pi \mbf j\cdot \mbf x_{K}}
	\end{bmatrix} .
	\end{equation}
where $\mu_{\mbf j} = \exp(-4\pi^{2}c^{2} \norm{\mbf j}_{2}^{2}h)$.
 \qed\end{example}

%%% Local Variables: 
%%% mode: latex
%%% TeX-master: "../koopmanism"
%%% End: 
\section{Koopman mode analysis}\label{sec:km-analysis}

\subsection{Eigenfunctions and Koopman Modes of \texorpdfstring{$\mbf U$}{%
U
}%
}
\label{subsec:koopman-modes}
% fxnote, fxwarning, fxerror, and fxfatal can all be used
Thus far, we have avoided putting structure on the function space $\mc F$.  When $\mc F$ is a vector space, the Koopman operator is linear.  It, therefore, makes sense to study its spectral properties as this will give us insight into the dynamics of the system, similar to the case of linear finite-dimensional systems.  We make the further assumptions that $\mc F$ is a Banach space under some norm, $\norm{\cdot}$, and that $U$ is a bounded, and hence continuous, operator on this space.  

Let $\{ \phi_{1}, \dots, \phi_{n}\}$ be a set of eigenfunctions of $U$, where $n=1,2, \dots,$ or $\infty$, not necessarily forming a complete basis set for $\mathcal F$.  In the discrete-time case, we have that
	\begin{equation}\label{eq:eigenvector-discrete-time}
	[U\phi_{i}](p) = \lambda_{i}\phi_{i}(p).
	\end{equation}
In the continuous-time case, the $\lambda$'s are eigenvalues of the generator $U$ of the Koopman semigroup, $\{U^{t}\}$.  The eigencondition is then
	\begin{equation}\label{eq:eigenvector-cont-time}
	[U^{t} \phi_{i}](p) = e^{\lambda_{i} t} \phi_{i}(p),
	\end{equation}
so that $\{e^{\lambda_{i}}\}$ are the eigenvalues for the Koopman semigroup.  

We first note two simple properties of eigenfunctions: their algebraic structure (Prop. \ref{prop:efunc-semigroup}) and their role in spectral equivalence of the systems (Prop. \ref{prop:top-conj}).

\begin{prop}[Algebraic structure of eigenfunctions under products]
\label{prop:efunc-semigroup}
Assume $\mc F$ is a subset of all $\C$-valued functions on $M$ that forms a vector space which is closed under pointwise products of functions.  Then, the set of eigenfunctions forms an Abelian semigroup under pointwise products of functions.  In particular, if $\phi_{1}, \phi_2 \in \mc F$ are eigenfunctions of $U$ with eigenvalues $\lambda_1$ and $\lambda_2$, then $\phi_1 \phi_2$ is an eigenfunction of $U$ with eigenvalue $\lambda_1 \lambda_2$.  

Furthermore, if $p \in \R^+$ and $\phi$ is an eigenfunction with eigenvalue $\lambda$, then $\phi^{p}$ is a eigenfunction with eigenvalue $\lambda^p$, where $\phi^p(x) := (\phi(x))^p$.  If $\phi$ is an eigenfunction that vanishes nowhere and $r \in \R$, then $\phi^r$ is an eigenfunction with eigenvalue $\lambda^r$.  The eigenfunctions that vanish nowhere form an Abelian group.
\end{prop}

\begin{proof}
Assume $U\phi_1 = \lambda_1 \phi_1$ and $U\phi_2 = \lambda_2 \phi_2$ and put $\psi(x) = \phi_1(x)\phi_2(x)$.  In discrete time, 
	\begin{align*}
	[U\psi](x) &= \psi(T(x)) = \phi_1(T(x))\phi_2(T(x)) \\
	&= [U\phi_1](x)\, [U\phi_2](x) = \lambda_1 \lambda_2 \phi_1(x)\phi_2(x) \\
	&= \lambda_1 \lambda_2 \psi(x).
	\end{align*}
Hence, the set of eigenfunctions is closed under pointwise products.  An analogous computation holds for continuous time.

Note that constant functions are eigenfunctions at eigenvalue 1.  Hence the constant function that is equal to 1 everywhere is an eigenfunction of $U$ and acts as the identity element.  Combining this with the above closure property and standard properties of pointwise products of functions shows that the set of eigenfunctions is an Abelian semigroup.

Let $U\phi = \lambda\phi$ and fix $p \in \R^+$.  Then
	\begin{align*}
	[U\phi^p](x) &= \phi^p(Tx) = (\phi(Tx))^p = ( \lambda \phi(x) )^p \\
	&= \lambda^p \phi^p(x).
	\end{align*}
If $\phi$ vanishes nowhere, then $\phi^{-1}(x) := 1 /\phi(x)$ is well-defined.  Then, the above chain of identities remains valid for $r \in \R$ replacing $p \in \R^+$.  Hence, the Abelian semigroup of eigenfunctions also contains all of its inverses.  Therefore, the set of eigenfunctions vanishing nowhere is an Abelian group.
\end{proof}

\begin{example}[Analytic observables of stable/unstable systems]
\label{ex:decaying-analytic}
Let $\dot{x} = \lambda x$, with $x, \lambda \in \C$ and $\abs{\lambda} \neq 1$.  Then $\Phi^t(x) = e^{\lambda t}x$.  Let $\phi(x) = x$.  Then
	\begin{align*}
	[U^{t} \phi](x)& = \phi(\Phi^{t}(x)) = \phi(e^{\lambda t}x) \\
	&= e^{\lambda t}x = e^{\lambda t} \phi(x),
	\end{align*}
which implies that $\phi$ is an eigenfunction of $U$.  By proposition \ref{prop:efunc-semigroup}, any $\phi_n(x) := (\phi(x))^n = x^n$ is an eigenfunction of $U$ with eigenvalue $\lambda^n$.

Let $f(x)$ be an analytic function.  Then $f(x) = \sum c_n x^n = \sum c_n \phi_n(x)$, where $c_n = \frac{1}{n!}\frac{d^n f(0)}{dx^n}$.  Therefore,
	\begin{equation*}
	[Uf](x) = \sum c_n [U\phi_n](x) = \sum \lambda^n c_n \phi_n(x).
	\end{equation*}
\qed\end{example}

The second property shows the spectral equivalence of topologically conjugate transformations.
\begin{prop}[Spectral equivalence of topologically conjugate systems]\label{prop:top-conj}
Let $S: M \to M$ and $T: N \to N$ be topologically conjugate; i.e., there exists a homeomorphism $h : N \to M$ such that $S \circ h = h \circ T$.  If $\phi$ is an eigenfunction of $U_S$ with eigenvalue $\lambda$, then $\phi \circ h$ is an eigenfunction $U_T$ at eigenvalue $\lambda$.
\end{prop}

\begin{proof}
Fix $x \in M$ and let $y\in N$ be such that $x = h(y)$.  The result follows from the chain of equalities:
	\begin{align*}
	\lambda (\phi \circ h)(y) &=  
    \lambda \phi(x) = [U_S \phi](x) = \phi[ S(x) ] \\ 
    &= \phi\{ S[h(y)] \} = \phi\{ h[T(y) ] \} \\
    &= [U_T (\phi \circ h)](y).
	\end{align*}
\end{proof}

\begin{example}[Topological conjugacy of diagonalizable systems]
Let $\mbf y^{(k)} = (y_1^{(k)},y_2^{(k)})^{\msf T}$, where the superscript $(k)$ indexes time, and let $\mbf y^{(k+1)} = T\mbf y^{(k)}$, where $T$ is a matrix.  Assume that $T$ has eigenvectors $\mbf v_1,\mbf v_2$ at eigenvalues $\lambda_1,\lambda_2$ such that $\mbf v_i \neq \mbf e_j$, where $\mbf e_j$ is the canonical basis vector.  If $V = [\mbf v_1,\, \mbf v_2]$, then after defining new coordinates $\mbf x^{(k)} = (x_{1}^{(k)}, x_{2}^{(k)})^{\msf T} = V^{-1}\mbf y^{(k)}$, we get
	\begin{equation*}
	\begin{bmatrix}
	x_{1}^{(k+1)} \\
	x_{2}^{(k+1)}
	\end{bmatrix} = \begin{bmatrix}
	\lambda_1 & 0 \\
	0 & \lambda_2
	\end{bmatrix} \begin{bmatrix} x_{1}^{(k)} \\ x_{2}^{(k)} \end{bmatrix} =: \Lambda \begin{bmatrix} x_{1}^{(k)} \\ x_{2}^{(k)} \end{bmatrix}
	\end{equation*}
The maps $\Lambda$ and $T$ are topologically conjugate by $\Lambda V ^{-1} = V^{-1} T $.

Note, that $\phi_1(\mbf x^{(k)}) = x_1^{(k)}$ and $\phi_2(\mbf x^{(k)}) = x_2^{(k)}$ are eigenfunctions of $U_\Lambda$ at eigenvalues $\lambda_1$ and $\lambda_2$ respectively.  By proposition \ref{prop:efunc-semigroup} and example \ref{ex:decaying-analytic}, we have that $\phi_{m,n}(\mbf x^{(k)}) := [\phi_1(\mbf x^{(k)})]^m [\phi_2(\mbf x^{(k)})]^n \equiv [x_1^{(k)}]^m [x_2^{(k)}]^n$ is an eigenfunction of $U_\Lambda$ at eigenvalue $\lambda_1^m \lambda_2^n$.  By proposition \ref{prop:top-conj}, $\phi_{m,n} \circ V^{-1}$ is an eigenfunction of $U_T$ at eigenvalue $\lambda_1^m \lambda_2^n$, where $V^{-1}$ has taken the place of $h$ in proposition \ref{prop:top-conj}.
\qed
\end{example}

Now, assume $f \in \mathcal F$ is an observable in the closed, linear span of a set of linearly independent eigenfunctions $\{ \phi_{i} \}_1^n$ (recall \(n\) could be finite or infinite). Then
	\begin{equation}
	f(p) = \sum_{i=1}^{n} c_{i}(f)\phi_{i}(p),
	\end{equation}
for some constants $c_{i}(f) \in \C$.  The dynamics of $f$ are particularly simple:
	\begin{equation}
	\begin{aligned}
	[Uf](p) &= f(T(p)) = \sum_{i=1}^{n} c_{i}(f)\phi_{i}(T(p)) \\
	&= \sum_{i=1}^{n} c_{i}(f)[U\phi_{i}](p) \\
	&= \sum_{i=1}^{n} \lambda_{i} c_{i}(f) \phi_{i}(p),
	\end{aligned}
	\end{equation}
and similarly
	\begin{equation}
	[U^{m}f](p) = \sum_{i=1}^{n} \lambda_{i}^{m} c_{i}(f) \phi_{i}(p).
	\end{equation}
The extension to vector-valued observables $F = (f_{1},\dots, f_{K})^{\msf T}$, where each $f_{i}$ is in the closed linear span of the eigenfunctions, is trivial:
	\begin{equation}\label{eq:KM}
	\begin{aligned}
	[U^{k}F](p) &= \sum_{i=1}^{n} \lambda_{i}^{m} \phi_{i}(p) \begin{bmatrix} c_{i}(f_{1}) \\ \vdots \\ c_{i}(f_{K})\end{bmatrix} \\
	&= \sum_{i=1}^{n} \lambda_{i}^{m} \phi_{i}(p) C_{i}(F),
	\end{aligned}
	\end{equation}
where $C_{i}(F) := [c_{i}(f_{1}),\dots c_{i}(f_{K})]^{\msf T}$.  Motivated by \eqref{eq:KM}, we have the following definition.

\begin{defn} \label{def:koopman-mode}
Let $\phi_i$ be an eigenfunction for the Koopman operator corresponding to the eigenvalue $\lambda_i$.  Given a vector-valued observable $F: M \to V$, the Koopman mode, $C_{i}(F)$, corresponding to $\phi_{i}$ is the vector of the coefficients of the projection of $F$ onto $\linspan\{\phi_i\}$.
\end{defn}

\begin{rem}
The importance of defining Koopman modes with respect to eigenfunctions, rather than eigenvalues, becomes apparent when we consider vector-valued observables and non-simple eigenvalues.  For example, let $\lambda$ have a two-dimensional eigenspace $E_\lambda$ and let $\phi_1$ and $\phi_2$ be a basis for it.  Let $f_1 = c_1 \phi_1 + c_2 \phi_2$ and $f_2 = c_3 \phi_2$ be scalar-valued.  Define $F = (f_1, f_2)^{\msf T}$.  The Koopman modes corresponding to $\phi_1$ and $\phi_2$ are
	\begin{align*}
	C_1(F) &= \begin{bmatrix} c_1 \\ 0 \end{bmatrix} & \text{and} &  & C_2(F) &= \begin{bmatrix} c_2 \\ c_3 \end{bmatrix},
	\end{align*}
respectively.  Note that both of these Koopman modes have $\lambda$ as the associated eigenvalue.  Therefore, if the Koopman mode was defined with respect to the eigenvalue $\lambda$, then it would not be a well-defined object.  However, when the eigenspace is one-dimensional, there is no confusion in saying ``the Koopman mode corresponding to $\lambda$''.
\qed\end{rem}

\begin{rem}
The definition of Koopman modes can be carried over with a slight modification to generalized eigenfunctions.  When an observable can be expanded in terms of only eigenfunctions, the Koopman modes are time-invariant objects.  However, when a generalized eigenfunction is present in the expansion, the Koopman modes become time-dependent objects.  For example, let $\phi$ be an eigenfunction and $\psi$ a generalized eigenfunction of $U$ corresponding to $\lambda \neq 0$:
	\begin{equation*}
	U\phi = \lambda \phi \quad \text{and} \quad U\psi = \phi + \lambda \psi.
	\end{equation*}
Let $F = C_1(F) \phi + C_2(F) \psi$ be a vector-valued observable.  Then
	\begin{align*}
	U^k F &= \left(C_1(F)\lambda^k + C_2(F) k\lambda^{k-1}\right) \phi + \lambda^k C_2(F) \psi \\
	&= \lambda^k \left( C_1(F) + \frac{k}{\lambda} C_2(F)\right) \phi + \lambda^k C_2(F) \psi
	\end{align*}
for $k\geq 0$.  The Koopman mode for $\phi$ at time $k$ is the time-dependent quantity $C_1(F) + k\lambda^{-1} C_2(F)$.

However, since Koopman modes of generalized eigenfunctions have not been treated in the literature, \emph{whenever we refer to Koopman modes in this paper, we implicitly mean a Koopman mode corresponding to an eigenfunction.}
\qed\end{rem}

To be completely explicit, the eigenfunctions are $\C$-valued observables on the state space $M$ and the eigenpairs $(\lambda_{i},\phi_{i})$ depend only upon the dynamics $(M,T)$ and the function space $\mathcal F$, \emph{not} on a particular observable.  The $C_{i}(\cdot)$'s can be thought of as a mapping from the observable space into a vector space $V$; for example, in \eqref{eq:KM} above, $C_{i}$ maps $\mc F$ into $\C^{K}$.  The map $F \mapsto \phi_{i}C_{i}(F)$ is then a vector-valued projection operator onto the subspace $\linspan\{\phi_{i}\}$.

\begin{rem} \label{rem:simplicity}
Given \(\mathcal F\) as a Banach space of scalar functions, one can ask for conditions on the geometric multiplicity of \(\lambda\), i.e., dimension of the eigenspace \(E_{\lambda}\). The general answer to this question depends on the dynamics \(T\) and on the particular space of observables \(\mathcal F\) chosen. We can give an answer for the case most studied in literature, when \(T\) preserves a measure \(\mu\) with \(\mathcal F = L^{2}(M, \mu)\). In this case, all the eigenvalues of the associated Koopman operator \(U\) are on the unit circle. 

When \(T\) is an ergodic transformation (i.e., when any measurable set \(S\) invariant under \(T\) is either of zero or full measure), all eigenvalues of \(U\) are simple (see \citet[, \S 2.4][]{Petersen:1989uv}). When \(T\) is not ergodic, the state space can be partitioned into ergodic sets: minimal invariant sets \(S\) such that the restriction \(T|_S :S \to S\) to any \(S\) is ergodic. Since all ergodic sets \(S\) are disjoint, they support mutually singular functions from \(\mathcal F\). As a result, the number of linearly independent eigenfunctions of \(U\) at any particular eigenvalue \(\lambda\) is bounded from above by the number of ergodic sets in the state space.  The number of such ergodic sets is highly dependent on the character of dynamics. The partition into ergodic sets, the ergodic partition, will be discussed in more detail in Section \ref{sec:level-sets}.

Note that the computational method discussed later in Section \ref{subsubsec:DMD} assumes fixing an initial condition \(p_{0} \in M\), which effectively selects the ergodic set $S\subset M$ which contains the point \(p_{0}\). Then the Koopman operator $U|_S$ acting on $L^{2}(S,\mu|_S)$ has simple eigenvalues, where $\mu|_S$ is the ergodic measure on the ergodic component $S$.

There is an interesting relation between ergodic dynamics and Proposition \ref{prop:efunc-semigroup} when the space of observables \(\mc F = L^{2}(M,\mu)\) is defined with respect to an ergodic measure \(\mu\).  Given an eigenfunction $\phi \in \mc F$ with eigenvalue $\lambda$, Proposition \ref{prop:efunc-semigroup} guarantees that $\phi^n$ is an eigenfunction with eigenvalue $\lambda^n$, as long as $\phi^n \in \mc F$.  If the eigenvalue is periodic ($\lambda^k = \lambda$ for some $k\geq 2$), then $\phi^k$ lies in the eigenspace $E_\lambda$.  Ergodicity guarantees the simplicity of $E_\lambda$, and hence there exists a non-zero \(c \in \C\), such that  $\norm{\phi^k - c\, \phi}_{2} = 0$, in \(L^{2}(M,\mu)\) norm.
%
%The simplicity of eigenspaces might appear to conflict with the wealth of eigenfunctions that can be constructed algebraically, as in Proposition \ref{prop:efunc-semigroup}. Even in the case of an ergodic system, algebraically constructed functions need to be evaluated in the context of the space of observables \(\mathcal F\). For example, if \(\mathcal F = L^{2}\), two eigenfunctions with the same eigenvalue need to be linearly independent in the \(L^{2}\)-sense, rather than pointwise, for the associated eigenspace to be geometrically non-simple.
\qed\end{rem}

\begin{example}[Linear systems\cite{Rowley:2009ez}]\label{ex:linear-sys}
We first look at the case when the dynamics are given by a linear map, $A : M \to M$, on some finite-dimensional, inner-product space $M$; i.e., $\mbf x_{m+1} = A\mbf x_{m}$. Suppose $A$ has a complete set of eigenvectors, denoted by $\{ \mbf v_{1},\dots, \mbf v_{n} \}$, with corresponding eigenvalues $\{\lambda_{1},\dots,\lambda_{n}\}$.  Let $\{\mbf w_{j}\}_{1}^{n}$ be the eigenvectors of the adjoint $A^{*}$ with eigenvalues $\{\overline{\lambda}_{j}\}_{1}^{n}$, normalized so that $\inner{\mbf v_{j}}{\mbf w_{k}} = \delta_{jk}$.  Consider the observable defined as $\phi_{j}(\mbf x) = \inner{\mbf x}{\mbf w_{j}}$.  Then 
	\begin{equation}
	\begin{aligned}
	[U\phi_{j}](\mbf x) &= \phi_{j}(A\mbf x) = \inner{A\mbf x}{\mbf w_{j}} = \inner{\mbf x}{A^{*}\mbf w_{j}}\\
	&= \inner{\mbf x}{\overline{\lambda}_{j} \mbf w_{j}} = \lambda_{j}\inner{\mbf x}{\mbf w_{j}} = \lambda_{j} \phi_{j}(\mbf x).
	\end{aligned}
	\end{equation}
We see that $\phi_{j}$ is an eigenfunction of the Koopman operator.  However, the functions $\{ \phi_j\}_{1}^{n}$ do not exhaust all of the eigenfunctions of $U$.  For example, by proposition \ref{prop:efunc-semigroup}, $\phi_j(\mbf x)\phi_k(\mbf x) = \inner{\mbf x}{\mbf w_{j}}\inner{\mbf x}{\mbf w_{k}}$ is an eigenfunction of $U$.  In particular, $g(\mbf x) := \inner{\mbf x}{\mbf w_{j}}^{k}$ is an eigenfunction of $U$ with eigenvalue $\lambda^{k}$ for any $k\in \N$.

Let $F$ be the vector-valued observable defined as $F(\mbf x) = \mbf x$ when $\mbf x \in \linspan \{\mbf v_{1},\dots, \mbf v_{\ell} \}$, where $\ell < n$, and zero otherwise; i.e., $F$ acts as the identity on a subspace of $M$ spanned by the first $\ell$ eigenvectors and has the complement of that subspace as its kernel.  Then
	\begin{equation}
	F(\mbf x) = \sum_{j=1}^{\ell} \inner{\mbf x}{\mbf w_{j}} \mbf v_{j} = \sum_{j=1}^{\ell} \phi_{j}(\mbf x) \mbf v_{j}
	\end{equation}
and
	\begin{equation}
	[U^{m}F](\mbf x) = \sum_{j=1}^{\ell} \lambda_{j}^{m} \phi_{j}(\mbf x) \mbf v_{j}
	\end{equation}
From these expressions, we see that the eigenvector, $\mbf v_{j}$, of the linear map $A$ is the Koopman mode, $C_{j}(F)$, corresponding to $\phi_{j}$.
 \qed\end{example}

\begin{example}[Rotations of the circle]
Let the state space be the interval $M = [0,1)$.  Let $\omega \in (0,1)$ and define $T:M\to M$ by
	\begin{equation}
	T(p) = p + \omega, \mod 1.
	\end{equation}
Note that there is a natural identification of $M$ with the circle $\mathbb{T} = \R/2\pi\Z$ and the functions on $M$ with $2\pi$-periodic functions on $\R$.  It is well-known that if $\omega \in \Q$, then every initial condition is periodic and if $\omega$ is irrational, then the trajectory starting from any initial condition densely fills $M$.  Note that the dynamics preserve the Lebesgue measure.  Let $\mathcal F = L_{\C}^{1}(M)$, be the space of Lebesgue integrable $\C$-valued functions on $M$ and consider the observable $\phi_{n}(p) = e^{i 2\pi n p }$, $n \in\Z$.  Then
	\begin{equation}
	\begin{aligned}
	[U\phi_{n}](p) &= \phi_{n}(T(p)) = e^{i 2\pi n(p+\omega)} \\
	&= e^{ i 2 \pi n\omega}\, \phi_{n}(p).
	\end{aligned}
	\end{equation}
Therefore, for any $n \in \Z$, $\phi_{n}$ is an eigenfunction of $U$ with eigenvalue $\lambda_{n} = e^{i2\pi n\omega}$.  Since the trigonometric polynomials are dense in $L^{1}(\mathbb{T})$, then for $f_{\ell}(p) = \sum_{n \in \Z} \hat{f}_{\ell}(n) e^{i 2\pi n p}  \in L^{1}(\mathbb{T})$,
	\begin{equation}
	[Uf_{\ell}](p) = \sum_{n\in \Z} \hat{f}_{\ell}(n) e^{i 2\pi n\omega} \phi_{n}(p).
	\end{equation}
If $F = (f_{1},\dots,f_{K})^{\msf{T}}$ is the vector-valued observable, then
	\begin{equation}
	[UF](p) = \sum_{n \in \Z} e^{i 2\pi n\omega} \phi_{n}(p) \begin{bmatrix} 
	\hat{f}_{1}(n) \\ \vdots \\ \hat{f}_{K}(n) 
	\end{bmatrix}.
	\end{equation}
Hence the vectors of Fourier coefficients are the Koopman modes of the system.
 \qed\end{example}

\begin{example}[Partial differential equations]
We continue with the heat equation example from above (example \ref{ex:heat}).  Recall that a map was defined on the space of Fourier coefficients by
	\begin{equation}\label{eq:T-galerkin}
	(T(\mbf a))_{\mbf j} = \exp(-4\pi^{2}c^{2} \norm{\mbf j}_{2}^{2}h) a_{\mbf j}.
	\end{equation}
Note that the canonical coordinate projections, $\phi_{\mbf j}(\mbf a) := a_{\mbf j}$, are eigenfunctions for the Koopman operator, with eigenvalues $\lambda_{\mbf j} = \exp(-4\pi^{2}c^{2} \norm{\mbf j}_{2}^{2}h)$, by the computation,
	\begin{equation}
	\begin{aligned}
	[U\phi_{\mbf j}](\mbf a) &= \phi_{\mbf j}(T(\mbf a)) = (T(\mbf a))_{\mbf j}\\
	&= \exp(-4\pi^{2}c^{2} \norm{\mbf j}_{2}^{2}h) a_{\mbf j} \\
	&= \exp(-4\pi^{2}c^{2} \norm{\mbf j}_{2}^{2}h) \phi_{\mbf j}(\mbf a).
	\end{aligned}
	\end{equation}
For the observable $f_{\mbf x}(\mbf a) := \sum_{\mbf j \in \Z^{2}} a_{\mbf j} e^{i 2\pi \mbf j\cdot \mbf x}$, we get 
	\begin{equation*}
	[U^{m}f_{\mbf x}](\mbf a) = \sum_{\mbf j \in \Z^{2}} \lambda_{\mbf j}^{m} \phi_{\mbf j}(\mbf a) e^{i 2\pi \mbf j\cdot \mbf x}.
	\end{equation*}
Suppose we can only measure the temperature at a finite number of locations $\{\mbf x_{1},\dots, \mbf x_{k}\}$, then for $F = (f_{\mbf x_{1}},\dots, f_{\mbf x_{K}})^{\msf T}$,
	\begin{equation}
	[U^{m}F](\mbf a) = \sum_{\mbf j \in \Z^{2}} \lambda_{\mbf j}^{m} \phi_{\mbf j}(\mbf a) \begin{bmatrix} e^{i 2\pi \mbf j\cdot \mbf x_{1}} \\ \vdots \\ e^{i 2\pi \mbf j\cdot \mbf x_{K}} \end{bmatrix}.
	\end{equation}
In the expression above, each Koopman mode,
	\begin{equation}
	C_{\mbf j}(F) = \begin{bmatrix} e^{i 2\pi \mbf j\cdot \mbf x_{1}} \\ \vdots \\ e^{i 2\pi \mbf j\cdot \mbf x_{K}} \end{bmatrix},
	\end{equation}
is just a ``shape'' function on the physical space $\mbb B = [ -1/2, 1/2] \times [-1/2, 1/2]$.  This stresses the point that the eigenfunctions are defined on the state space $M$ while the Koopman modes are functions in the output space.
 \qed\end{example}

%Although no reference to measure preserving dynamical systems has been made at this point, the Koopman operator was first studied in that context for Hamiltonian systems \cite{Koopman:1931ug} and later invertible transformations that preserved a probability measure $\mu$.  In that setting, the spectrum of the Koopman operator is quite nice.  To see this, suppose $T: M \to M$ is an invertible transformation such that both $T$ and $T^{-1}$ are measurable with respect to the Borel $\sigma$-algebra, $\mc B$.  Let $\mu$ be a measure preserved by $T$ ($\mu(T^{-1}A) = \mu(A)$ for all $A \in \mc B$) and let the space of observables be $\mc F = L^{2}_{\C}(M,\mu)$.  In this case, the Koopman operator is unitary\cite{Koopman:1931ug,Petersen:1989uv}.  Hence, the spectrum of $U$ lies on the unit circle.  %See section \ref{sec:averaging} for methods to compute the corresponding eigenfunctions.  
%
%While the measure-preserving setting leads to nice properties of the Koopman operator, the framework presented in this paper extends beyond that classical setting.  We will stay in the discrete-time setting for expository purposes with analogous developments holding for continuous time settings.

The development thus far has only focused on the case when an observable is in the closed linear span of some set of eigenfunctions of the Koopman operator.  No assumption was made on whether this set was a complete set for $U$ or even if $U$ possessed a complete set of eigenfunctions.  One could ask what conditions we could impose on the system that are sufficient to guarantee that \(U\)  has a spectral decomposition. This is the case for measure-preserving dynamical systems, as we now explain.

Let $\mc A \subset M$ be the attractor of the dynamical system and $\mu$ the unique invariant measure supported on $\mc A$.  Often, $\mu$ will be a so-called physical measure.  These types of measures exhibit the important property
	\begin{equation}\label{eq:B-regular}
	\frac{1}{n} \sum_{k=0}^{n-1} [U^k f](p) \to \int_M f\, d\mu
	\end{equation}
for any continuous observable $f:M \to \C$ and for Lebesgue-almost every $p \in M$ belonging to a positive Lebesgue measure set $V \subset M$ containing the attractor (see \citet{Young:2002vc} for a more detailed discussion).  Such measures are important in applications since they guarantee the existence of well-defined time-averages even when an experiment starts with initial conditions not on the attractor.  In such cases, we can restrict our attention to the dynamics and observables on the attractor and recover all the asymptotic behavior of the system.

The situation on the attractor is quite nice when we consider the function space $\mc F = L^2(\mc A, \mu)$.  The restriction of the dynamics to the the attractor, $T\vert_{\mc A} : A \to A$, can be shown to be invertible $\mu$-almost everywhere.  The restriction of the Koopman operator to the attractor, $U\vert_{\mc A} : L^2(\mc A, \mu) \to L^2(\mc A, \mu)$, can then be defined by $U\vert_{\mc A}f = f \circ T\vert_{\mc A}$.  In this case, the operator is unitary\cite{Koopman:1931ug,Petersen:1989uv}, implying that all of the eigenvalues lie on the unit circle and the eigenfunctions are orthogonal\cite{Petersen:1989uv,Mezic:2004is,Mezic:2005ji}.

Since $U\vert_{\mc A}$ is unitary, there exists a spectral resolution\cite{Petersen:1989uv}
	\begin{equation}\label{eq:spectral-resolution-on-attractor}
	U\vert_{\mc A}f = \int_{S^1} \lambda\, dE(\lambda)f
	\end{equation}
where $E$ is a projection-valued Borel measure on the unit circle; i.e., for any Borel set $S$ in the unit circle, $E(S)$ is a projection operator.  The measure $E$ is supported on the spectrum of $U\vert_{\mc A}$.  $E$ can be decomposed into two measures, $E_p$ and $E_c$, that are supported on the point spectrum and the continuous part of the spectrum, respectively.  For any $f \in L^2(\mc A, \mu)$, the spectral resolution becomes
	\begin{equation}\label{eq:spectral-decomp-mpt}
	\begin{aligned}
	U\vert_{\mc A}^k f &= \sum_{j} e^{i2\pi\omega_j k} P_{j}f \\
	&\quad+ \int_{0}^{1} e^{i 2\pi \theta k}\, dE_{c}(\theta) f
	\end{aligned}
	\end{equation}
where $P_{j}: L^2(\mc A, \mu) \to L^2(\mc A, \mu)$ is the orthogonal projection onto the the eigenspace corresponding to the eigenvalue $\lambda_j = e^{i2\pi \omega_j}$ and $E_c$ is the projection-valued measure corresponding to the continuous part of the spectrum.  Either term on the right-hand side of \eqref{eq:spectral-decomp-mpt} could be zero depending on whether the operator has no point spectrum or no continuous part of the spectrum.   When the eigenvalues are simple, we get for a vector-valued observable $F = (f_{1},\dots, f_{K})^{\msf T} \in \bigoplus_{i=1}^K L^2(\mc A, \mu)$,
	\begin{equation}\label{eq:spectral-decomp-vector-valued-1}
	\begin{aligned}
	[U\vert_{\mc A}^k F](p) &= \sum_{j} e^{i2\pi\omega_j k} \phi_{j}(p) C_{j}(F) \\
	&\quad+ \int_{0}^{1} e^{i 2\pi \theta k}\, dE_{c}(\theta) F(p).
	\end{aligned}
	\end{equation}
where $\phi_{j}$ is the eigenfunction corresponding to the eigenvalue $\lambda_j = e^{i2\pi\omega_j k}$ and we have used $P_{j}F(p) = \phi_{j}(p) C_{j}(F)$ (when $\lambda_i$ is non-simple, this identity does not necessarily hold).  Finally, we note that the constant functions on the attractor are eigenfunctions of $U\vert_{\mc A}$ at eigenvalue 1 and that $P_0 f := \int_{\mc A} f\, d\mu$ defines a projection onto the constant functions.  Then, \eqref{eq:spectral-decomp-vector-valued-1} becomes
	\begin{equation}\label{eq:spectral-decomp-vector-valued-2}
	\begin{aligned}
	[U\vert_{\mc A}^k F](p) &= \int_{\mc A} F(p)\, d\mu(p)\\
	&\quad+ \sum_{ \{j :\, \omega_j \neq 0\} } e^{i2\pi\omega_j k} \phi_{j}(p) C_{j}(F) \\
	&\quad+ \int_{0}^{1} e^{i 2\pi \theta k}\, dE_{c}(\theta) F(p),
	\end{aligned}
	\end{equation}
(for further details on this decomposition consult \citet{Mezic:2005ji}).  

Therefore, for the case of measure-preserving transformations or those systems possessing a physical measure, the asymptotic dynamics of the Koopman operator are given by the contributions of three components: (1) the average value of the observable, (2) the portion admitting a Koopman mode expansion (the part corresponding to the point spectrum), and (3) the contribution of the continuous part of the spectrum.  The third component, $dE_{c}(\theta) F(p)$, we refer to as the Koopman mode distribution, or the KM distribution for short.  Whereas the Koopman mode expansion (the second component in the decomposition) is fairly well-understood with respect to its relation to the physics of a problem, the contribution of the KM distribution does not enjoy the same level of understanding and has received little attention, so far, in the literature.

\subsection{Computation of Koopman Modes}
\label{subsec:KM-comp}
% fxnote, fxwarning, fxerror, and fxfatal can all be used

In the examples that were presented thus far, it was fairly easy to determine the eigenpairs of the Koopman operator and the corresponding Koopman modes.  For a general system, however, things will not be so easy.  This section will discuss a few methods to compute the projection of an observable onto the eigenspaces of the Koopman operator, from both the theoretical and numerical viewpoints.

\subsubsection{Theoretical Results}

%The first tool is given by the following theorem, which is a generalization of the normal Laplace transform.
The first tool is given by the following theorem.   It is a special case of a result found in \citet{Yosida:1995ul}.

\begin{thm}\label{thm:krengel-mean-ergodic-thm}
Let $\mc F$ be a Banach space and $U:\mc F \to \mc F$.  Assume $\norm{U} \leq 1$.  Let $\lambda$ be an eigenvalue of $U$ such that $\abs{\lambda} = 1$.  Let $\hat{U} = \lambda^{-1}U$ and define
	\begin{equation*}
	A_{K}(\hat{U}) = \frac{1}{K} \sum_{k=0}^{K-1} \hat{U}^k.
	\end{equation*}
Then $A_{K}$ converges in the strong operator topology to the projection operator on the subspace of $\hat{U}$-invariant function; i.e, onto the eigenspace $E_{\lambda}$ corresponding to $\lambda$.  That is, for all $f \in \mc F$,
	\begin{equation}\label{eq:banach-avg-projection}
	\lim_{K\to \infty} A_{K} f = \lim_{K\to \infty} \frac{1}{K} \sum_{k=0}^{K-1} \hat{U}^k f = P_{\lambda}f.
	\end{equation}
where $P_\lambda: \mc F \to E_{\lambda}$ is a projection operator.
\end{thm}

\begin{proof}
See \citet{Yosida:1995ul} or \citet{Krengel:1985ew}.
\end{proof}

Consider the case when the eigenvalues are simple and $\abs{\lambda_1} = \cdots = \abs{\lambda_\ell} = 1$ and $\abs{\lambda_n} < 1$ for $n>\ell$.  Then, $\lambda_{j}= e^{i 2\pi \omega_j}$ for some real $\omega_j$, when $j \leq \ell$.  For vector-valued observables, the projections defined by \eqref{eq:banach-avg-projection} take the form
	\begin{equation}\label{eq:GLAd-fourier}
	\phi_{j}C_{j}(F) = \lim_{K\to \infty} \frac{1}{K}  \sum_{k=0}^{K-1} e^{-i2\pi \omega_{j} k} [U^{k}F],
	\end{equation}
for $j = 1, \dots, \ell$.  Hence, Theorem \ref{thm:krengel-mean-ergodic-thm} reduces to Fourier analysis, as one might expect, for those eigenvalues on the unit circle and the projections can be computed with any implementation of a fast Fourier transform.  For further discussion, consult \citet{Mezic:2004is} or \citet{Mezic:2005ji}.

When an observable is a linear combination of a finite collection of eigenfunctions corresponding to simple eigenvalues, we get an extension of the above theorem to eigenvalues not having unit modulus.

\begin{thm}[Generalized Laplace Analysis]\label{thm:glad-thm}
Let $\{ \lambda_1, \dots, \lambda_m \}$ be a (finite) set of simple eigenvalues for $U$ ordered so that $\abs{\lambda_1} \geq \cdots \geq \abs{\lambda_m}$ and let $\phi_i$ be an eigenfunction corresponding to $\lambda_i$.  For each $n \in \{1,\dots, N\}$, assume $f_{n}: M \to \C$ and $f_n \in \linspan\{\phi_1,\dots, \phi_m \}$.  Define the vector-valued observable $F = (f_1,\dots, f_N)^{\msf T}$.

Then the Koopman modes for $F$ can be computed via
	\begin{equation}\label{eq:GLAd}
	\begin{aligned}
	&\phi_j C_j(F) \\
	&\quad= \lim_{K\to \infty} \frac{1}{K} \sum_{k=0}^{K-1} \lambda_j^{-k}\left[ U^k F - \sum_{i=1}^{j-1}\lambda_i^k \phi_i C_i(F) \right].
	\end{aligned}
	\end{equation}
\end{thm}

\begin{proof}
Since  each $\phi_{i}$ is an eigenfunction, \(\linspan\{\phi_{1},\dots,\phi_{m}\}\) is a $U$-invariant subspace, so the restriction of $U$ to this subspace is a finite-dimensional linear operator and can be represented with a matrix.  Any $f_n \in \linspan\{\phi_1,\dots, \phi_m \}$ can be written as $f_n = \sum_{j=1}^{m} c_j(f_n) \phi_j$.  Then $f_n - \sum_{i=1}^{j-1} c_i(f_n) \phi_i  \in \linspan\{ \phi_j,\dots, \phi_m\}$.  $U$ restricted to $\linspan\{ \phi_j,\dots, \phi_m\}$ has eigenvalues $\{ \lambda_j,\dots, \lambda_m \}$.  Then $\lambda_j^{-1} U$ restricted to $\linspan\{ \phi_j,\dots, \phi_m\}$ has eigenvalues $\{ 1, \frac{\lambda_{j+1}}{\lambda_j}, \dots, \frac{\lambda_{m}}{\lambda_j}\}$.  The modulus of any element of this set is $\leq 1$.  Taking the average, as in \eqref{eq:banach-avg-projection}, using the operator $\lambda_j^{-1} U$ restricted to $\linspan\{ \phi_j,\dots, \phi_m\}$ gives the projection onto span of $\lambda_j^{-1} U$-invariant functions.  The $\lambda_j^{-1} U$-invariant functions are just elements of $\linspan\{\phi_j\}$.  Then
	\begin{equation*}
	\lim_{k\to\infty} \frac{1}{K} \sum_{k=0}^{K-1} [\lambda_{j}^{-1} U]^k \left(f_n - \sum_{i=1}^{j-1} c_i(f_n) \phi_i\right)
	\end{equation*}
is just $c_{j}(f_n)\phi_j$.  This is equivalent to $\eqref{eq:GLAd}$ when $F = f_n$.  The extension a vector-valued $F$ is obvious.
\end{proof}

\begin{rem}
Theorem \ref{thm:glad-thm} is a simple consequence of theorem \ref{thm:krengel-mean-ergodic-thm}.  The case of $F$ having elements in a generalized eigenspace is more difficult and is treated in forthcoming work by the authors.
\qed\end{rem}

\begin{rem}
Analogous expressions hold for continuous time, with \eqref{eq:GLAd-fourier} and \eqref{eq:GLAd} replaced by
	\begin{equation}\label{eq:GLAc-fourier}
	\phi_j C_j(F) = \lim_{\mc T\to \infty} \frac{1}{\mc T} \int_{0}^{\mc T} e^{-i 2\pi \omega_j t}[ U^t F] dt.
	\end{equation}
and
	\begin{equation}\label{eq:GLAc}
	\begin{aligned}
	&\phi_j C_j(F) \\
	&= \lim_{\mc T\to \infty} \frac{1}{\mc T} \int_{0}^{\mc T} e^{-\lambda_j t}\left[ U^t F - \sum_{i=1}^{j-1}e^{\lambda_i t} \phi_i C_i(F) \right] dt.
	\end{aligned}
	\end{equation}
respectively.
\qed\end{rem}

The first thing to note is that for theorem \ref{thm:glad-thm} and its continuous-time analogue a set of eigenvalues is needed; they are not computed as part of the theorem.  To get the projections $\phi_{j}\, C_{j}(F)$, the more unstable modes must be subtracted off of the dynamics before the time-average is computed.

The most common case for which Theorems \ref{thm:krengel-mean-ergodic-thm} and \ref{thm:glad-thm} (and their continuous-time analogues) are applied occurs when we restrict our attention to a compact invariant subset $\mc C$ of the basin of attraction for some attractor $\mc A \subset M$ and take for $\mc F$ the product $L^2(\mc A, \mu) \times \mc H(\mc C)$, where $\mu$ is the unique invariant measure supported on the attractor and $\mc H(\mc C)$ is the space of analytic functions on $\mc C$.  In this case, an eigenvalue for $U$ satisfies $\abs{\lambda} \leq 1$.

\begin{example}[Harmonic oscillator]
Consider the Harmonic oscillator
	\begin{align*}\label{eq:harmonic-oscillator}
	\dot{p}_{1}(t) &= p_2(t) \\
	\dot{p}_{2}(t) &= -\omega^2 p_1(t)
	\end{align*}
Letting $\mbf p(t) = (p_1(t), p_2(t))^{\msf T}$, the solution flow is
	\begin{align*}
	\mbf p(t) = \Phi^t(\mbf p(0)) = \begin{bmatrix}
	\cos \omega t & \frac{1}{\omega}\sin \omega t \\
	-\omega\sin \omega t & \cos \omega t
	\end{bmatrix}\begin{bmatrix} p_1(0) \\p_2(0) \end{bmatrix}
	\end{align*}
Note that this system is divergence-free and therefore preserves volume in the state space; all eigenvalues of $U$ have modulus 1.

Let $F(\mbf p(t)) = \mbf p(t)$.  Equation \eqref{eq:GLAc-fourier} is nonzero only for $\omega_{1} = \omega$ and $\omega_{2} = - \omega$.  Then for $\lambda_1 = e^{i\omega}$ and $\lambda_2  = e^{-i\omega}$, we get
	\begin{align*}
	\phi_{1}(\mbf p(0))C_{1}(F) &= \frac{1}{2} \left[\begin{pmatrix}
	p_1(0) \\ p_2(0) \end{pmatrix} - i \begin{pmatrix}  {p}_2(0) / \omega \\ -\omega p_1(0) \end{pmatrix} \right] \\
	&= \frac{1}{2}\left(p_1(0) - i \frac{p_2(0)}{\omega}\right)\begin{bmatrix} 1 \\ i\omega \end{bmatrix}
	\end{align*}
and
	\begin{align*}
	\phi_{2}(\mbf p(0))C_{2}(F) &= \frac{1}{2} \left[\begin{pmatrix}
	p_1(0) \\ p_2(0) \end{pmatrix} + i \begin{pmatrix}  {p}_2(0) / \omega \\ -\omega p_1(0) \end{pmatrix} \right] \\
	&= \frac{1}{2}\left(p_1(0) + i \frac{p_2(0)}{\omega}\right)\begin{bmatrix} 1 \\ -i\omega \end{bmatrix},
	\end{align*}
respectively.
We can write
	\begin{align*}
	[U^t F](\mbf p(0)) &= F(\mbf p(t)) = \Phi^t(\mbf p(0)) \\
	&= e^{i \omega t} \phi_1(\mbf p(0))C_{1}(F) \\
	&\quad + e^{-i \omega t}\phi_{2}(\mbf p(0))C_2(F),
	\end{align*}
or more explicitly,
	\begin{align*}
	&\begin{bmatrix} p_1(t) \\ p_2(t) \end{bmatrix} = e^{i \omega t} \frac{1}{2}\left(p_1(0) - i \frac{p_2(0)}{\omega}\right)\begin{bmatrix} 1 \\ i\omega \end{bmatrix}\\
	&\qquad+ e^{-i \omega t} \frac{1}{2}\left(p_1(0) + i \frac{p_2(0)}{\omega}\right)\begin{bmatrix} 1 \\ -i\omega \end{bmatrix}.
	\end{align*}
We recognize the familiar normal mode expansion for the harmonic oscillator.  The normal modes are the Koopman modes $C_1(F) = [1, \, i\omega]^{\msf T}$ and $C_2(F) = [1,\, -i\omega]^{\msf T}$, whereas the Koopman eigenfunctions $\phi_{1,2}(\mbf p(0)) = \frac{1}{2}(p_1(0) \mp p_2(0)/\omega)$ are the terms of the normal mode expansion that are functions of only the initial conditions.
\qed\end{example}

%As a final remark, the notions of convergence in the above theorems are with respect to a norm on the Banach space of observables, whereas, often, the appropriate notion of convergence is pointwise almost everywhere.  For example, suppose the space of observables is $L^2(M,dx)$.  The above averages guarantee convergence in the mean, with no statement made about any specific initial condition in $M$.  However, if we are conducting a physical experiment, we choose a specific initial condition $p \in M$.  We would like to know if the \emph{numerical values} $\frac{1}{K} \sum_{k=0}^{K-1} [\hat{U}_{j}^{k} F](p)$ converge.  While not true in general, the answer is in the affirmative for the special case of systems possessing a physical measure (see \eqref{eq:B-regular} above and the corresponding discussion).

While in principle the projections onto the stable and unstable modes can be computed directly using Theorem \ref{thm:glad-thm}, difficulties arise when we move past simple cases.  If an explicit representation of $U$ is not known for the observable $F$, we would need to compute the projections numerically.  For $\abs{\lambda} \neq 1$, this would require a numerical implementation of a Laplace transform; these are generally unstable computations, precluding the direct numerical implementation of Theorem \ref{thm:glad-thm} when $F$ has stable or unstable Koopman modes.  Therefore, Theorem \ref{thm:glad-thm} is more suited as an analytical tool, rather than a numerical one.  

\subsubsection{A numerical algorithm: the Dynamic Mode Decomposition (DMD)}
\label{subsubsec:DMD}
Usually, we do not have access to an explicit representation of the Koopman operator.  The behavior of the operator can only be ascertained by its action on an observable and usually at only a finite number of initial conditions.  Thus, we are led to consider data-driven algorithms for computing the Koopman modes.  By data, we mean a sequence of observations of a vector-valued observable along a trajectory $\{ T^k p \}$.  The following algorithms use these sequences of observations to approximate both the eigenvalues and the Koopman modes of $U$ without having to numerically implement a Laplace transform.  This is accomplished by finding the best approximation of $U$ on some finite-dimensional subspace and computing eigenfunctions of the resulting finite-dimensional linear operator.  The notion of \emph{best approximation} will be made clear in what follows.

Fix a vector-valued observable $F : M \to \C^{m}$ and consider the cyclic subspace
	\begin{equation}
	\mc K_{\infty} = \linspan\{ U^{k}F \}_{k=0}^{\infty};
	\end{equation}
that is, $\mc K_{\infty}$ is the space of vector-valued observables in which finite linear combinations of elements from $\{ U^{k}F \}_{k=0}^{\infty}$ are dense.

Fix an $r < \infty$ and consider the Krylov subspace 
	\begin{equation}
	\mc K_r = \linspan \{ U^{k}F \}_{k=0}^{r-1}.
	\end{equation}
We will assume that $\{ U^{k}F \}_{k=0}^{r-1}$ is a linearly independent set so that these functions form a basis for $\mc K_r$.  Note that $U\mc K_r \subset \mc K_{r+1}$, so that, in general, $\mc K_r$ is not $U$-invariant; it is only invariant if $U^{r} F \in \linspan\{ U^{k}F \}_{k=0}^{r-1}$.  

Let $P_r : \mc F^{m} \to \mc K_{r}$ be a projection from the space of vector-valued observables onto $\mc K_r$.  Then 
	\begin{equation}\label{eq:Krylov-projection}
	P_r U\vert_{\mc K_{r}} : \mc K_{r} \to \mc K_{r}
	\end{equation}
is a finite-dimensional linear operator.  This operator has a matrix representation, $\msf A_{r}:\C^r \to \C^r$, in the $\{ U^{k}F \}_{k=0}^{r-1}$-basis.  Note that this matrix is dependent upon (1) the vector-valued observable, (2) the number of time-steps $r$ used (the dimension of the Krylov subspace), and (3) the projection $P_r$ used which is specified by the type of approximation we choose to make; in the following algorithms, this projection is the least-square approximation for evolution from a single point $p \in M$.
% as
%	%%%       Equation       %%%
%	\begin{equation}
%	\msf A_r = \begin{bmatrix}
%	0 & 0 & \cdots & 0 & c_0 \\
%	1 & 0 & \cdots & 0 & c_1 \\
%	0 & 1 & \cdots & 0 & c_2 \\
%	\vdots & \vdots & \ddots & \vdots & \vdots \\
%	0 & 0 & \cdots & 1 & c_{r-1}
%	\end{bmatrix}.
%	\end{equation}
%	%%%%%%%%%%%%%%

If $(\lambda, \mbf v)$ is an eigenpair for $\msf A_r$, where $\mbf v = (v_0,\dots, v_{r-1})^{\msf T} \in \C^{r}$, then $\phi = \sum_{j=0}^{r-1} v_j [U^{j}F]$ is an eigenfunction of $P_r U\vert_{\mc K_{r}}$.  By restricting our attention to a fixed observable $F$ and a Krylov subspace, we have reduced the problem of finding eigenvalues and Koopman modes of the Koopman operator to finding eigenvalues and eigenvectors for a matrix $\msf A_r$.

A standard method for computing eigenvalues of a matrix is the Arnoldi algorithm and its variants.  These are iterative methods relying on Krylov subspaces.  The basic idea behind these algorithms is to project the matrix onto a lower-dimensional subspace and to compute the eigenvalues of the lower-rank matrix.  If the projection has a nice representation, then the eigenvalue problem for this lower-rank approximation can be efficiently solved.

The standard Arnoldi algorithm\cite{Arnoldi:1951ub,Meyer:2000tv,Saad:1986hx} assumes we have a matrix $\msf A : \C^{m} \to \C^{m}$ whose eigenvalues and eigenvectors we want to compute.  Starting from a random vector $\mbf b \in \C^m$ of unit norm, we form the Krylov subspace
	\begin{equation*}
	\mc K_r := \linspan \{ \mbf b, \msf A\mbf b, \dots, \msf A^{r-1}\mbf b\}.
	\end{equation*}
Assuming full rank, an orthonormal basis $\{ \mbf q_j \}_{1}^{r}$ for $\mc K_r$ can be found using a Gram-Schmidt procedure applied to $\{ \msf A^j \mbf b\}_{j=0}^{r-1}$.  Orthogonalization and renormalization is usually performed at each step $j$.  Letting $\msf Q_r$ be the matrix formed from the orthonormal basis, we get the relation
	\begin{equation*}
	\msf H_r = \msf Q_r^* \msf A \msf Q_r
	\end{equation*}
where $\msf H_r$ is of upper Hessenberg form and has the interpretation of the orthogonal projection of $\msf A$ onto $\mc K_r$.  $\msf H_r$ can be diagonalized efficiently.  The eigenvalues of $\msf H_r$ approximate the $r$ eigenvalues of $\msf A$ of largest magnitude.  Implementations of this algorithm use various additional methods to ensure numerical stability.

By applying the Arnoldi algorithm, we implicitly assume there exists a matrix $\msf A$ whose evolution $\msf A^k \mbf b \in \C^m$ matches the evolution $[U^k F](p) \in \C^m$ for $k=0,\dots, r$.  Unfortunately, since we do not have an explicit representation of the Koopman operator, we cannot use the standard Arnoldi algorithm.  This is due to the orthogonalization and renormalization performed at each step, which is equivalent to changing the observable $F$ at each step.  If $\mbf q_k$ is the vector formed from normalizing the component of $[U^k F]$ orthogonal to $\linspan\{ [U^j F] \}_{0}^{k-1}$, then there is some $G : M \to \C^m$ such that $\mbf q_k = G(p)$.  Hence we are never looking at the action of the Koopman operator along a single trajectory and observable, precluding the use of the Arnoldi algorithm with data obtained from simulation.

A variant of the Arnoldi algorithm, utilizing companion matrices, was first described in \citet{Ruhe:1984tr}.  The algorithm was popularized in the Fluids community by \citet{Rowley:2009ez} and \citet{Schmid:2010ba}.  In  \citet{Schmid:2010ba}, the algorithm was dubbed the Dynamic Mode Decomposition (DMD), whereas \citet{Rowley:2009ez} related the algorithm to the approximation of Koopman modes.  

The strength of the DMD algorithm is that it only requires a sequence of vectors $\{ \mbf b_k \}_{k=0}^{r}$, where 
	\begin{equation}
	\mbf b_k := U^k F(p) \in \C^m
	\end{equation}
for some fixed $F : M \to \C^m$ and fixed $p \in M$.  The algorithm gives the best approximation, at the point $p \in M$, of the projection $\phi C(F)$ onto the eigenfunction $\phi$ (eq. \eqref{eq:KM}).  As will be seen, this corresponds to a specific choice of the projection operator $P_r$ appearing in \eqref{eq:Krylov-projection}.

To derive the DMD algorithm, let
	\begin{equation*}
	\msf K_r := [ \mbf b_0, \dots, \mbf b_{r-1}]. %= [ \mbf b_0, \dots, \msf A^{r-1}\mbf b_0 ].
	\end{equation*}
The columns of $\msf K_r$ are the $\C^m$-vectors resulting from the point evaluations of the $\{U^{k}F\}$-basis for the Krylov subspace $\mc K_r$ at the point $p\in M$.

%where $\msf A$ is defined by \eqref{eq:koopman-matrix}.  
In general, $\mbf b_r$ will not be in the span of the columns of $\msf K_r$.  In this case, $\mbf b_r = \sum_{j=0}^{r-1} c_j \mbf b_j + \mbf \eta_r$, where the $c_j$'s are chosen to minimize the $\C^m$-norm of the residual $\mbf \eta_r$.  This corresponds to choosing the projection operator $P_r$ appearing in \eqref{eq:Krylov-projection} so that $P_r U^r F$ is the least-squares approximation to $U^r F$ \emph{at the point $p\in M$} as measured by the $\C^m$-norm; i.e.,
	\begin{align*}
	\norm{ [U^{r}F](p) - P_r [U^{r}F](p) }_{\C^m} &= \norm{ \mbf b_r - \sum_{j=0}^{r-1} c_j \mbf b_{j} }_{\C^m} \\
	&\leq \norm{ \mbf b_r - \sum_{j=0}^{r-1} d_j \mbf b_{j} }_{\C^m}
	\end{align*}
for any other $\{d_0,\dots, d_{r-1}\}$.

Since $\mbf b_r = \msf K_r \mbf c + \mbf \eta_r$, where $\mbf c = (c_0,\dots, c_{r-1})^{\msf T}$, we get
\begin{DIFnomarkup}
  \begin{equation*}
	U \msf K_r = [ \mbf b_1, \dots, \mbf b_{r}] = [ \mbf b_1, \dots, \mbf b_{r-1}, \msf K_r \mbf c + \mbf \eta_r],
  \end{equation*}
\end{DIFnomarkup}
or equivalently
	\begin{equation}\label{eq:krylov-shift}
	U \msf K_r = \msf K_r \msf A_r + \mbf \eta_r \mbf e^{\msf T}
	\end{equation}
where $\mbf e = (0,\dots, 0, 1)^{\msf T} \in \C^m$ and
	\begin{equation}
	\msf A_r = \begin{bmatrix}
	0 & 0 & \cdots & 0 & c_0 \\
	1 & 0 & \cdots & 0 & c_1 \\
	0 & 1 & \cdots & 0 & c_2 \\
	\vdots & \vdots & \ddots & \vdots & \vdots \\
	0 & 0 & \cdots & 1 & c_{r-1}
	\end{bmatrix}
	\end{equation}
is the $r\times r$ companion matrix; it is the matrix representation of $P_r U$ in the $\{U^k F\}_{k=0}^{r-1}$-basis.

Diagonalize the companion matrix:
	\begin{equation}\label{eq:diagonalized-companion}
	\msf A_r = \msf V^{-1} \msf \Lambda \msf V,
	\end{equation}
where $\msf \Lambda$ is the diagonal matrix of eigenvalues $\lambda_i$ and the columns of $\msf V^{-1}$ are eigenvectors of $\msf A_r$.  Inserting the expression for $\msf A_r$ into \eqref{eq:krylov-shift} and multiplying on the right by $\msf V^{-1}$ gives
	\begin{equation}\label{eq:matrix-krylov}
	U \msf K_r \msf V^{-1}= \msf K_r  \msf V^{-1} \msf \Lambda + \mbf \eta_r \mbf e^{\msf T}\msf V^{-1}.
	\end{equation}
Define
	\begin{equation}\label{eq:empirical-ritz-vectors}
	\msf E := \msf K_r \msf V^{-1}.
	\end{equation}
Then \eqref{eq:matrix-krylov} becomes
	\begin{equation}
	U \msf E = \msf E \msf \Lambda + \mbf \eta_r \mbf e^{\msf T} \msf V^{-1}.
	\end{equation}
For large enough $m$, it is hoped that $\norm{ \mbf \eta_r \mbf e^{\msf T} \msf V^{-1} }$ is small.  If that is the case, then $U \msf E \approx \msf E \msf \Lambda$ and the columns of $\msf E$ approximate some eigenvectors of $U$ and the diagonal elements of $\msf \Lambda$ approximate some eigenvalues of $U$.  Note that $\norm{ \mbf \eta_r \mbf e^{\msf T} \msf V^{-1}} = \mbf 0$ whenever $r > m$, since then the columns of $\msf K_r$ are linearly dependent which implies that $\mbf \eta_r =\mbf 0$.

\begin{defn}\label{def:empirical-ritz}
Let $\msf \Lambda$ and $\msf E$ be defined as in \eqref{eq:diagonalized-companion} and \eqref{eq:empirical-ritz-vectors}, respectively.  Let $\mbf w_i$ be the $i^{th}$ column of $\msf E$ and $\lambda_i$ be the $i^{th}$ diagonal element of $\msf \Lambda$.  Then $\mbf w_i$ is called an empirical Ritz vector and $\lambda_i$ is called an empirical Ritz value.
\end{defn}

Each empirical Ritz vector approximates $\phi_i(p) C_{i}(F)$, the projection of $F$ onto some eigenvector $\phi_i$, and the empirical Ritz values approximate the corresponding eigenvalues of $U$.  For this reason, we will generally refer the empirical Ritz vectors and values as the Koopman modes and eigenvalues computed by the DMD algorithm although this is not strictly true and loosens the terminology.

\begin{rem} 
The above algorithm is very much tied to the initial condition chosen.  This dependence arises since the empirical Ritz values and vectors are formed using a Krylov subspace that is generated by a sequence of vector-valued observations along a finite trajectory having initial condition $p \in M$.  A given initial condition may not reveal the full spectrum and different initial conditions can reveal different parts of the spectrum.  For example consider the dynamical system,
	\begin{equation}
	p_{k+1} = \begin{cases}
	\lambda_{1} p_{k}, & p_{k} \leq 0 \\
	\lambda_{2} p_{k}, & p_{k} > 0
	\end{cases},
	\end{equation}
where $0 < \lambda_1 < 1 < \lambda_2$.  Let $\phi_{1}(p) = \min\{p,0\}$ and $\phi_{2}(p) = \max\{0,p\}$.  These are eigenfunctions of the Koopman operator at eigenvalues $\lambda_1$ and $\lambda_2$, respectively.  Let $F(p) = p \equiv \phi_{1}(p) + \phi_{2}(p)$.  Analytically, we can decompose the observable into a sum of projections onto eigenspaces: $F(p) = P_1 F(p) + P_2 F(p)$, where $P_{i}$ is the projection onto $\linspan \phi_i$.    Choosing an initial condition $p < 0$ and applying the DMD algorithm only computes the projection onto the stable mode; the DMD algorithm only reveals $P_1 F(p)$.  Similarly, choosing $p > 0$ only reveals the unstable mode, $P_2 F(p)$.  Therefore, the DMD algorithm may only reveal a subset of the spectrum of the Koopman operator and the corresponding Koopman modes.

It should also be remarked that if $F \notin \linspan\{ \phi_{i} \}$ for some eigenfunction $\phi_i$, then the DMD algorithm will not reveal that mode.  This is often the case for natural choices for a set of observables, as was the case in the linear system example above (see example \ref{ex:linear-sys}, p.\ \pageref{ex:linear-sys}).
\qed\end{rem}

The version of the DMD algorithm described tends to be numerically ill-conditioned.  The is due to $\msf A^k \mbf b_0$ converging to the eigenspaces corresponding to the largest magnitude eigenvalues, resulting in the columns of $\msf K_r$ becoming nearly linearly dependent.  A robust version of the algorithm has been described in \citet{Schmid:2010ba}.  It amounts to first computing a singular value decomposition (SVD) of $\msf K_r$ and projecting $\msf A$ onto the Krylov subspace using the SVD basis (for details, consult \citet{Schmid:2010ba}).  \citet{Chen:2012jh} discusses variants of the Dynamic Mode Decomposition, relates it to discrete Fourier transforms, and introduces an ``optimized'' DMD algorithm that computes an arbitrary number of modes from data.

%The DMD algorithm algorithm is connected to a particular representation of $U$.  For a fixed observable $F$, the action of the Koopman operator can be represented as a sequence; $s(p) = ( F(p), UF(p), \dots, U^{k}F(p), \dots)$ for each $p \in M$.  This is similar to the representation of a stochastic process; each $p \in $M gives a different ``realization'' of $U$.  Then we have the sequence space $S = \{ s(p) : p \in M \}$.  When $F$ is bounded, $S$ is some subset of $\ell_{\C}^{\infty}(\N)$.  Define the shift operator $\mc L : S \to S$ as 
%	%%%       Equation       %%%
%	\begin{equation}
%	\mc L( F(p), UF(p), \dots ) = (UF(p), U^2 F(p), \dots )
%	\end{equation}
%	%%%%%%%%%%%%%%
%Hence the Koopman operator has the \emph{representation} $[UF](p) = \pi_0 \mc L [s(p)]$, where $\pi_0$ extracts the first coordinate of a sequence in $S$.  The shift operator $\mc L$ has a formal representation as an infinite-dimensional ``shift matrix'':
%	%%%       Equation       %%%
%	\begin{equation}
%	\begin{aligned}
%	&\mc L( F(p), UF(p), \dots ) =\\
%	&[F(p), UF(p), U^2 F(p), \dots ] \begin{bmatrix}
%	0 & 0 & 0 & \cdots & 0 & \cdots \\
%	1 & 0 & 0 & \cdots & 0 & \cdots \\
%	0 & 1 & 0 & \cdots & 0  & \cdots \\
%	\vdots & \vdots & \vdots & \ddots & \vdots & \cdots
%	\end{bmatrix}.
%	\end{aligned}
%	\end{equation}
%	%%%%%%%%%%%%%%
%Formally, the above algorithm considers a finite-dimensional truncation of the infinite-dimensional ``shift matrix''.

%%% Local Variables: 
%%% mode: latex
%%% TeX-master: "../koopmanism"
%%% End: 

\subsection{Applications of Koopman Modes}
\label{subsec:KM-apps}
% fxnote, fxwarning, fxerror, and fxfatal can all be used

The theory of Koopman modes has led to a number applications in the literature. Broadly, the uses of Koopman modes can be classified under two headings: model reduction and coherency.  Model reduction deals with extracting the spatial features of just a few Koopman modes and attempting to understand the physics of the system just based on those, neglecting the details contained in other Koopman modes.
The notion of coherency, on the other hand, deals with how observables relate dynamically with respect to a Koopman mode.  Coherency is always defined for a (not necessarily proper) subset of observables and an eigenvalue $\lambda$.  The subset of observables is coherent for $\lambda$ if the dynamics are identical.  This reduces to checking if the initial magnitudes and phases of $C_{\lambda}(F)$ are the same for each observable in the subset.  The following definition, first appearing in \citet{Susuki:2010ei}, makes this precise.

\begin{defn}[Coherency between Koopman Modes]\label{def:km-coherency}
Consider a vector-valued observable $F : M \to \C^m$, where $F = (f_1, \dots, f_m)^{\msf T}$ and $f_j : M \to \C$ for $j=1,\dots, m$.  Let $\{ C_1(F), \dots, C_{\ell}(F) \}$, $\ell < \infty$, be a collection of Koopman modes of interest.  Note that $C_i(F) = (c_{i,1}, \dots, c_{i,m})^{\msf T} \in \C^m$.  Fix $\eps_1, \eps_2 > 0$ and consider $j_1, j_2 \in \{1,\dots, m\}$.  Then, $f_{j_1}$ and $f_{j_2}$ are \textbf{$\mbf{(\eps_1,\eps_2)}$-coherent} (with respect to the chosen Koopman modes) if
\begin{enumerate}[(i)]
\item $ \abs{ \abs{c_{i,j_{1}} } - \abs{c_{i,j_{2}}} } < \eps_1$, and
\item $\abs{ \angle{c_{i,j_{1}} } - \angle{c_{i,j_{2}}} } < \eps_2$.
\end{enumerate}
for all $i=1, \dots, \ell$.
\qed\end{defn}

The choice of two epsilon values in the above definition allows the practitioner to set the tolerances of the magnitude and phase independently.  This is useful when one can tolerate more variation in either the modulus or the phase and still call the modes coherent.  Thus $c_{i,j_{2}}$ is coherent with $c_{i,j_{1}}$, if the complex number $c_{i,j_{2}}$ is contained inside some rectangle centered at $c_{i,j_{2}}$.

The examples we present are necessarily a subset of those that exist in the literature.

\subsubsection{Power systems\cite{Susuki:2011ef,Susuki:2010ei}}
Koopman mode analysis has seen application in the analysis of power systems.  In \citet{Susuki:2010ei,Susuki:2011ef}, the authors used Koopman mode analysis to identify coherency in the short-term swing dynamics of multi-machine power systems, with the New England Test System and IEEE Reliability Test System-1996 being used as test cases for the methodology\cite{Susuki:2010ei,Susuki:2011ef}.  In \citet{Susuki:2012gk}, the authors used Koopman mode analysis to identify precursors to the so-called coherent swing instability of power systems where a group of generators synchronously loses coherency with the rest of the system after a local disturbance.  We will focus only upon the identification of coherency, as pursued in \citet{Susuki:2010ei,Susuki:2011ef}, since it underlies the work in \citet{Susuki:2012gk} as well.

The New England system is a 39-bus system having 10 synchronous generators, while the IEEE systems has 73 buses and 99 synchronous generators\cite{Susuki:2010ei,Susuki:2011ef}.  As the analysis of the two systems is the same, we focus on the simpler New England system.

The swing dynamics of the New England system were given by the following systems of differential equations\cite{Susuki:2010ei,Susuki:2011ef}:
	\begin{equation}
	\begin{aligned}
	\frac{d\delta_i}{dt} &= \omega_i \\
	\frac{H_i}{\pi f_b}\frac{d\omega_i}{dt} &= -D_i \omega_i + P_{m i} - G_{ii} E_i^2 \\
	&\quad - \sum_{j=1 \atop j \neq i}^{10} E_i E_j C_{ij},
	\end{aligned}
	\end{equation}
where $C_{ij}$ is the coupling term
	\begin{equation*}
	C_{ij} = G_{ij} \cos(\delta_i - \delta_j) + B_{ij}\sin(\delta_i - \delta_j) .
	\end{equation*}
In this model, $i = 2, \dots 10$ indexed the generators, $\delta_i$ was the angular position of the rotor of generator $i$ relative to bus 1, and $\omega_i$ was the rotor speed of generator $i$ relative to that of bus 1; $D_i$ was the damping coefficient of generator $i$, $E_i$ was the voltage of the generator, and $P_{mi}$ the mechanical input power; $G_{ii}$ was the internal conductance of generator $i$, while $G_{ij} + \sqrt{-1} B_{ij}$ was the transfer impedance between generators $i$ and $j$; $H_i$ was a per unit time inertia constant and $f_b$ a frequency\cite{Susuki:2010ei,Susuki:2011ef}.  The variables $H_i$, $E_i$, $D_i$, $f_b$ and power loads were specified during simulations\cite{Susuki:2010ei,Susuki:2011ef}.  The variables $G_{ii}$, $G_{ij}$, and $B_{ij}$ were computed using power flow computations (see \citet{Susuki:2010ei,Susuki:2011ef}, and the references therein, for full numerical details).  The state space for each generator ($i=2,\dots,10$) was the cylinder $\mc C = [-\pi, \pi] \times \R$ and the full state space $M$ for this system was $M = \mc C \times \cdots \times \mc C = \mc C^9$.  Each generator exhibited a stable equilibrium at $(\delta_i^*, \omega_i^* = 0)$, for some $\delta_i^*$, computed using a power flow computation.

Let an observable $f_i : M \to \R$ be given by $f_i( \mbf \delta, \mbf \omega) = \omega_i$, where $\mbf \delta = (\delta_2, \dots, \delta_{10})$ and $\mbf \omega = (\omega_2,\dots, \omega_{10})$.  The vector-valued observable chosen for Koopman mode analysis was $F = (f_2, \dots, f_{10})^{\msf T}$, so that
	\begin{equation*}
	F(\mbf \delta, \mbf \omega) = \begin{bmatrix} \omega_2 \\ \vdots \\ \omega_{10} \end{bmatrix}.
	\end{equation*}
This was a physically-relevant observable since in practice one measures the rotor speeds for each generation plant\cite{Susuki:2010ei,Susuki:2011ef}.

The system was evolved for a short time period with a disturbance from the equilibrium state localized at rotor 8.  Study of the resulting trajectories showed that generators 2, 3, 6, and 7 were a coherent group\cite{Susuki:2010ei,Susuki:2011ef}; the four generators exhibited responses in angular frequencies $\omega_i$ having the same amplitude and phase.

The DMD algorithm was applied to the same simulation data.  The Koopman modes of interest were those that had the largest norms and corresponding growth rates $\abs{\lambda_j}$.  Figure \ref{fig:susuki2011_correction_fig7} shows a plot of the magnitude and phase of each component of the three Koopman modes with the largest growth rates and magnitudes.  These were labeled as modes 7, 8, and 9 and corresponded to frequencies of 1.3078, 1.0962, and 0.3727 Hz, respectively.  For each mode $j=7, 8, 9$, there were amplitudes $A_{ji}$ and phases $\alpha_{ji}$ for the generators $i=2,\dots, 10$.  In the figure, the amplitudes and phases for modes 7, 8, 9 are plotted in plotted with symbols $\ast$, $\times$, and $\circ$, respectively.  Number labels within the plots specify the generator.  It is seen that generators 2, 3, 6, 7, and 9 are coherent with respect to mode 8, while all but generator 9 are coherent with respect to mode 9.  Therefore, generators 2, 3, 6, 7, and 9 were coherent with respect to both modes 8 and 9, as one found with visual inspection of the trajectories.  While not done in this paper, the amplitudes and phases represented in this way allow using a number of clustering algorithms to automatically identify coherency.

\begin{figure}[htb]
\begin{center}
\framebox[\linewidth][c]{ \raisebox{50mm}{\footnotesize Waiting for permissions. See original publication.}}
\caption[DMD Koopman modes for Power Systems \citet{Susuki:2011jq}]{Power Systems - amplitudes and phases for the components of the three Koopman modes having largest growth rates and norms.  Koopman modes and eigenvalues were computed using the DMD algorithm.  The largest modes are labeled as modes 7, 8, and 9.  The amplitudes and phases of the components of mode 9 are plotted with the symbol $\circ$.  The numbers in the plot correspond to component of the mode.  Modes 7 and 8 are plotted similarly, expect with symbols $\ast$ and $\times$, respectively.   Generators 2, 3, 6, 7, and 9 are coherent with respect to mode 8.  All generators, except generator 8, are coherent with respect to mode 9.}
\label{fig:susuki2011_correction_fig7}
\end{center}
\end{figure}

\subsubsection{Jet in Crossflow\cite{Rowley:2009ez}}\label{subsubsec:jet-in-crossflow}
One of the earliest applications of Koopman modes was to the study of nonlinear fluid flows.  \citet{Rowley:2009ez} introduced the concepts to the fluids community and demonstrated the methodology by computing a subset of the Koopman modes, using the DMD algorithm, for a jet in a crossflow.  The jet in a crossflow configuration is a common way of mixing the jet fluid with a uniform crossflow.  The crossflow moves parallel to a flat plate and the jet is injected through an orifice in the plate.  It is known that such a system can exhibit self-sustained oscillations\cite{Rowley:2009ez}, and Koopman modes were used to automatically identify the relevant frequencies and corresponding three-dimensional flow structures.

The flow field was studied in a $(L_x, L_y, L_z) = (75, 20, 30)\delta_0^*$ computational box, where $\delta_0^*$ was the displacement thickness at the crossflow inlet.  The incompressible Navier-Stokes equations over a flat plate were solved using a Fourier-Chebyshev spectral method with a grid resolution of $256\times 201 \times 144$ (see the reference for all the simulation details).  Therefore, each point in the state space $M$ is a sequence of Fourier-Chebyshev coefficients.  The vector-valued observable, $F : M \to \R^m$, was chosen as the velocity measurements of the flow field at the grid points.  Hence, $m = 3(256\times 201 \times 144) \approx 2.2 \times 10^7$ (three velocity components at each grid point).  The situation is similar to the heat equation example above (example \ref{ex:heat}, p.\ \pageref{ex:heat}) where the state space was a space of Fourier coefficients and the observable was a heat distribution on a square.  

\begin{rem}\label{rem:basis-space-choice}
The choice of state space $M$ and a basis for the observables are not unique.  For a fluid experiment on some physical domain $\mbb B$, we assume that the solutions exists in some function space.  Usually this is the space of finite energy flows so that a velocity profile $\mbf u(\mbf x)$ exists in $L^2(\mbb B, d\mbf x)$.  Depending on the boundary conditions, many different bases for $L^2(\mbb B, d\mbf x)$ exist.  The particular basis chosen depends on the computational method used to solve the Navier-Stokes equations.  Above, this method was a Fourier-Chebyshev spectral method; the resulting basis for $L^2(\mbb B, d\mbf x)$ consisted of trigonometric and Chebyshev polynomials.  The state space $M$ was the sequence space of Fourier-Chebyshev coefficients corresponding to $L^2(\mbb B, d\mbf x)$ functions.  However, if periodic boundary conditions for the fluid flow are assumed, the trigonometric polynomials could be used to represent $L^2(\mbb B, d\mbf x)$ and $M$ would be the space of Fourier coefficients.  Therefore, the choice of solution method for the Navier-Stokes equations determines the basis functions and implicitly defines the state space $M$ that represents the system.
\qed\end{rem}

An initial velocity profile and boundary conditions were specified for simulations and this configuration corresponded to a fixed $p \in M$.  Leting $\mbf b_k := [U^k F](p)$, the DMD algorithm was performed on the sequence of observables $\{ \mbf b_{200}, \mbf b_{202}, \dots, \mbf b_{700} \}$.  The reason for the delay in the start time of the sequence was to neglect the transient terms.

The top row of figure \ref{fig:rowley2009_fig1} shows the time signals of the streamwise velocity recorded by a sensor probe close to the wall, just downstream from the jet orifice and a probe located downstream and on the jet shear layer.  In the bottom row, the spectral content of the time signals are shown in black, whereas the spectral peaks identified by the DMD algorithm are shown in red.  Note that for the probe near the wall (left column of figure \ref{fig:rowley2009_fig1}) the signal contains only low-frequency components, whereas the signal near the jet contains both low- and high-frequencies.  The probes are local measurements and only pick up a subset of the full spectrum for the fluid flow.  The Koopman modes are global objects and, as shown in the figure, the DMD algorithm identifies both the low- and high- frequency components of the flow field.

Figure \ref{fig:rowley2009_fig2} shows the spectrum of the Koopman operator as computed by the DMD algorithm.  Most eigenvalues lie on the unit circle implying that the flow field is near an attractor.  The time-averaged flow (steady-state component) corresponds to $\lambda =1$ and is indicated in blue in the left image of figure \ref{fig:rowley2009_fig2}.  The rest of the eigenvalues have colors smoothly varying from red to white with the colors corresponding to the magnitude of the associated global mode.  Red corresponds to large magnitudes for the Koopman modes, white to low magnitude.  The magnitudes are given by the total energy of the mode (2-norm).  The right image of figure \ref{fig:rowley2009_fig2} shows the magnitudes of the Koopman modes at each frequency.  The color scheme is the same as for the left image.

If we order the Koopman modes in order of decreasing magnitude, mode 1 corresponds to the time-average flow and the rest come in complex-conjugate pairs; modes 2 and 3 correspond to complex-conjugate eigenvalues and have the same magnitude.  Figure \ref{fig:rowley2009_fig3}, shows the streamwise velocity components of mode 2 (left) and mode 6 (right).  Each mode oscillates at a single frequency, with mode 2 corresponding to a high-frequency ($St = 0.141$) and mode 6 to a low-frequency ($St = 0.0175$); $St$ is the Strouhal number.  These correspond to the tallest red line and the left most red line in the bottom row of figure \ref{fig:rowley2009_fig1}, respectively.  In both images of figure \ref{fig:rowley2009_fig3}, the red surfaces correspond to positive streamwise velocities and blue surfaces to negative streamwise velocities.  Mode 2 is associated with shear layer vortices with additional vortices extending toward the wall.  Mode 6 has large structures along the wall associated with shedding of the wall vortices.  This shedding of wall vortices is coupled to the main jet body as indicated by the mode having structure along the jet body.

\begin{figure*}[ht]
\begin{center}
\framebox[\linewidth][c]{ \raisebox{50mm}{\footnotesize Waiting for permissions. See original publication.}}
\caption[Jet in Crossflow: spectra of traces of two observables. \citet{Rowley:2009ez}]{Jet in Crossflow - The top row is the time signal of the streamwise velocity for a probe near the wall, downstream of the jet orifice (left) and for a probe downstream, in the jet trajectory (right).  In the bottom row, the spectral content of the corresponding probe signals are shown in black.  In red, the part of the spectrum of the Koopman operator captured by DMD algorithm.  Only the positive frequencies are shown since eigenvalues occur in complex-conjugate pairs.  (Original in \citet{Rowley:2009ez}, Journal of fluid mechanics by Cambridge University Press. Reproduced with permission of Cambridge University Press in the format reprint in a journal via Copyright Clearance Center.)}
\label{fig:rowley2009_fig1}
\end{center}
\end{figure*}

\begin{figure*}[ht]
\begin{center}
\framebox[\linewidth][c]{ \raisebox{50mm}{\footnotesize Waiting for permissions. See original publication.}}
\caption[Jet in Crossflow: spectra of the Koopman operator.  \citet{Rowley:2009ez}]{Jet in Crossflow - (left) The spectrum of Koopman operator as identified with the DMD algorithm.  Most of the Koopman eigenvalues are on the unit circle.  The mode corresponding to the time-averaged flow (corresponding to $\lambda =1$) is indicated in blue.  The other eigenvalues are colored from red to white based on the total energy of the associated Koopman mode.  Red corresponds to high-energy modes, white to low-energy modes. (right) The magnitudes of the Koopman modes at the each frequency.  The color scheme is the same as for the image on the left.  (Original in \citet{Rowley:2009ez}, Journal of fluid mechanics by Cambridge University Press. Reproduced with permission of Cambridge University Press in the format reprint in a journal via Copyright Clearance Center.)}
\label{fig:rowley2009_fig2}
\end{center}
\end{figure*}

\begin{figure*}[ht]
\begin{center}
\framebox[\linewidth][c]{ \raisebox{50mm}{\footnotesize Waiting for permissions. See original publication.}}
\caption[Jet in Crossflow: level sets of Koopman modes. \citet{Rowley:2009ez}]{Jet in Crossflow - Koopman modes 2 (left) and 6 (right) corresponding to high- ($St_2 = 0.141$) and low- ($St_6 = 0.0175$) frequencies, respectively.  $St$ is the Strouhal number.  Red contours correspond to positive streamwise velocities and blue contours to negative streamwise velocities.  Reproduced from (Original in \citet{Rowley:2009ez}, Journal of fluid mechanics by Cambridge University Press. Reproduced with permission of Cambridge University Press in the format reprint in a journal via Copyright Clearance Center.)}
\label{fig:rowley2009_fig3}
\end{center}
\end{figure*}

\subsubsection{Self-sustained oscillations in a turbulent cavity\cite{Seena:2011ft}}
\citet{Seena:2011ft} investigated the causes of self-sustained pressure oscillations in fluid flow over a cavity (see figure \ref{fig:seena2011_fig13} for a profile of the experiment geometry) by using the Dynamic Mode Decomposition algorithm.  We denote this domain by $\mbb B$.  The goals of the study were to identify the vortical structures that drove the hydrodynamic oscillations and obtain dynamical information about those structures\cite{Seena:2011ft}.  Both thin and thick incoming boundary layers were studied (Reynold's numbers at the cavity of 12000 and 3000, respectively); only the turbulent case ($Re =  12000$) is covered here.

As with the case of the jet in a crossflow example, the fluid evolution was governed by the incompressible Navier-Stokes equations corresponding to the specified cavity geometry (see \citet{Seena:2011ft} for details on the computational domain).  Abstractly, the state space $M$ was a sequence space of coefficients for basis functions on $\mbb B$.  Since the system was not periodic in all dimensions, the basis functions were not the 3-dimensional trigonometric polynomials (see Remark \ref{rem:basis-space-choice}).  A family of observables, parameterized by points in $\mbb B$, was given by functions mapping a point in $M$ (a sequence of coefficients) to the fluid pressure at a point $\mbf x \in \mbb B$. %; i.e., an observable $f_{\mbf x} : M \to \mbb B$ was given by $P(\mbf x) = f_{\mbf x}(p) = \sum_{i=1}^{\infty} p_{j}\mbf b_{i}(\mbf x)$, where $P(\mbf x)$ is the pressure of the fluid at $\mbf x \in \mbb B$, $\{\mbf b_{i}\}_{i=1}^{\infty}$ is the set of basis functions used to represent solutions of the fluid equations, and $p = (p_1, p_2, \dots) \in M$ is the set of coefficients for the basis functions.

The solutions of the Navier-Stokes equations were computed using the Crank-Nicolson method using a second-order central difference scheme in space\cite{Seena:2011ft}.  Cavity flows at the high Reynold's number were simulated with a Large Eddy Simulation (LES)\cite{Seena:2011ft}.  Broadly speaking, LES filters the governing equations by removing the small-scale structures and replacing them with models.  Consult \citet{Seena:2011ft} for full details of the simulation parameters and \citet{Germano:1991te} for a discussion of LES.  

The vector-valued observable chosen for analysis was the fluid pressure at each of the computational grid points in and above the cavity.  The DMD algorithm was applied to a sequence of 124 flow snapshots recorded after allowing transients to decay\cite{Seena:2011ft}.  Figure \ref{fig:seena2011_fig12a} shows the Koopman eigenvalues resulting from the DMD algorithm.  On the left, almost all of the eigenvalues are seen to be on the unit circle, implying that the flow field is on or near an attractor.  Colors in the plot correspond to the total energy of the corresponding Koopman mode.  On the right, the total energies of the Koopman modes at each frequency are shown.  The four dominant peaks are labeled.  Only positive frequencies are marked since eigenvalues occur in complex-conjugate pairs.  The mode labeled as 1 occurs at $\omega = 0$ and corresponds to the steady component of the flow.  The corresponding Koopman mode is shown in figure \ref{fig:seena2011_fig13}(a).  Solid, black lines correspond to high-pressure regions, whereas broken, gray lines correspond to low-pressure regions.  Koopman modes 2, 3, and 4, corresponding to the labeled peaks in figure \ref{fig:seena2011_fig12a}(b), are shown in figure \ref{fig:seena2011_fig13}(b)-(d).  Modes 2 and 3 correspond to $\omega = 4.6\ rad/s$ and $3.5\ rad/s$, respectively.  The exact frequency corresponding to mode 4 was not reported, but was less than $10\ rad/s$.  The low- and high-pressure regions of the modes oscillate at a fixed frequency and suggest self-sustained oscillations in the cavity\cite{Seena:2011ft}.  These results were consistent with results reported in literature\cite{Seena:2011ft}.  The study successively used the DMD algorithm to identify the large vortices in the cavity driving the self-sustained oscillations.

\begin{figure*}[ht]
\begin{center}
\framebox[\linewidth][c]{ \raisebox{50mm}{\footnotesize Waiting for permissions. See original publication.}}
\caption[Turbulent cavity flow: spectrum of Koopman operator. \citet{Seena:2011ft}]{Turbulent cavity flow - (left) Koopman eigenvalues computed using the DMD algorithm for turbulent flow inside the cavity at $Re = 12000$.  The eigenvalues are colored based on the energy of the corresponding Koopman mode. (right) The energy of the Koopman mode at each frequency, $\omega$ $(rad/s)$. (Original in \citet{Seena:2011ft}, International journal of heat and fluid flow by Institution of Mechanical Engineers (Great Britain) Reproduced with permission of Institution of Mechanical Engineers] in the format reuse in a journal/magazine via Copyright Clearance Center.)}
\label{fig:seena2011_fig12a}
\end{center}
\end{figure*}

\begin{figure*}[ht]
\begin{center}
\framebox[\linewidth][c]{ \raisebox{50mm}{\footnotesize Waiting for permissions. See original publication.}}
\caption[Turbulent cavity flow: DMD modes \citet{Seena:2011ft}]{DMD modes inside the cavity at $Re = 12000$ for modes labeled 1,2,3, and 4 in figure \ref{fig:seena2011_fig12a}. (Original in \citet{Seena:2011ft}, International journal of heat and fluid flow by Institution of Mechanical Engineers (Great Britain) Reproduced with permission of Institution of Mechanical Engineers] in the format reuse in a journal/magazine via Copyright Clearance Center.)}
\label{fig:seena2011_fig13}
\end{center}
\end{figure*}

\begin{rem}
Figures \ref{fig:rowley2009_fig2} and \ref{fig:seena2011_fig12a} show that the Koopman spectrum, as computed by the DMD algorithm, has some eigenvalues inside the unit circle, even though an initial block of data points was discarded before computation so that transients could die out.  They are most likely spurious eigenvalues resulting from the finite truncation of the data and the DMD algorithm and not due to slowly decaying modes contained in the data.  Unfortunately, no detailed investigation has been done on such eigenvalues; most attention has been focused on the modes corresponding to eigenvalues on the unit circle and those inside the unit circle have been ignored.
\qed\end{rem}

\subsubsection{Energy efficiency in buildings\cite{Eisenhower:2010tv,Georgescu:2012tt}}
Energy use and efficiency in buildings has received much attention in recent years and Koopman mode analysis has recently seen application in this field, as well.  A researcher is generally interested in temperature distributions in buildings, as this can tell much about HVAC and controller performance or if the building is operating near its design point\cite{Eisenhower:2010tv,Georgescu:2012tt}.  Such measurements can also be used for model validation\cite{Eisenhower:2010tv,Georgescu:2012tt}.

Analysis requires understanding of the heat flow in the building subject to forcing from weather or even human traffic between areas of the building.  In the most general case, heat transfer equations must be solved for a complicated domain with varying boundary conditions.  The interesting observables are the temperatures at each point in the building, $\mbb B$.  The temperature distributions in the building are assumed to exist in some function space ($L^2(\mbb B, d\mbf x)$, for example).  A temperature distribution can be represented in a basis $\{ b_{1}, b_{2}, \dots \}$ for the function space.  The state space $M$ is the sequence space of coefficients for these basis functions; i.e., a point $p \in M$ is $p = (c_1, c_2, \dots )$, for some coefficients $c_k \in \C$.  An observable maps a point in $M$ to the temperature at a point $\mbf x \in \mbb B$.  When collecting real data, the temperature is measured at a finite number of points, usually dictated by the placement of the temperature sensors by the building designers.  The vector-valued observable of interest is the vector of the temperatures recorded by the building sensors.

In the most general case, the full model for the building is too complicated to solve.  A variety of simplifying assumptions and a numerical package is needed to compute solutions to the building system.  EnergyPlus\cite{Crawley:2000tc} is a simulation package for modeling energy and water use in buildings.  It is a free program offered by the U.S. Department of Energy and is widely used.\footnote{\url{http://apps1.eere.energy.gov/buildings/energyplus/}}
At a high level, an EnergyPlus building model consists of specifying the location of every surface in the building.  A list of rooms is then created and the interaction of the rooms (between shared surfaces) is specified.  For example, heat conduction through a certain type of material may be specified for one wall, whereas radiative heat transfer is important for windows.  Models can also include weather data, HVAC usage, room occupancy, and building water usage.  The software makes the assumption that rooms are well-mixed so that no temperature gradients exist in the room.  Under the assumptions of the software, heat transfer in the building is approximated by a system of coupled ordinary differential equations.

In \citet{Eisenhower:2010tv}, Koopman mode analysis was used to decompose the temperature evolution in a building into purely periodic global modes.  The building investigated was the Y2E2 building at Stanford University and was modeled and simulated in the EnergyPlus software package.  The physical building has 2370 HVAC sensors recording data throughout the building\cite{Eisenhower:2010tv}.  The vector-valued observable chosen was the vector consisting of temperature readings taken at the sensors on the second floor\cite{Eisenhower:2010tv}.  Koopman mode analysis was used to validate the EnergyPlus model by comparing the most energetic Koopman modes for the simulation and the raw data.

Figure \ref{fig:eisenhower2010_fig5} shows the spectral content of the Koopman modes computed using the DMD algorithm with the sensor data on top and the EnergyPlus data on the bottom.  The sensors are along the vertical axis with the period of the Koopman mode along the horizontal axis.  A vertical streak in figure \ref{fig:eisenhower2010_fig5} corresponds to a ``global'' mode; i.e., most sensors are affected by the Koopman mode at that frequency.  In both images, a strong streak is seen at the 24 hour period which is due to the daily forcing from the weather.  Horizontal blue lines indicate near constant temperatures at those sensors and correspond to rooms served by their own fan units\cite{Eisenhower:2010tv}.  

Figure \ref{fig:eisenhower2010_figs6_7} shows the 24 hour Koopman mode for both the sensor data and the EnergyPlus model.  The top row corresponds to the sensor data, the bottom row to the simulation.  The magnitude of the mode is shown on the left with the phase on the right.  Note that the magnitude and phase of the Koopman mode is reported in reference to the outside air temperature (OAT).  The relative magnitude of a mode is given in decibels (dB)
	\begin{equation*}
	\abs{ KM_i } = 20 \log_{10} \abs{\frac{C_i(F)}{ C_{OAT}(F)}}
	\end{equation*}
and the phase is given in degrees
	\begin{equation*}
	\angle KM_i  = \angle C_i(F) - \angle C_{OAT}(F),
	\end{equation*}
where $C_i(F)$ and $C_{OAT}(F)$ are the Koopman modes as computed by the DMD algorithm.
For the sensor data, the magnitude of the 24 hour mode relative to the external temperature was about $-6$ dB throughout the building.  This implied that the peak temperature oscillation inside the building was smaller than the magnitude of the temperature fluctuations outside.  For the EnergyPlus mode, the temperature magnitude inside was much closer to the external fluctuations since the relative magnitude of the mode was about 0 dB.  Significant deviations  in the relative phases were also seen between the sensor data and the model.  These discrepancies implied that the parameters used in the simulation model were incorrect.  The analysis lead to suggestions on the model parameters to modify\cite{Eisenhower:2010tv}.

In practice, even under the simplifying assumptions utilized in programs such as EnergyPlus, a detailed model of the building can be prohibitively expensive to simulate over the time scales needed (e.g., hours up to years).  Methods for model reduction are in order.  Zoning approximations are one approach.  In a detailed EnergyPlus model, each room is treated as a unique thermal zone.  Each room is assumed well-mixed so that the thermal properties are uniform in the room.  The idea of zoning is to lump adjacent rooms together in such a way that the thermal properties can then be assumed uniform across those rooms.  This procedure results in less regions that need to be simulated in the model.  Usually, zoning approximations are performed heuristically\cite{Georgescu:2012tt}.  However, model accuracy is quite sensitive to the zoning approximation used\cite{Georgescu:2012tt}.

\citet{Georgescu:2012tt} used the notion of coherency between Koopman modes (def.\ \ref{def:km-coherency} above) to create zoning approximations for buildings and studied the Engineering Sciences Building (ESB) at the University of California, Santa Barbara as a test case for the methodology.  The particular observable chosen was $F = (f_1, \dots, f_m)^{\msf T}$, where each $f_j : M \to \R$, $j=1, \dots, m$, represented the temperature of a room in the ESB.  In the detailed EnergyPlus model, there were $m = 191$ zones.    An additional physical assumption was imposed so that rooms were grouped into a zone if they were both on the same floor and adjacent as well as $\eps$-coherent.  

The modes of interest were those having periods of one year, 24 hours, 12 hours, 8 hours, and 6 hours.  These corresponded to the most energetic modes.  Figure \ref{fig:georgescu2012_fig4} shows the Koopman modes of the ESB EnergyPlus simulation for the three most energetic modes.  The magnitudes of the Koopman modes are given in the top row.  Units are degrees Celsius relative to the average temperature in the room.  The average temperatures were not reported.   The bottom row corresponds to the phase in radians of the Koopman mode.  Consider the block of four rooms in the center of the right hand side of the building (the ones that are dark blue in the magnitude plot of the year long mode).  By visual inspection, the middle two rooms on this block would be lumped into a single zone.  While the colors of those two rooms vary between each of the six images in figure \ref{fig:georgescu2012_fig4}, within the same image those two rooms have the same color, and hence almost identical values.  These two rooms satisfy the $\eps$-coherency definition for some small $\eps > 0$.

\begin{figure}[ht]
\begin{center}
\includegraphics[width=0.45\textwidth]{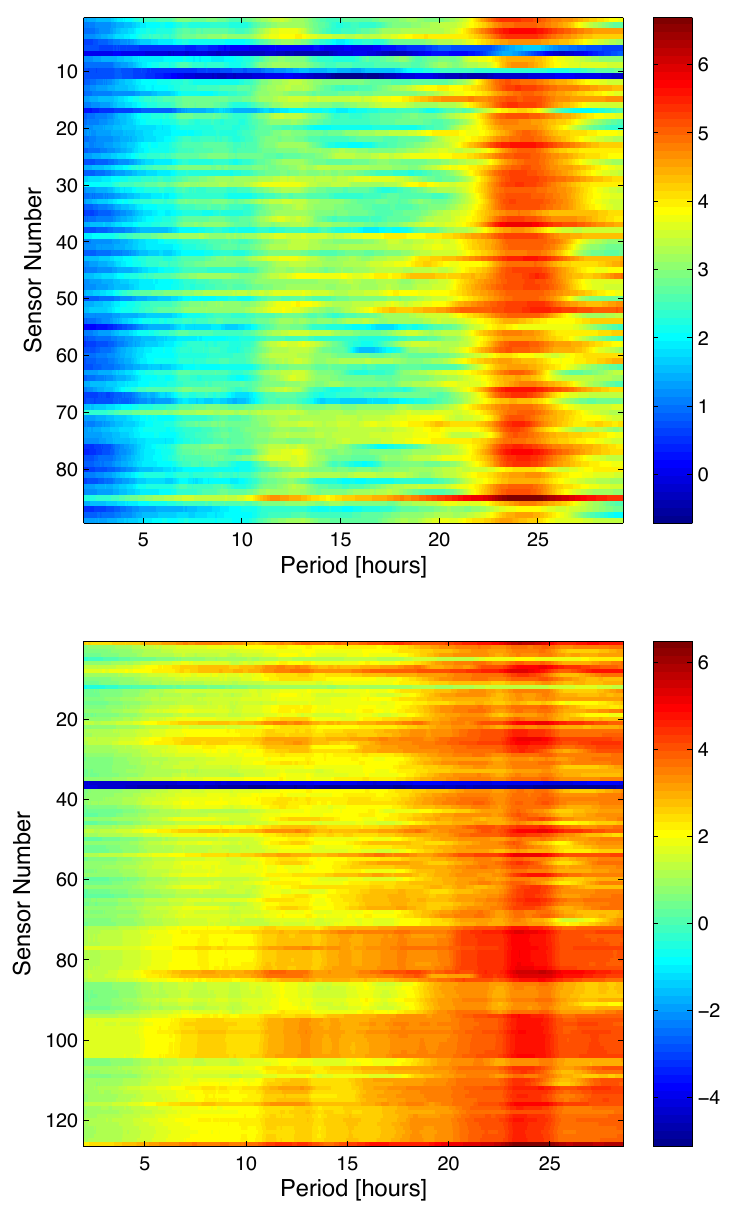}
\caption[Y2E2 building: comparison of data and EnergyPlus Koopman spectra. \citet{Eisenhower:2010tv}]{Y2E2 building decomposition - Koopman spectra for the sensor data (top) and the EnergyPlus model (bottom).  A large spectral content is seen at the 24 hour period due to outdoor forcing conditions.  Horizontal blue lines correspond to near constant temperatures at those sensors.  These sensors correspond to rooms that are served by their own fan units allowing a tight control of temperature.  Reproduced from \citet{Eisenhower:2010tv}.}
\label{fig:eisenhower2010_fig5}
\end{center}
\end{figure}

\begin{figure}[ht]
\begin{center}
\includegraphics[width=0.9\linewidth]{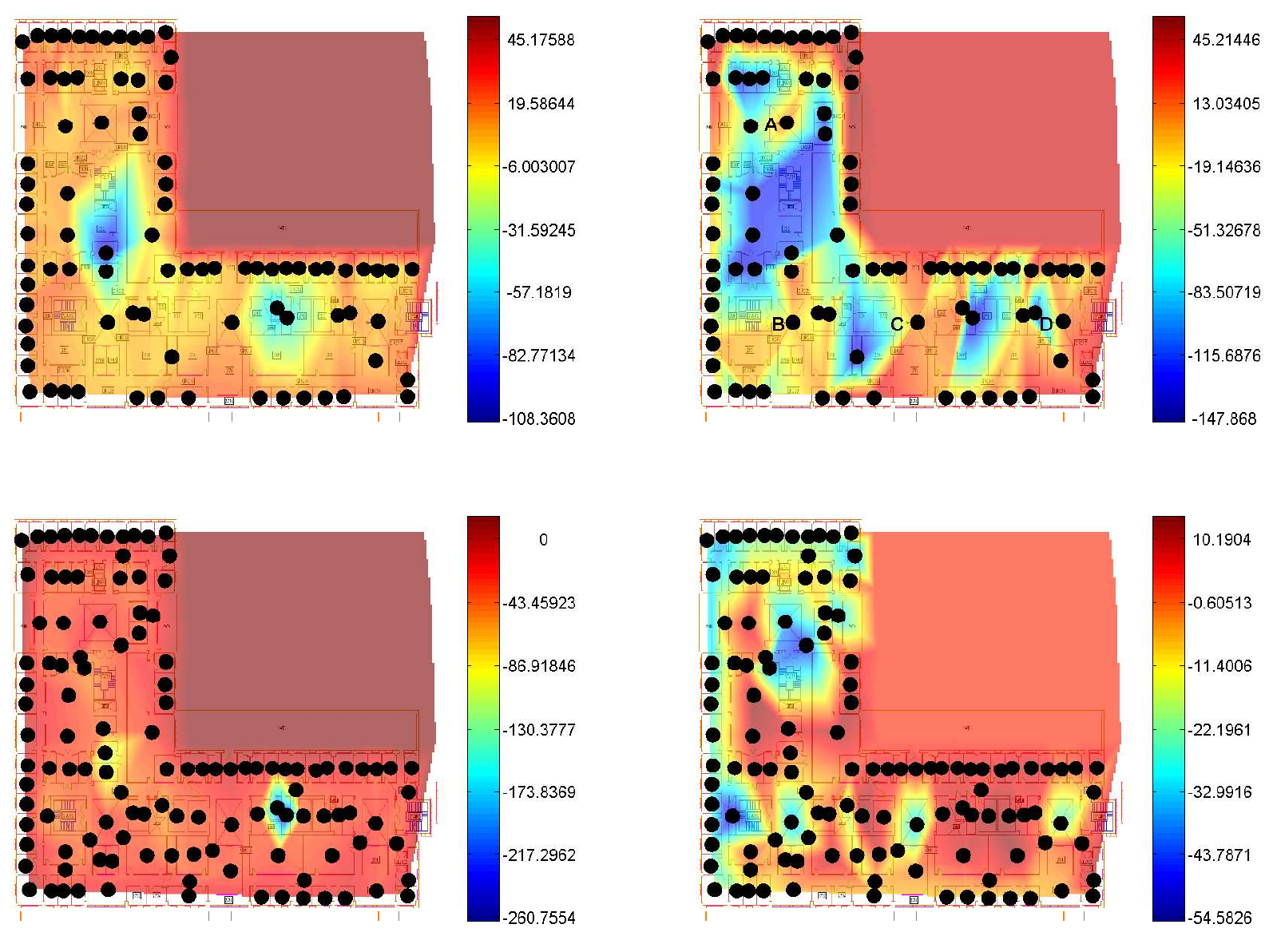}
\caption[Y2E2 building: Koopman mode for 24 hr period. \citet{Eisenhower:2010tv}]{Y2E2 building decomposition - Koopman mode for the 24 hour period.  (top row) The magnitude and phase, respectively, of the raw data's Koopman mode.  (bottom row) The magnitude and phase of the EnergyPlus model's Koopman mode. The magnitude units are given in decibels and phase in degrees.  Both are relative to the external 24 hour mode.  The discrepancy between scales shows the model mismatch with the real data.  Reproduced from \citet{Eisenhower:2010tv}.}
\label{fig:eisenhower2010_figs6_7}
\end{center}
\end{figure}

\begin{figure*}[ht]
\begin{center}
\includegraphics[width=0.6\linewidth]{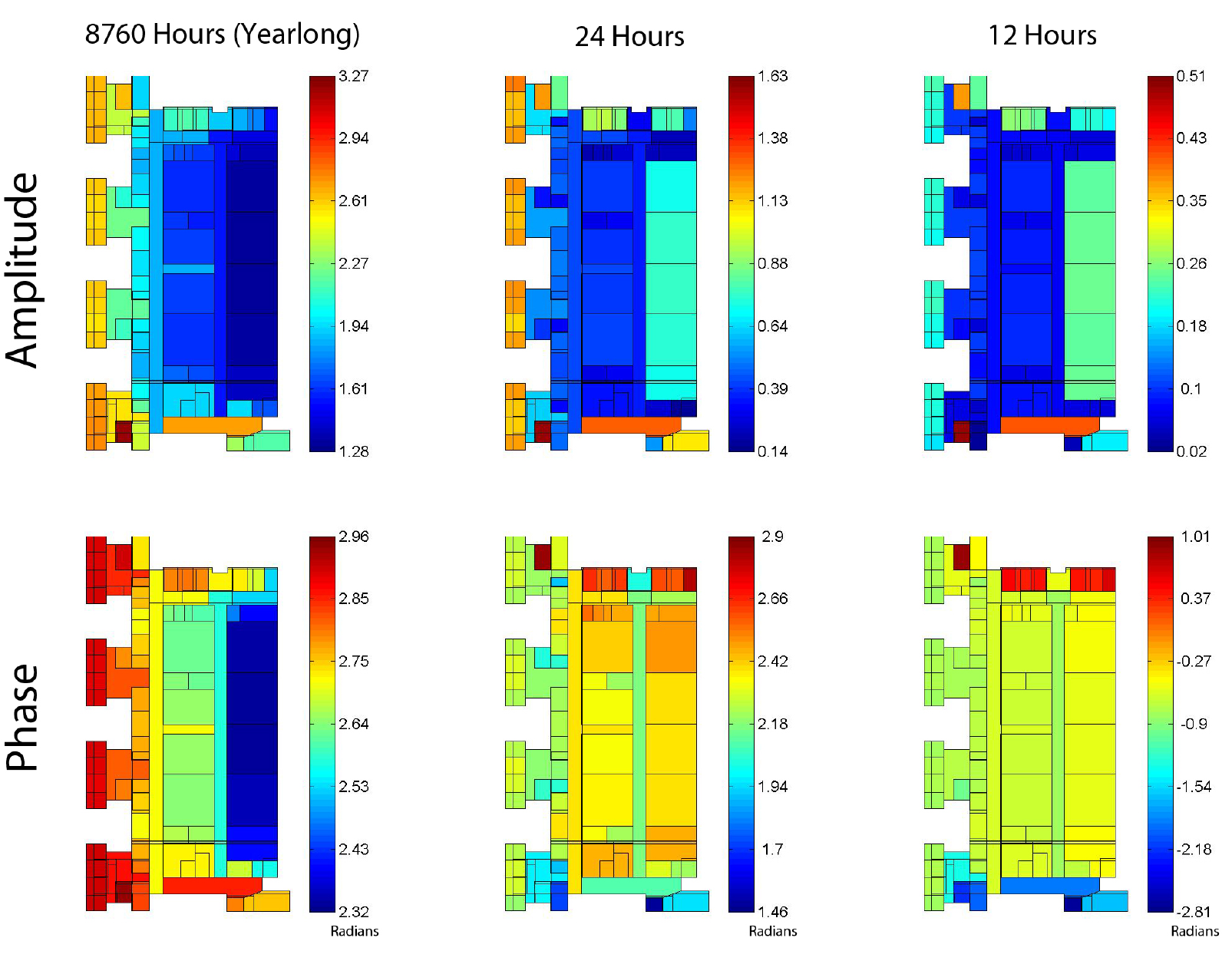}
\caption[Zoning approximations of buildings based on Koopman modes. \citet{Georgescu:2012tt}]{Zoning approximations of buildings - Koopman mode magnitude and phase for the modes having the largest magnitudes.   Only one floor of the ESB is shown.  (top) the magnitudes of the Koopman modes.  Units are in degrees Celsius relative to the mean temperature of the room.  (bottom) the phase of the Koopman mode in radians.  Reproduced from \citet{Georgescu:2012tt}}
\label{fig:georgescu2012_fig4}
\end{center}
\end{figure*}

%\begin{figure*}[htb]
%  \begin{center}
%    \begin{subfigure}[t]{.45\linewidth}
\framebox[\linewidth][c]{ \raisebox{50mm}{\footnotesize Waiting for permissions. See original publication.}}
%    \caption{New England Test System}\label{fig:susuki2011_correction_fig7}
%  \end{subfigure}
%  \begin{subfigure}[t]{.45\linewidth}
\framebox[\linewidth][c]{ \raisebox{50mm}{\footnotesize Waiting for permissions. See original publication.}}
%    \caption{IEEE Reliability Test System-1996}\label{fig:susuki2011_correction_fig12}
%  \end{subfigure}
%  \caption{Distribution of the amplitudes and phases for the components of the three Koopman modes of largest norm.  Numbers with in the figure represent integer labels of the generators.  Circled data points represent coherent generators. Reproduced from \citet{Susuki:2011jq}}
%\end{center}
%\end{figure*}

%%% Local Variables: 
%%% mode: latex
%%% TeX-master: "../koopmanism"
%%% End: 

\section{Analysis of state spaces using eigenquotients}\label{sec:state-space-analysis}
Previous sections discussed the Koopman eigenfunctions in the context of spectral decomposition of the Koopman operator, and we did not pay much attention to values that eigenfunctions take on the state space. In this section, we demonstrate that eigenfunctions carry information about transport between parts of the state space and provide a way to identify  invariant sets. Going further, we will endow the collections of invariant sets with their own metric topology, which allows for associating smaller invariant sets into larger coherent structures.

The material in this section appeared originally in a string of papers \cite{Levnajic:2010gq,Mezic:1999fu,Budisic:2012woa,Budisic:2009iy,Susuki:2009cw,Mezic:2004is} with the technical details given also in dissertations of two of the authors. \cite{Mezic:1994tv,Budisic:2012td}

\subsection{State-space analysis of Koopman eigenfunctions}
\label{sec:level-sets}

The level sets of Koopman eigenfunctions at eigenvalues \(\abs{\lambda}=1\) form invariant sets, and periodic and wandering chains of sets in the state space. The eigenfunction level-set partitions depend not only on the choice of the eigenspace \(E_{\lambda}\) from which eigenfunctions are taken, but also on the choice of a particular function in the eigenspace. However, an exhaustive procedure for analysis of level-set partitions can be devised, based on partition products, such that the result in the limit does not depend on the particular choice of eigenfunctions studied, but rather only on the eigenvalue \(\lambda\), which determines the dynamics of level sets, e.g.,  periodic or invariant. In this section, we first focus on the ergodic partition, which corresponds to analysis of the invariant eigenspace of the Koopman operator, i.e., \(\lambda = 1\). The generalization to ``periodic'' eigenspaces, i.e., those corresponding to \(\lambda = e^{i 2\pi \omega}\) for \(\omega \in \Q\), is straightforward, and we present it at the end of this section. 

From this point onward, we will assume the measurable setup of the Koopman operator: \(T:M \to M\) will be a measurable map between measure spaces \((M, \mathfrak B)\), with at least one invariant measure \(\mu\). The observables \(\mathcal F\) are, at the very least, a subset of complex-valued measurable functions.
Given an invariant function \(\phi\), i.e., a function in invariant eigenspace of the Koopman operator \(\phi \in E_{1}(U)\), we can form its level set partition \(\zeta(\phi) := \{ S_{z} : \forall z \in \C \}\), where the level sets are \(S_{z} := \{ x \in M : \phi(x) = z\}\).  Since \(\phi\) is a measurable function, its level sets \(S_{z}\) are measurable sets. Additionally, they are invariant sets, by virtue that each \(S_{z}\) collects \emph{all} state-space points on which \(\phi\) takes the value \(z\), i.e., if it contains \(x\), it will also contain \(T(x), T^{2}(x), \dots\), since \(\phi\) is constant along trajectories. A partition of the state space into invariant sets is called \emph{a stationary partition}.

Unless the system is ergodic, the choice of \(\phi \in E_{1}\) is not unique (see Remark \ref{rem:simplicity}): often there will be at least two functions \(\phi, \psi \in E_{1}\), which are not linearly related and yield two different partitions \(\zeta(\phi), \zeta(\psi)\). The \emph{product of two partitions} \(\zeta(\phi) \vee \zeta(\psi)\), which contains intersections of sets in each of the partitions (see Fig.~\ref{fig:part-product}), is again a stationary partition. The product is \emph{finer} than either of the original partitions, as sets in \(\zeta(\phi)\) or \(\zeta(\psi)\) can be recovered as unions of sets in \(\zeta(\phi) \vee \zeta(\psi)\). Naturally, given a finite set of functions \(\phi_{k}\), \(i=1,\dots,K\), the product \(\bigvee_{k=1}^{K}\zeta(\phi_{k})\) of the associated partitions can be formed incrementally.

\begin{figure}[htb]
  \centering
  \includegraphics[width=.45\textwidth]{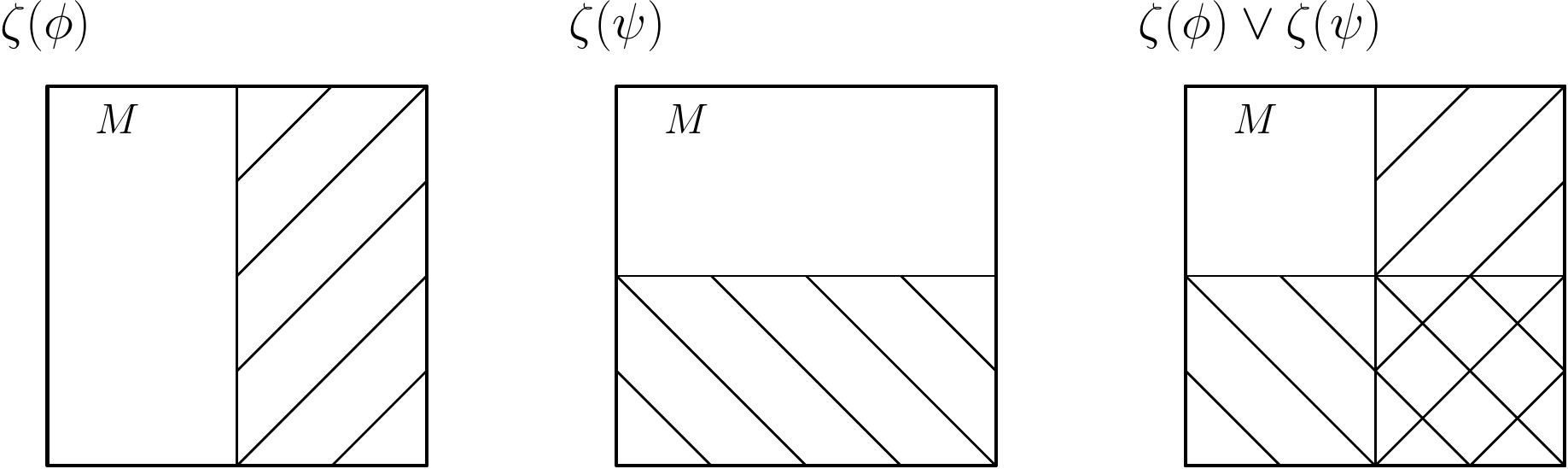}
  \caption{Partitions of the state space $M$ into $\zeta(\phi)$, $\zeta(\psi)$ and their product $\zeta(\phi) \vee \zeta(\psi)$.}
  \label{fig:part-product}
\end{figure}

Given a partition \(\zeta\) of the space and some set \(A\), \emph{the quotient map} \(\pi_{\zeta} : M \to A\) is any map that separates the partition \(\zeta\): \(\pi_{\zeta}(x) = \pi_{\zeta}(y)\) iff \(\exists S \in \zeta\) such that \(x,y \in S\). In other words, all the points in any \(S \in \zeta\) will be mapped to a single point \(p \in A\). The image set \(\xi := \pi(M)\) is termed \emph{the quotient} \(\xi\) of the partition \(\zeta\). If we are given just the map \(\pi_{\zeta}\), we could reconstruct the partition \(\zeta\) from level sets of \(\pi_{\zeta}\), in terms of values of the quotient map codomain \(A\). 

For a product partition \(\zeta_{K} = \bigvee_{k=1}^{K}\zeta(\phi_{k})\), the most straightforward way to construct a quotient map is by arranging eigenfunctions in a vector-valued function to define \(\pi_{K} : M \to \C^{K}\), \(\pi_{K}(x) := ( \phi_{1}(x), \phi_{2}(x), \dots, \phi_{K}(x) )\).  Notice, however, that the quotient map is not unique: any bijection \(h : \C^{K} \to A\), e.g., a coordinate transform, can be used to form another quotient map through composition \(h \circ \pi_{K}\). While the construction of \(\pi_{K}\) might appear academic, it is significant in that it translates the problem of constructing partition \(\zeta_{K}\) from the topological setting, where intersections of sets were used to check whether two points \(x,y\) are in the same product set, to an analytical setting, where the equality \(\pi_{K}(x) = \pi_{K}(y)\) can be checked, either numerically or analytically. We will make use of this fact heavily in Section \ref{sec:quotient}.

Evaluation of eigenfunctions \(\phi_{k}\), which feature in construction of quotient maps, is easy when we are given a method for evaluating the projection of the space of observables onto an eigenspace  \(P_{\lambda} : \mathcal{F} \to E_{\lambda}\). In this case, we can compute a dense set of eigenfunctions in \(E_{\lambda}\) by evaluating \(P_{\lambda} f\) for a dense set \(f_{k} \in \mathcal{F}\). Moreover, when \(\mathcal F\) is a Hilbert space, we can compute the projection on an orthogonal basis for \(\mathcal F\), obtaining a (not necessarily orthogonal) basis for \(E_{\lambda}\). A practical way to numerically evaluate \(P_{\lambda}\) for \(\abs{\lambda} = 1\) using trajectory averages is explained in Section \ref{sec:averaging}.

For \(\lambda = 1\), projecting functions \(f_{k}\) using \(P_{1}\) results in  an infinity of invariant functions that we can attempt to use in forming incremental level-set partitions. It is not trivial that the partition intersection can be extended to the case of an infinity of \(U\)-invariant functions \(\phi_{k}\). The first result of this sort was obtained by \citet{Sine:1968vj} for continuous observables  \(\mathcal{F} = C(M)\). Such setting is fairly restrictive for deterministic dynamics. Therefore, we focus on a more general result\cite{Mezic:1999fu,Mezic:2004is} that holds for \(\mathcal{F} = L^{1}(M,\mu)\). Assume we are given a countable, bounded family \(\phi_{k} \in E_{1} \subset L^{1}(M,\mu)\), for \(k = 1,2, \dots\),  whose span is dense in the invariant eigenspace \(E_{1}\) of the Koopman operator \(U\). The incremental partition into level sets 
\[
\zeta_e = \bigvee_{k=1}^{\infty}\zeta(\phi_{k})
\]
exists as the level set partition of the map \(\pi_{e} : M \to \ell^{\infty}\),  defined as \(\pi_{e} := (\phi_{1}, \phi_{2}, \dots)\), where \(\ell^{\infty}\) is the space of bounded sequences. The map \(\pi_{e}\) is termed the \emph{ergodic quotient map}, with \emph{ergodic quotient} \(\xi_{e} := \pi_{e}(\zeta)\) as its image set.

While the concept of the ergodic partition was known since at least Rokhlin, \cite{Rokhlin:1949te,Rokhlin:1966te} this computationally feasible construction is fairly recent.\cite{Mezic:1999fu,Mezic:2004is} The justification for the name of the ergodic partition comes from analyzing restrictions of $T$ to invariant sets $S \in \zeta_e$. Each $S$ carries a measure $\mu_S$ such that the restricted dynamics $T : S \to S$ is \emph{ergodic with respect to} $\mu_S$. By one of the several equivalent definitions of ergodicity,\cite{Petersen:1989uv} this  means that for any \(f \in L^{1}(S, \mu_{S})\)
\begin{equation}
\int_{S} f\mu_{S} = \lim_{N \to \infty} \frac{1}{N} \sum_{n=0}^{N-1} f \circ T^{n}(x),\label{eq:erg-space-time}
\end{equation}
at almost every \(x \in S\), with respect to \(\mu_{S}\). The measures \(\mu_{S}\) are called \emph{ergodic measures} and are extreme points of the space of invariant measures (by the Ergodic Decomposition Theorem). \cite{Katok:1995th} Moreover, integrals against ergodic measures define the projection onto the invariant eigenspace \(P_{1} f(x) = \int_{S(x)} f d\mu_{S(x)}\), where \(S(x)\) denotes the ergodic set containing point \(x\). This is a consequence of the equality \eqref{eq:erg-space-time} along with Mean Ergodic Theorems,\cite{Budisic:2012td} e.g., von Neumann theorem for \(\mathcal F = L^{2}(M,\mu)\) or Yosida ergodic theorem for general Banach spaces, used in earlier sections as Theorem \ref{thm:krengel-mean-ergodic-thm}.

The ergodic partition is unique up to \(\mu\)-measure zero sets, i.e., given two stationary partitions \(\zeta_{1}\) and \(\zeta_{2}\) that satisfy the above properties, for any \(S_{1} \in \zeta_{1}\) there exists \(S_{2} \in \zeta_{2}\), such that \(\mu(S_{1} \triangle S_{2}) = 0\). Consequently, the partition \(\zeta_{e}\) does not depend on the choice of functions \(\phi_{k}\) used to construct it: different choices of functions just yield different \emph{representations} of the quotient map \(\pi_{e}\), however, the limit partition will be the same.

When there is only one \(T\)-invariant measure \(\mu\), up to multiplication by a scalar, the system is \emph{uniquely ergodic}. In that case, the eigenspace \(E_{1}\) is one-dimensional, containing only a.e.\ constant functions. In turn, the ergodic partition contains a single set, and the ergodic quotient map maps almost all points (with respect to \(\mu\)), into a single point in the sequence space \(\ell^{\infty}\). While uniquely ergodic systems might be of interest in general, they are not the target of this approach. The ergodic quotient analysis is intended to help the analysis of state spaces which contain a lot of ergodic measures, e.g., systems with families of periodic or quasiperiodic orbits and systems possessing interspersed regions of regular and irregular dynamics.

For most conservative regular systems, ergodic partitions are similar to partitions of the state space into orbits. However, in zones of irregular dynamics, e.g., on strange attractors and in zones of strong mixing, it is difficult to gain intuition about behavior of the systems simply from orbits.  Orbits are inherently zero-dimensional (for maps) or one-dimensional sets (for flows), but in chaotic regimes, every one of them densely fills a bigger set, possibly even a set of a positive invariant measure, e.g., positive volume. Moreover, in a mixing zone any two trajectories look nothing alike, yet they are contained in the same set and have the same statistical behavior. Through such reasoning, we might be interested not in description of trajectories themselves, but rather in description of minimal invariant sets containing individual trajectories. Precisely, the ergodic partition \(\zeta\) is a \emph{measurable hull} of the decomposition of the space into orbits \(T^{n}(x)\), i.e., a partition of the space into minimal measurable invariant sets that contain orbits. \cite{Rokhlin:1966te} In this sense, \(\zeta\) is the appropriate counterpart of the state space portrait in the context of measurable dynamical systems.

The understanding of the ergodic partition, derived from functions in the invariant eigenspace \(E_{1}\), can be extended to eigenspaces \(E_{\lambda}\) for which \(\abs{\lambda} = 1\). When the eigenfunctions are taken from a  ``periodic'' eigenspace \(E_{\lambda}\), for an eigenvalue with property \(\lambda^{b} = 1\), for some \(b \in \N\) (period), the extension is straightforward. Level sets  of \(\phi \in E_{\lambda}\) are then periodic sets, which are arranged in \(b\)-\emph{periodic chains}, \( \{ S_{a}, T(S_{a}), T^{2}(S_{a}), \dots, T^{b-1}(S_{a}) \}\). Similarly, if \(\lambda = e^{i2\pi\omega}\) is such that \(\omega \not \in \Q\), the chains of level sets extend into infinity, forming \emph{wandering chains}.
The partition \(\zeta(\phi)\) is again an invariant partition, in the sense that for each \(S \in \zeta(\phi)\), \(T(S) \in \zeta(\phi)\). However, \(\zeta(\phi)\) is not stationary, as \(S\) are not invariant sets; it is  an invariant \(b\)-periodic partition or, when \(\omega \not \in \Q\), an invariant wandering partition. The partition products also generalize: if \(\phi, \psi \in E_{\lambda}\), \(\zeta(\phi) \vee \zeta(\psi)\) will also be a \(b\)-periodic/wandering invariant partition. 

The invariant measures can be generalized to the concept of complex eigenmeasures,\cite{Mezic:2004is} which satisfy
\[
\mu(T^{-1} S) = \lambda^{-1} \mu(S).
\]
Integrating against extrema of eigenmeasures enables us to evaluate \(P_{\lambda}\) away from \(\lambda = 1\), and construct \emph{eigenquotient} maps \(\pi_{\lambda}(x) = (\dots, P_{\lambda} f_{k}, \dots)\), analogously to constructions of ergodic quotient maps from a basis \(f_{k}\) for \(\mathcal F\). In this sense, the ergodic quotient map is the eigenquotient map at \(\lambda = 1\).

The eigenquotient maps collect basis functions for the eigenspace \(E_{\lambda}\). Note that this basis set might be overdetermined, since we do not know ahead of time what \(\dim E_{\lambda}\) is. In Remark \ref{rem:simplicity}, we have indicated that the dimension of \(E_{\lambda}\) is bounded by the number of mutually singular components of measure \(\mu\) which was used to define the space of observables as \(\mathcal F = L^{p}(M,\mu)\). These components are precisely the ergodic measures, and there are as many of them as there are ergodic sets. 

Given an eigenfunction \(\phi \in E_{\lambda}\), for \(\abs{\lambda}=1\), its modulus and complex angle functions are respectively \(\abs{\phi}\), \(\angle \phi\), for which \(\phi(x) = \abs{\phi(x)} \exp[i2\pi\angle\phi(x)]\).
The modulus function \(\abs{\phi}\) is an invariant function, since \(U\abs{\phi} = \abs{\phi \circ T} = \abs{\lambda \phi} = \abs{\phi}\). On the other hand, \(\angle \phi \in E_{\lambda}\) is a factor map: it conjugates the dynamics of \(T\) with a dynamical system \(\theta \mapsto \theta + \omega\), evolving on a circle \(\mathbb S^{1}\), with \(\omega\) as the angle of the eigenvalue \(\lambda = r e^{i2 \pi \omega}\),\cite{Mezic:2004is,Wichtrey:2010ta} since 
\[
\angle \phi(Tx) = \angle \phi(x) + \omega.
\]

To illustrate what these functions might reveal about dynamics, we will take the Chirikov Standard Map
\begin{align} 
  \begin{aligned}
    x_{n+1} &= x_{n} + p_{n} + \epsilon \sin(2\pi x_{n}) \\
    p_{n+1} &= p_{n} + \epsilon \sin(2\pi x_{n})
  \end{aligned} \label{eq:std-map}
\end{align}
with \(\epsilon=0.15\). The system preserves the area measure \(dm\) , therefore, the Koopman operator \(U:L^{2}(M,dm) \to L^{2}(M,dm)\) is unitary and its eigenvalues all lie on the unit circle. Figure \ref{fig:stdmap-efunctions} shows 
level sets of an invariant eigenfunction, and modulus and angle of an eigenfunction at \(e^{i2\pi\omega}\) for \(\omega = 1/3\), demonstrating the distinction between \(\phi \in E_{1}\) and the modulus of \(\psi \in E_{\lambda}\).

\begin{figure*}[htb]
  \centering
  \begin{subfigure}[t]{.3\linewidth}\centering
    \includegraphics[height=42mm]{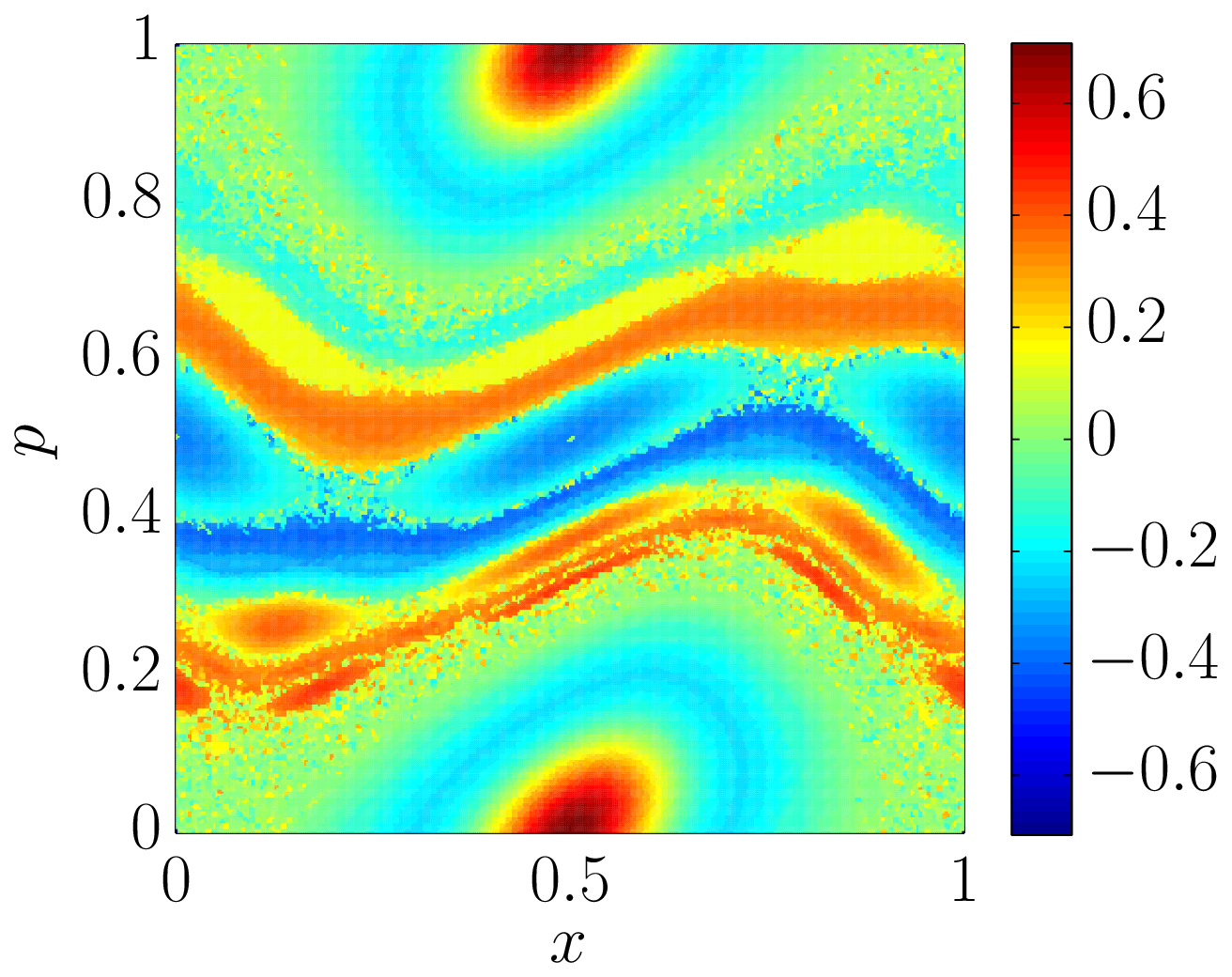}
    \caption{Invariant eigenfunction ${\phi \in E_{1}}$.\label{fig:stdmap-invariant}}
  \end{subfigure}\hspace{.04\linewidth}
  \begin{subfigure}[t]{.3\linewidth}\centering
    \includegraphics[height=42mm]{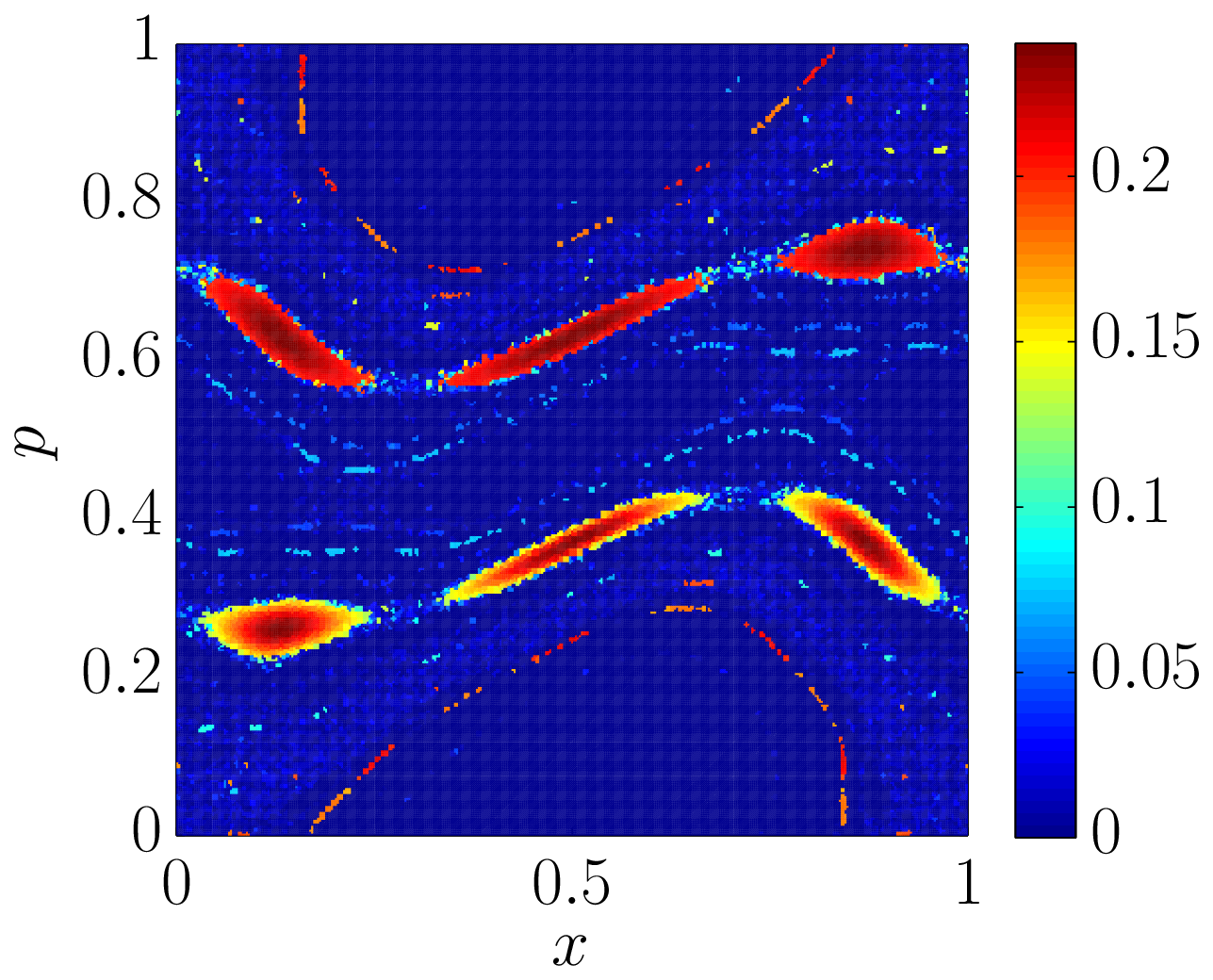}
    \caption{Modulus \(\abs{\psi}\) of a \(3\)-periodic eigenfunction $\psi$. Modulus of any eigenfunction for \(\abs{\lambda} = 1\) is an invariant function. \label{fig:stdmap-modulus}}
  \end{subfigure}\hspace{.04\linewidth}
  \begin{subfigure}[t]{.3\linewidth}\centering
\includegraphics[height=42mm]{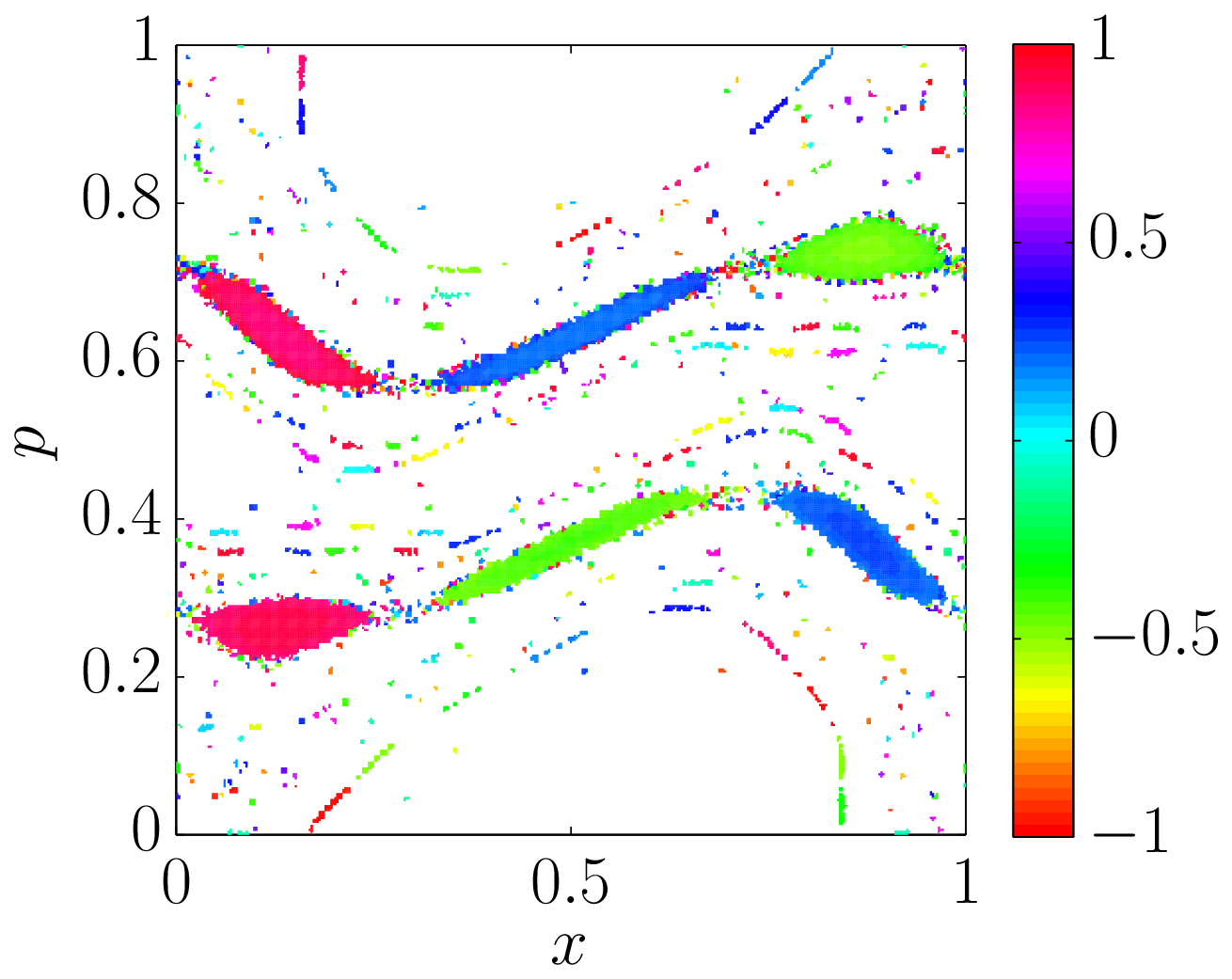}
    \caption{Angle \(\angle \psi\)  of a \(3\)-periodic eigenfunction $\psi$. Angle of an eigenfunction at is a factor map for dynamics. (Only regions where angle is well defined, i.e., \(\psi \not = 0\), are colored.) \label{fig:stdmap-angle}}
  \end{subfigure}
  \caption[Level sets of invariant and periodic eigenfunctions.]{Level sets of eigenfunctions \(\phi \in E_{1}\) and \(\psi \in E_{\lambda}\) for  \(\lambda = e^{i2\pi\omega}\) at \(\omega = 1/3\) for the Chirikov Standard Map at \(\epsilon = 0.15\). (Functions were evaluated by computing Fourier averages, see Section \ref{sec:averaging})} \label{fig:stdmap-efunctions}
\end{figure*}

Even though both \(\abs{\psi}\) and \(\phi\) are invariant functions, \(\abs{\psi}\) is non-zero only over trajectories that are periodic with period \(3\). At \(\epsilon = 0.15\), such trajectories are mostly concentrated in two large groups of periodic chains near the middle of the state space. On the other hand, invariant functions that are not formed by taking absolute values of periodic functions do not necessarily make a distinction between periodic and non-periodic dynamics. The angle function \(\angle \psi\) can be used to infer the order in which sets are ordered in the associated periodic chains. Even such simple visualizations can be useful aids in quick assessment of a new dynamical system, by identifying regions that are not dynamically connected.

%%% Local Variables: 
%%% mode: latex
%%% TeX-master: "../koopmanism"
%%% End: 
\subsection{Geometry of eigenquotients}
\label{sec:quotient}

In previous section, we explained how eigenfunctions of the Koopman operator connect to invariant and periodic structures   in the state spaces of dynamical systems. We have also indicated that the finest invariant partitions \(\zeta\) can be represented in sequence spaces using quotient maps \(\pi : M \to \ell^{\infty}\) and their image sets, eigenquotients \(\xi = \pi(M)\). In this section, we will be treating geometry of sets \(\xi\), explaining it for the case of the ergodic quotient, which corresponds to invariant sets. As explained in the previous section, a generalization to periodic sets is straightforward.

Identification of ergodic sets can be thought of as checking whether for two state space points \(p_{1},p_{2} \in M\) it holds that \(\pi(p_{1}) = \pi(p_{2})\) or not; essentially, we are using a discrete topology on ergodic quotient \(\xi\) to compare points in the state space \(M\). We can, however, use other metric topologies on the ergodic quotient to extract additional information about the state space, e.g., to obtain a low-dimensional representation of dynamics, or identify functions acting as integrals of motion over invariant sets, not necessarily entire state spaces. By a low-dimensional representation of dynamics, we mean that we are looking for a minimal set of directions in the state space, in which we can move a point across the boundary of an ergodic set \(S\) and land in another ergodic set \(S'\) where trajectories look similar, on average, to those in \(S\).

\begin{figure}[htb]
  \centering
\framebox[\linewidth][c]{ \raisebox{50mm}{\footnotesize Waiting for permissions. See original publication.}}
  \caption[Reeb graphs for three Hamiltonian systems. \citet{Budisic:2012woa}]{Sketches of two Hamiltonian oscillators: a harmonic oscillator and a double-well oscillator. Top row: state space portrait. Bottom row: graphic representation of Fomenko-Reeb graphs for level-sets of Hamiltonian functions. Dashed arrows indicate how moving an initial condition across the level sets reflects on the Reeb graph.
(Based on original in \citet{Budisic:2012woa}, Physica D: Nonlinear phenomena by North-Holland. Reproduced with permission of North-Holland in the format reuse in a journal/magazine via Copyright Clearance Center.)}
  \label{fig:reeb}
\end{figure}

As a motivation for the analysis of ergodic quotient we take the Morse theory analysis of Hamiltonian systems. Consider two state spaces sketched in the top row of Figure \ref{fig:reeb}. Each state space contains a region where trajectories are tightly layered together, and where they can be parametrized using a single continuous parameter: distance from elliptical fixed points.  For Hamiltonian systems, these properties can be formally stated by inducing a topology on the state space based on the energy function. In other words, the ``similarity'' of neighboring ergodic sets that we are after is here represented by the similar values that the energy function attains on neighboring ergodic sets.

Level sets of any monotonic function of the Hamiltonian can be represented by \emph{Reeb graphs} sketched next to the state spaces. Reeb graphs are a topological tool, featured in Morse theory, which analyzes manifolds through level sets of differentiable functions on them. The Reeb graphs provide a concise description of changes in topology of such level sets as the level, i.e., value of the function, is changed. In dynamical systems, they have been used to study level sets of integrals of motion for integrable Hamiltonian systems;\cite{Bolsinov:2004dw,Dragovic:2009tf} in this context, they are known as Fomenko-Reeb graphs. Since integrals of motion are invariants of projections \(P_{1}\) onto the \(\lambda=1\) eigenspace, in certain settings we could draw parallels between Fomenko-Reeb graphs and the ergodic quotient \(\xi\).

The vertices in Fomenko-Reeb graphs correspond to those values of the Hamiltonian \(H : M \to \R\) at which the topology of level sets of \(H\) changes. The edges connecting the vertices correspond to families of connected components of level sets that are homotopically related with respect to continuous variation of values of \(H\). Figure \ref{fig:reeb}  illustrates what the Fomenko-Reeb graphs look like for two simple Hamiltonian systems. Without going into details of their construction here, an important feature of these graphs is that continuous segments in them correspond to one-parameter families of periodic orbits. For the double-well oscillator, strands merge at the value of energy where two separate wells merge across the separatrix to form the outer family of periodic orbits. Therefore, by looking at number of independent parameters in a Fomenko-Reeb graph, we can deduce the number of families of periodic orbits in the state space.

Currently, Fomenko-Reeb graph analysis is formulated only for integrable Hamiltonian systems.\cite{Bolsinov:2004dw} By analyzing the ergodic quotient, we can establish an approach similar in spirit, even for systems that do not have an explicit energy function, but might contain families of invariant sets in the state space. The ergodic quotient map collects, in a sense, \emph{all} possible invariant functions, that can be thought as generalizations of energy. Unlike the case of Hamiltonian systems, where topology of the energy codomain was used to define topology on the state space, the ergodic quotient \(\xi\) is a subset of an infinite-dimensional sequence space \(\ell^{\infty}\), which can be endowed with many euclidean-like topologies that are not all mutually equivalent. For this reason, we need to choose a metric structure \((\xi, d)\), where \(d\) is a distance function, so that we can obtain a context in which a generalization of Reeb analysis of Hamiltonian systems to non-Hamiltonian systems is possible. 

The role played by the Hamiltonian function in Fomenko-Reeb graphs is taken up by the quotient map \(\pi\). The image space of Hamiltonian was \(\R\) with its natural metric topology; the \(\ell^{\infty}\) topology on the ergodic quotient \(\xi\) is too sensitive to establish analogous arguments. Instead, we will use a Sobolev space metric which requires interpretation of the image under the quotient map \(\pi(x)\) as a spatial Fourier transform of ergodic measures \(\mu_{x}\). Within this setting,  we seek to extract coherent onion layers in the state space by identifying connected components in \((\xi, d)\). The challenge lies in extracting a particularly appropriate low-dimensional parametrization of \(\xi\) even though it is embedded within the infinite-dimensional space \(\ell^{\infty}\).

 First, as observables which will be projected onto \(E_{1}\) we choose the  normalized harmonic basis for \(\mathcal F\), e.g., on \(M \cong [0,1]^{D}\) 
\[
f_{k}(x) = e^{ i 2\pi k \cdot x }, 
\]
where \(k \in \Z^{D}\) and \(k \cdot x = \sum_{j=1}^{d} k_{j} x_{j}\). This results in the representation of the ergodic quotient map \(\pi\) 
\begin{align*}
  \pi(p) &= ( \dots, [P_{1} f_{k}](p), \dots )\\
  &= ( \dots, \hat \mu_{p}(k), \dots )
\end{align*}
that  maps the points of the state space \(p\) to Fourier coefficients \(\hat\mu_{p}(k) := \int_{M} f_{k} d\mu_{p}\) of the associated ergodic measure \(\mu_{p}\). This opens up opportunities to use metrics available on the space of Fourier coefficients, e.g., weighted \(\ell^{2}\) metrics, to compare dynamics in different ergodic sets, and generalize the Fomenko-Reeb analysis from Hamiltonian systems to a more general class of systems. Note, however, that when we are interested only in identifying ergodic partitions using discrete topology on \(\xi\), i.e., ask whether \(x,y \in M\) are elements of the same ergodic set \(S \in \zeta\) by checking if \(\pi(x) = \pi(y)\) , any continuous basis of observables will be sufficient.\cite{Mezic:2004is} The choice of basis is largely a matter of convenience of constructing other metric structures on \(\xi\), alternative to the discrete metric.

Due to boundedness of Fourier coefficients, the ergodic measures are elements of the Sobolev space \(W^{2,-s}\), whose norm can be defined as a weighted euclidean norm
\begin{align}
  \norm{ \mu }^{2}_{2,-s} := \sum_{k \in \Z^{D}} \frac{ \abs{\hat
      \mu(k)}^{2} }{\left[1 + (2k\pi)^{2} \right]^{s}},\label{eq:neg-sob-norm}
\end{align}
with the index determined as \(s = (D+1)/2\) where \(D = \dim M\). We use \(\ell^{2,-s}\) as the symbol for the associated Fourier coefficient space with the weighted euclidean norm. An excellent introduction to the Sobolev space theory is the classic textbook by \citet{Adams:2003wi}, which contains all the material relevant for this analysis.\cite{Budisic:2012td}

To establish a metric structure on the ergodic quotient, we use the \(\norm{.}_{2,-s}\)\nobreakdash-induced metric, and analyze continuity of ergodic quotient maps \(\pi\) in it. The \(\ell^{2,-s} \cong W^{2,-s}\) metric structure induces a distance-like function on the state space 
\begin{equation}
d_{s}(p_{1},p_{2})^{2} := \norm{ \mu_{p_{1}} - \mu_{p_{2}} }^{2}_{2,-s}.\label{eq:sob-pseudodistance}
\end{equation}
This function is not a true distance but a pseudo-distance precisely because the points at zero distance from each other are those that lie in the same ergodic set. Clearly, when the system is uniquely ergodic with respect to \(\mu\), the function \(d_{s}\) evaluates to zero for \(\mu\)-almost any pair of points \(p_{1}, p_{2}\).

\begin{rem}The choice of the Sobolev space \(W^{2,-s}\) has a justification in comparison of ergodic measures interpreted as measures of residence times of trajectories in measurable sets of a compact state space \(M\).\cite{Mathew:2011ev,Mathew:2005eq} Let \(B(x,r)\) denote euclidean balls in \(M\) and \(\chi_{x,r}\) their characteristic (indicator) functions. The quantity \(\mu_{p}[B(x,r)]\) is the residence time of trajectory \(T^{n}(p)\) in \(B(x,r)\), due to ergodicity of the system \(T : S \to S\), i.e., equality of integrals and time averages \eqref{eq:erg-space-time}:
\[
\mu_{p}[B(x,r)] = \int_{S} \chi_{x,r} d\mu_{p} = \lim_{N \to \infty} \frac{1}{N} \sum_{n=0}^{N-1} \chi_{x,r} \circ T^{n}(p)
\]
where \(S \subset M\) is the element of the ergodic partition containing the initial condition \(p\). 

The distance between trajectories originating at \(p_{1}, p_{2}\) can then be formulated by integrating the difference in residence times over all spherical sets, with \(R\) chosen such that \(M \subset B(x,R)\):
\begin{equation}
d(p_{1},p_{2})^{2} := \int_{\mathrlap 0}^{\mathrlap R}\int_{M} \abs{\mu_{p_{1}}[B(x,r)]-\mu_{p_{2}}[B(x,r)]}^{2}dx dr. \label{eq:emp-pseudodistance}
\end{equation}
The choice for the index of Sobolev norm \(s = (D+1)/2\) results in the equivalence of norms, i.e., existence of constants \(\alpha,\beta > 0\) such that for any \(p_{1},p_{2} \in M\) 
\[
\alpha\ d(p_{1},p_{2}) \leq d_{s} = \norm{\mu_{p_{1}} - \mu_{p_{2}}}_{2,-s} \leq \beta\ d(p_{1},p_{2}).
\]
This argument shows that \eqref{eq:emp-pseudodistance} is again only a pseudo-distance, like \(d_{s}\).
We will return to this interpretation in Section \ref{sec:ctsinds}; the formulation \eqref{eq:neg-sob-norm} is more useful in analysis of the ergodic quotient \(\xi\).
\end{rem}

The space $(\xi, \norm{.}_{2,-s})$ can be analyzed computationally, allowing for extraction of coherent structures in state spaces of dynamical systems. We are looking to extract connected components of \(\xi\), i.e., curves $\mathcal C \subset \ell^{2,-s}$. In our motivational Hamiltonian examples, regions mapping to curves in Fomenko-Reeb graphs (see Fig.~\ref{fig:reeb}) contained an onion layering of trajectories, which could be interpreted as regions of uniform dynamical behavior. We can visualize such regions by coloring all the points $x \in \pi^{-1}(\mathcal C)$ using the same color. 

In a recent paper,\cite{Budisic:2012woa} we have presented an algorithm that performed the analysis described above. Using averaging of harmonic observables along trajectories (see Sec. \ref{sec:averaging}), it evaluates an approximation of the ergodic quotient map \(\pi\) on a set of initial conditions \(\{p_{n}\}_{n=1}^{N}\)  covering the region in state space that is to be analyzed. The connected components are extracted based on pairwise evaluations of the induced Sobolev metric \(d(p_{1}, p_{2}) \sim \norm{\mu_{p_{1}} - \mu_{p_{2}}}_{2,-s}\). 

The Diffusion Maps algorithm\cite{Coifman:2005bk,Coifman:2006cy} was used to provide a coordinate change for the ergodic quotient \(\xi\).  The particular ergodic quotient map \(\pi\) used earlier was constructed using averaged harmonic functions as coordinates. This choice was driven by desire to efficiently evaluate the Sobolev metric as a weighted euclidean metric. However, the harmonic basis set is chosen without any regard for intrinsic geometry of \(\xi\), and therefore it might not be the most efficient way of analyzing geometry of \(\xi\) .

The Diffusion Maps algorithm interprets \(\xi\) as a heat-conductive object, and computes the modes of heat spread along it, resulting in a coordinate change \(\Psi : \xi \to \ell^{2}\).
The \(\ell^{2}\) distance between points \(\Psi(\xi)\) corresponds to the diffusion distance: an intrinsic, coordinate-independent distance along the ergodic quotient. This reflects the fact that the Diffusion Maps algorithm obtains intrinsic geometry of \(\xi\), regardless of the coordinate system that \(\xi\) was originally represented in. Components of \(\Psi\) are functions \(\psi_{k} : \xi \to \R\), ordered with \(k \in \N\) according to the spatial scales over which they vary. For example, if \(\xi\) is a simple line segment, as it is the case for the harmonic oscillator, \(\psi_{k}\) will just be Legendre polynomials: solutions of the heat equation with no-flux boundary condition. For more general \(\xi\), the diffusion modes \(\psi_{k}\) become more complicated, however, they retain the scale-ordering of harmonic functions. The \(\psi_{k}\) are good candidates for the analysis similar to the Reeb graph analysis for manifolds, which was previously noticed in medical image analysis.\cite{Shi:2008ef}

As mentioned earlier, a bijection composed with the ergodic quotient map \(\pi\) results in another ergodic quotient map. Therefore, by forming \(\chi := \Psi \circ \pi\), with components \(\chi_{k} = \psi_{k} \circ \pi\)  we obtain the ergodic quotient represented in a ``good'' coordinate set, which can be efficiently truncated to obtain low-dimensional representations of the ergodic quotient. The euclidean distance over a truncated set of diffusion coordinates approaches the diffusion distance limit exponentially fast, in number of diffusion modes retained \cite{Coifman:2005bk}. Since euclidean distance is the most common distance used in applied problems, a host of off-the-shelf algorithms can be used to post-process the ergodic quotient, e.g., a \(k\)-means clustering or proximity graph analysis for extraction of connected components.

\begin{figure*}[htb]
  \centering
  \begin{subfigure}[t]{.45\linewidth}
    \includegraphics[height=50mm]{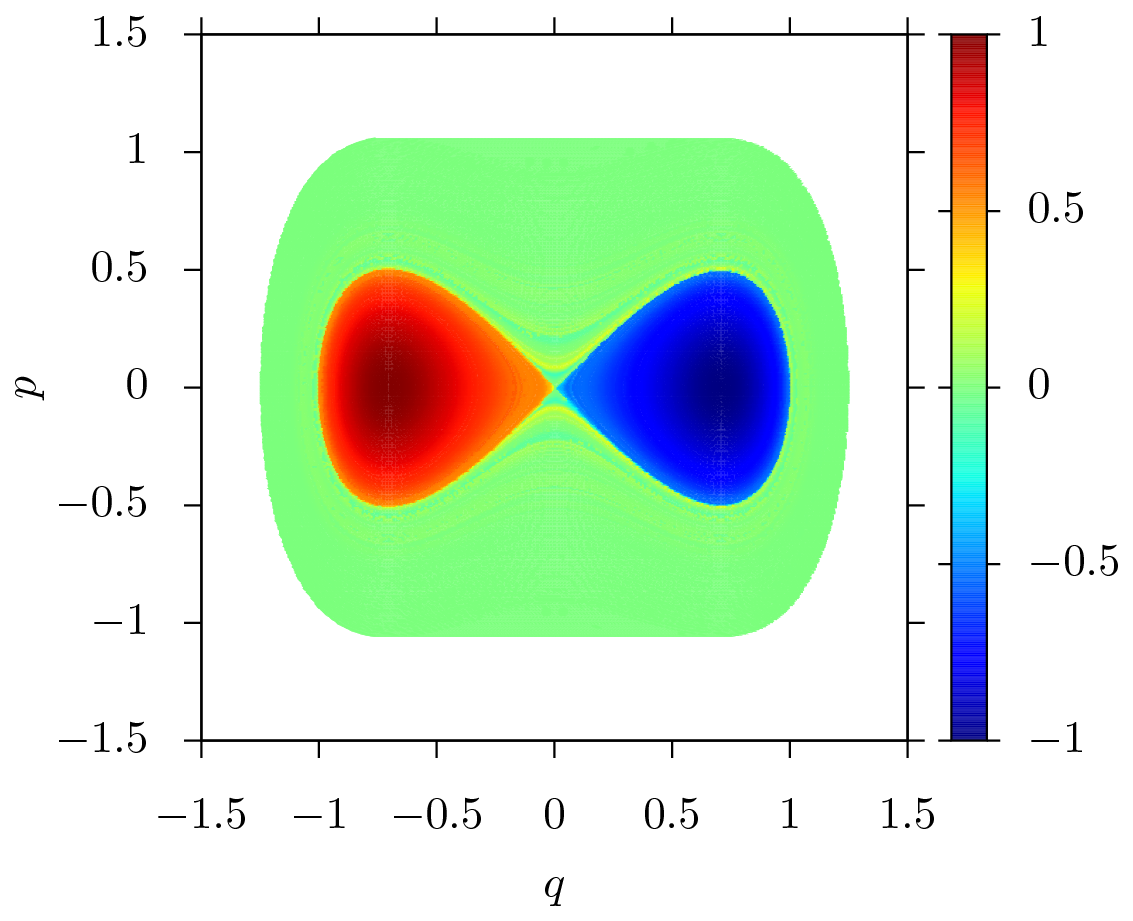}
    \caption{State space of a double-well potential. Color is the first diffusion coordinate $\chi_1$, corresponding to color in \protect\subref{fig:duffing-eq}}\label{fig:duffing-ss}
  \end{subfigure}
  \begin{subfigure}[t]{.45\linewidth}
    \includegraphics[height=50mm]{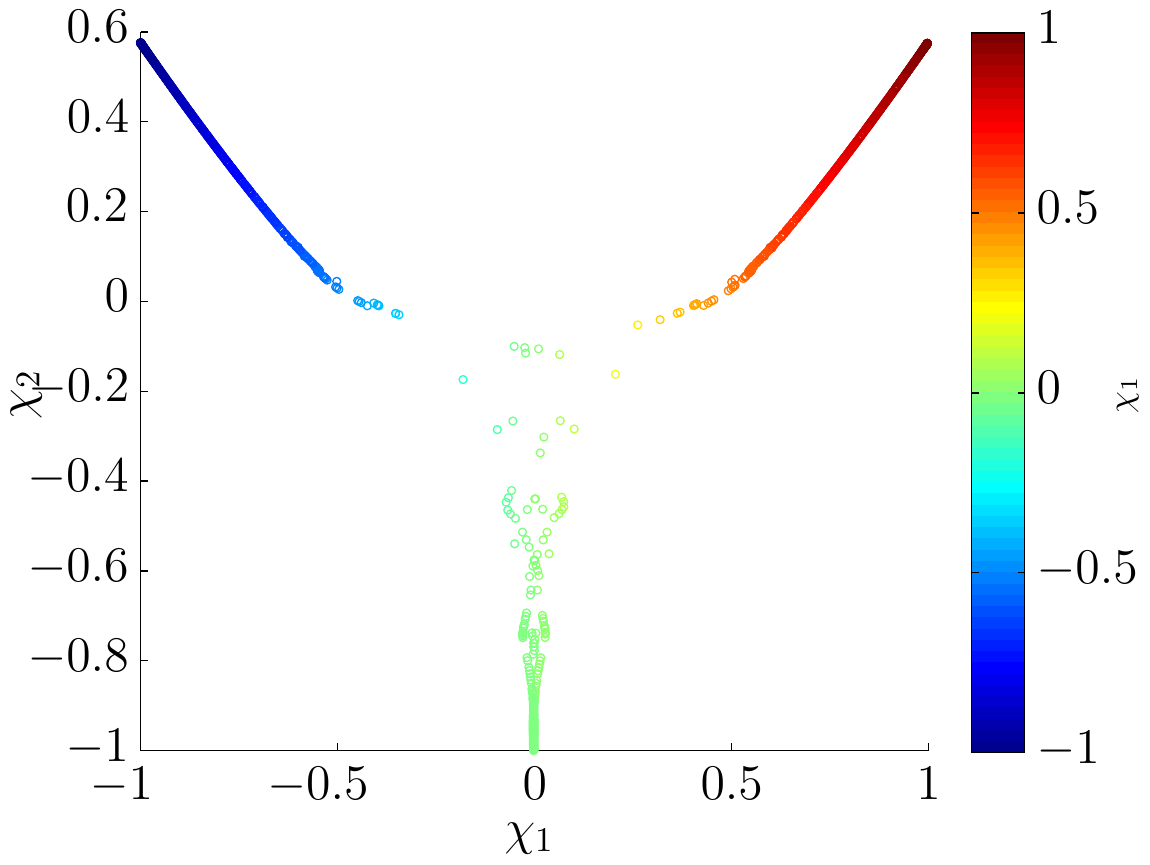}
    \caption{Embedding of $\xi$ for a double-well potential into first two diffusion coordinates. Colors are values of $\chi_1$ for easier comparison with \protect\subref{fig:duffing-ss}. (Cf. Fig.~\ref{fig:reeb}, center)}\label{fig:duffing-eq}
  \end{subfigure}\\
  \begin{subfigure}[t]{.45\linewidth}
\framebox[\linewidth][c]{ \raisebox{50mm}{\footnotesize Waiting for permissions. See original publication.}}
\caption{Two vortices in the Poincar\'e section of the periodic 3D Hill's vortex flow extracted based on continuous segments in $\xi$.}\label{fig:hill-ss}
  \end{subfigure}
  \begin{subfigure}[t]{.45\linewidth}
\framebox[\linewidth][c]{ \raisebox{50mm}{\footnotesize Waiting for permissions. See original publication.}}
  \caption{Embedding of $\xi$ for the periodic 3D Hill's vortex flow into first two diffusion coordinates.}\label{fig:hill-eq}    
  \end{subfigure}
  \caption[Ergodic quotients for double-well and 3d Hill oscillators. \citet{Budisic:2012woa}]{Geometric analysis of the ergodic quotient for a planar Hamiltonian system with Hamiltonian \eqref{eq:duffing} and a periodically forced 3d fluid-like flow defined by \eqref{eq:hill}.
(Original panels (\protect\subref{fig:hill-ss}) and (\protect\subref{fig:hill-eq}) in \citet{Budisic:2012woa}, Physica D: Nonlinear phenomena by North-Holland. Reproduced with permission of North-Holland in the format reuse in a journal/magazine via Copyright Clearance Center.)
}\label{fig:erg-quot-geom}
\end{figure*}

We use the double-well potential to illustrate this analysis, with results presented in Figure \ref{fig:duffing-eq}. The Hamiltonian function for this system, which serves as the basis for the well-known Duffing oscillator, is given by
\begin{align}
  H(q,p) = \frac{1}{2} p^2 - k\left ( \frac{1}{2} q^2 - \frac{1}{4}b
    q^4 \right);\label{eq:duffing}
\end{align}
in this analysis we chose $k = 1$, $b=2$. From Figure \ref{fig:duffing-eq} it is evident that the computations using metric $\norm{.}_{2,-s}$ retain the desired intuition established by Figure \ref{fig:reeb}: the diffusion coordinate \(\chi_{2}\) approximates the energy function of the system, while coordinate \(\chi_{1}\) discriminates between wells of the potential. The gaps in the numerical result are due to finite number of initial conditions $x$ used to evaluate the ergodic quotient map, and irregularities are due to the finite averaging process used to evaluate the projection \(P_{1}\); the Diffusion Maps algorithm can be adaptively tuned to tolerate such errors. When the state space is visualized using pseudo-coloring based on diffusion coordinates, the regions of uniform dynamical behavior are clearly distinguished.

The presented process goes beyond Hamiltonian flows: as an illustration we use a periodically forced 3d fluid flow based on the classical Hill's vortex flow.\cite{Budisic:2012woa} It is a solution to an ODE system on $x = (R,z,\theta) \subset \R^+ \times \R \times \T$:
\begin{align}\label{eq:hill}
  \begin{bmatrix}
    \dot R \\ \dot z \\ \dot \theta
  \end{bmatrix} = \underbrace{
    \begin{bmatrix}
      2Rz \\ 1 - 4R - z^2 \\ c/2R
    \end{bmatrix}
    + \epsilon
    \begin{bmatrix}
      \sqrt{2R} \sin \theta \\
      z(\sqrt{2R})^{-1} \sin\theta \\
      2\cos \theta
    \end{bmatrix}
    \sin 2\pi t}_{A(x,t)},
\end{align}
where parameters  $c$ and $\epsilon$ are swirl and perturbation strengths, respectively. Figure \ref{fig:hill-ss} shows two vortices that exist at $c = \epsilon = 0.3495$ colored based on the colors assigned to components of 
ergodic quotient shown in Figure \ref{fig:hill-eq}.

These two examples demonstrate that functions $\chi_k \in \mathcal F$ are invariant eigenfunctions for the Koopman operator, but whose level sets resemble energy functions. In this sense, instead of finding a function that would provide a ``good'' labeling of ergodic sets, as we did with energy functions in motivational Hamiltonian examples, we constructed a set of such a function from data generated by analyzing a basis for the invariant eigenspace of the Koopman operator. The entire process is made computationally feasible by truncating the set \(\chi_{k}\), therefore discarding the fine-scale features in favor of computability, while retaining coarse-grained features in an organized manner.

%%% Local Variables: 
%%% mode: latex
%%% TeX-master: "../koopmanism"
%%% End: 
\subsection{Infinite-time averages as projections onto eigenspaces}
\label{sec:averaging}

In our presentation of analyses of the state space using eigenfunctions of the Koopman operator, we have assumed we can compute eigenfunctions \(\phi^{(\lambda)}\) by projecting an observable \(f\) onto the eigenspace at \(\lambda\) using the projection operator \(P_{\lambda}\). When \(\lambda = e^{i2\pi \omega}\), i.e., for eigenvalues on the unit circle, we can evaluate the associated projection operator \(P_{\lambda}\) using infinite-time averages. For this reason, we restrict ourselves to the eigenspaces \(E_{\lambda}\) for which \(\abs{\lambda} = 1\) in this section.

The focus of our interest will be the averages
\begin{align}
  \frac{1}{N} \sum_{n=0}^{N-1} e^{-i2\pi \omega n} f \circ T^n(x), \label{eq:finite-Fourier}
\end{align}
which we would want to extend pointwise into the limit as $N \to \infty$. To describe the set of observables $\mathcal F$ which have well defined limits, we start by assuming that the map $T$ conserves a measure $\mu$ on the Borel algebra $\mathcal A$ in $M$. For \(\mu\)-integrable observables \(f\) the finite averages of the form \eqref{eq:finite-Fourier} can be extended into $N\to\infty$ to result in well defined \emph{Fourier averages} $\favg{f}$:
\begin{thm}[Wiener, Wintner]\label{thm:wienerwintner}
  Let $T : M \to M$ preserve a measure $\mu$ on a measurable space $M$. Then the set of points $\Sigma(f) \subset M$ on which the limit 
  \begin{align}
    \favg{f}(x) := \lim_{N\to\infty} \frac{1}{N}
    \sum_{n=0}^{N-1} e^{-i2\pi \omega n} f \circ T^n(x)\label{eq:fourieravg}
  \end{align}
is well defined can be chosen for all $f \in L^1(M, \mu)$ and $\omega \in \R$, independently of $\omega
$, and such that $\mu(M/ \Sigma) = 0$.
\end{thm}
\begin{proof}
  The original proof by \citet{Wiener:1941wy} was found to contain an error. Several correct proofs have been compiled by \citet{Assani:2003us}. 
\end{proof}
Moreover, when \(L^1(M, \mu)\) contains a dense countable set, the convergence set \(\Sigma\) can be chosen independently of \(f\).\cite{Mezic:2004is} When frequency \(\omega = 0\) is chosen, Fourier averages \eqref{eq:fourieravg} are referred to as \emph{ergodic averages}:
  \begin{equation}
    \tilde{f}(x) := \lim_{N\to\infty} \frac{1}{N}
    \sum_{n=0}^{N-1} f \circ T^n(x). \label{eq:ergodicavg}
  \end{equation}
As mentioned before, the concept of eigenmeasures \(\mu_{p}^{(\omega)}\) for  map \(T\) provides means for evaluating the projection operator, to which we now add a trajectory-wise formulation: 
\[
P_{\lambda} f(x) = \favg f(x) = \int_{M} f d\mu_{x}^{(\omega)}.
\]
A recent monograph by \citet{Wichtrey:2010ta} provides a detailed analysis of existence of Fourier averages and their applications in linear, nonlinear, and control systems.

While this presentation would perhaps suffice in the abstract, in practical settings we should take additional interest in:
\begin{inparaenum}[(a)]
\item the influence of choice observable \(f\) on the information contained in eigenfunctions \(P_{\lambda}f\),
\item the ``size'' of the convergence set \(\Sigma\) on which the averages converge, and
\item the rate or error in approximating \(\favg f\)  using finite averages.
\end{inparaenum}

The influence of $\omega$ on the value of $\favg f(x)$ is easily seen by fixing $f$ and $x$ and studying the sequence of (complex) numbers $a_n := f \circ T^n(x)$. The average
\[
\sum_{n=0}^{N-1} e^{-i2\pi\omega n}  a_n
\]
is then just a Discrete Fourier Transform (DFT) of the sequence $a_n$. 

From the basic knowledge of Fourier analysis, the values $\favg f$ will be non-zero only for those $\omega$ that correspond to temporal modes present in the sequence $a_n$, which represent the trace of the observable $f \circ T^n(x)$. Whether $a_n$ contains a mode at $\omega$ depends not only on trajectory $T^n(x)$, but also on the choice of the observable. This is easiest to see when $f$ is chosen as a constant function (if constant functions are in $\mathcal F$). In that case,  $\favg f \equiv 0$  for all non-zero $\omega$, and $\favg[0]{f} \equiv f$, regardless of the underlying dynamics $T$.

\begin{rem}Several concepts in this paper are associated with Joseph Fourier's name. This is understandable, as his ideas are at the root of all of them, but unfortunate as it might imply direct connections where there are none.

Several recent works\cite{Mezic:2004is, Wichtrey:2010ta} use the name \emph{harmonic averages} for \(\favg{f}\), in reference to harmonic analysis of dynamical flows. To avoid confusion with harmonic means of numbers,  we decided to use \emph{Fourier averages} instead, as do Wiener, Wintner, and Assani. This name is connected with the \emph{temporal} Fourier transforms of sequences, as explained above, since for fixed \(x \in M\) and \(f \in \mathcal F\), the function  \(\omega \mapsto \favg{f}(x)\) is the Discrete Fourier Transform of the sequence \(n \mapsto f\circ T^{n}(x)\).

Only through ergodicity, i.e., equality of space and time averages \(\int_{M} f\,d\mu = \tilde f\), do we obtain a connection with \emph{spatial} Fourier transform. Choosing \(f_{k}(x)\) as harmonic functions \(e^{i2\pi k \cdot x}\), where \(k\) is now a wavenumber, we can interpret any sequence \(k \mapsto \tilde f_{k}(x)\) in the ergodic quotient \(\xi\) as the spatial Fourier transform of the averaging ergodic measure \(\mu_{x}\). We refer to the image set of the spatial Fourier transform as the \emph{spatial Fourier coefficients}.

The eigenquotient maps, i.e., \(x \mapsto (\dots, P_{\lambda} f(x), \dots)\), could be justifiably named Fourier quotient maps, in analogy to ergodic/Fourier average dichotomy. Despite our earlier choices of terminology,\cite{Budisic:2012woa} we expect this connection to hold only for eigenvalues \(\lambda = e^{i2pi\omega}\), when \(P_{\lambda} f_{k}(x) = \favg{f}(x)\)  As one could conceivably formulate an eigenquotient map for \(\abs{\lambda} \not = 1\), in this paper we chose to use the term eigenquotient instead.

Finally, the harmonic functions \(f_{k}(x) = e^{i2\pi k \cdot x}\), \(k \in \Z^{d}\), i.e., solutions to Laplace equation \(\Delta f = 0\) on a torus \(\T^{d}\), are sometimes called \emph{Fourier functions} or Fourier harmonics. We use these functions as observables, i.e., functions that are acted on by the Koopman operator, to facilitate evaluation of Sobolev norms. Out of all, this connection in name is the least significant of all presented here, nevertheless, we mention it in this remark to clarify our terminology.
\end{rem}

\begin{figure*}[htb]
  \centering
  \begin{subfigure}[t]{.3\linewidth}\centering
    \includegraphics[height=42mm]{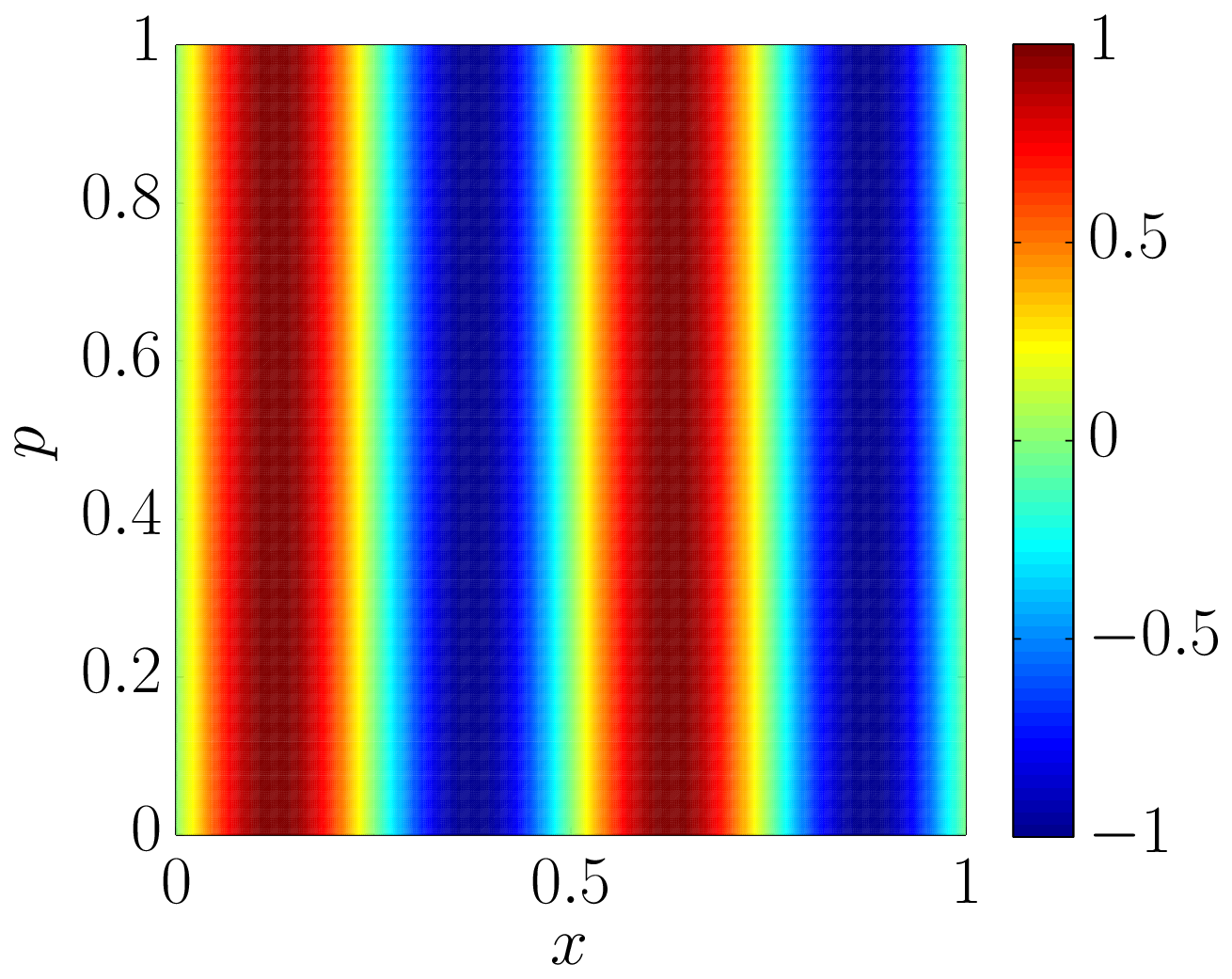} 
    \caption{Level sets of $f_1(x,p) = \sin(4 \pi x )$}\label{fig:favg_obs_f1_f}
  \end{subfigure}
  \begin{subfigure}[t]{.3\linewidth}\centering
    \includegraphics[height=42mm]{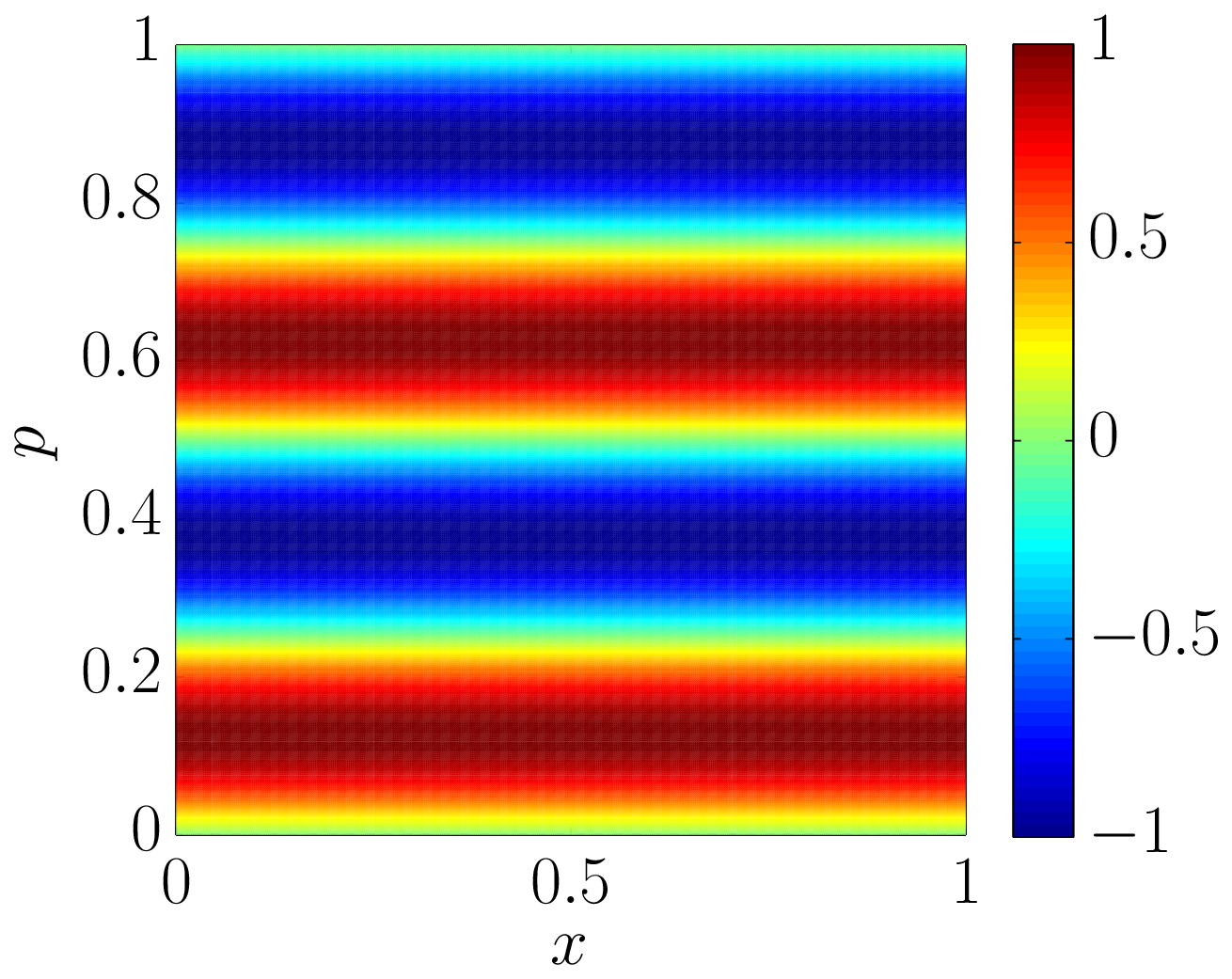} 
    \caption{Level sets of $f_2(x,p) = \sin(4 \pi p )$}\label{fig:favg_obs_f2_f}
  \end{subfigure}
  \begin{subfigure}[t]{.3\linewidth}\centering
    \includegraphics[height=42mm]{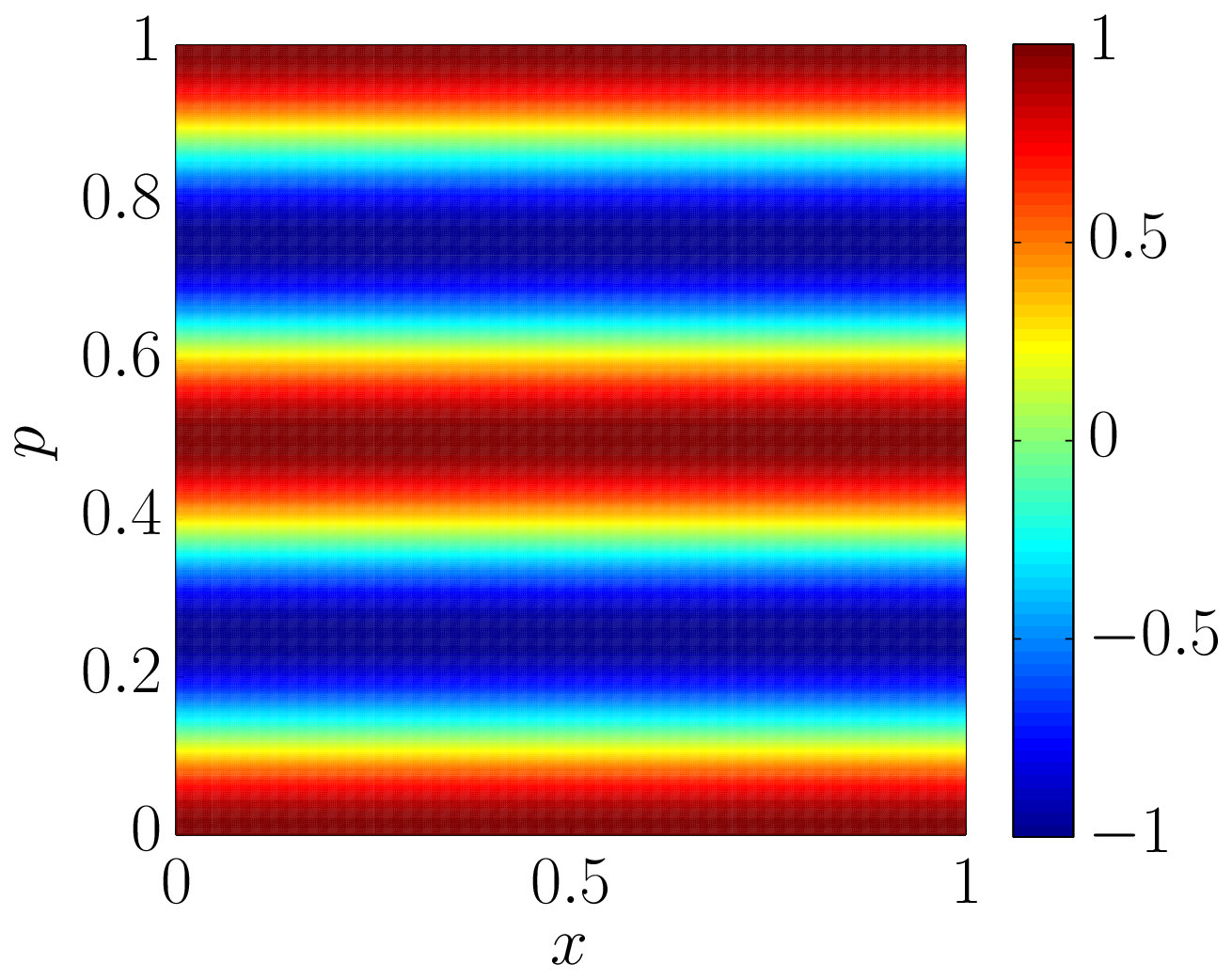}
    \caption{Level sets of $f_3(x,p) = \cos(4 \pi p )$}\label{fig:favg_obs_f3_f}
  \end{subfigure}\\
  \begin{subfigure}[t]{.3\linewidth}\centering
    \includegraphics[height=42mm]{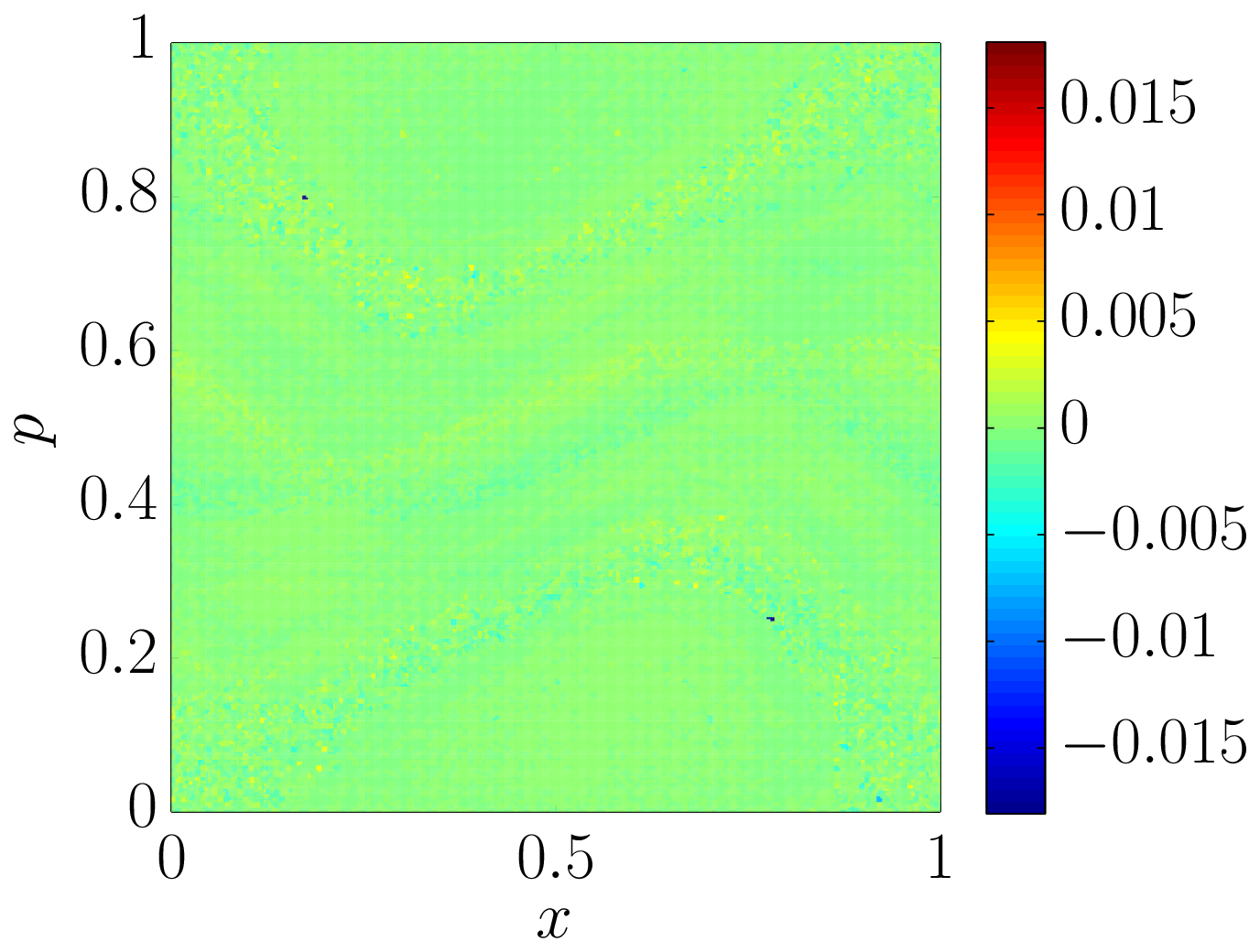} 
    \caption{Level sets of $\tilde f_1(x,p)$ }\label{fig:favg_obs_f1_avg}
  \end{subfigure}
  \begin{subfigure}[t]{.3\linewidth}\centering
    \includegraphics[height=42mm]{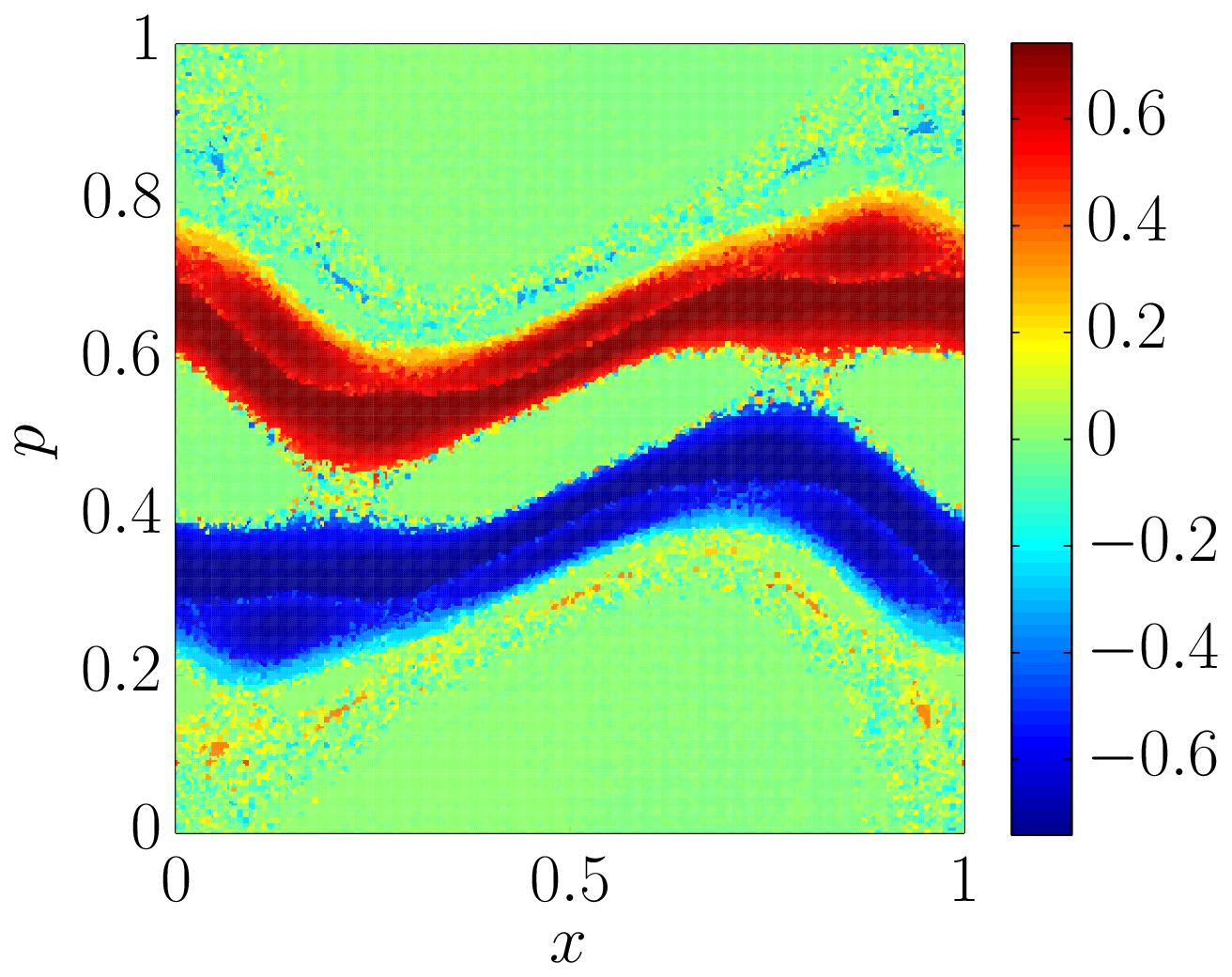} 
    \caption{Level sets of $\tilde f_2(x,p)$ }\label{fig:favg_obs_f2_avg}
  \end{subfigure}
  \begin{subfigure}[t]{.3\linewidth}\centering
    \includegraphics[height=42mm]{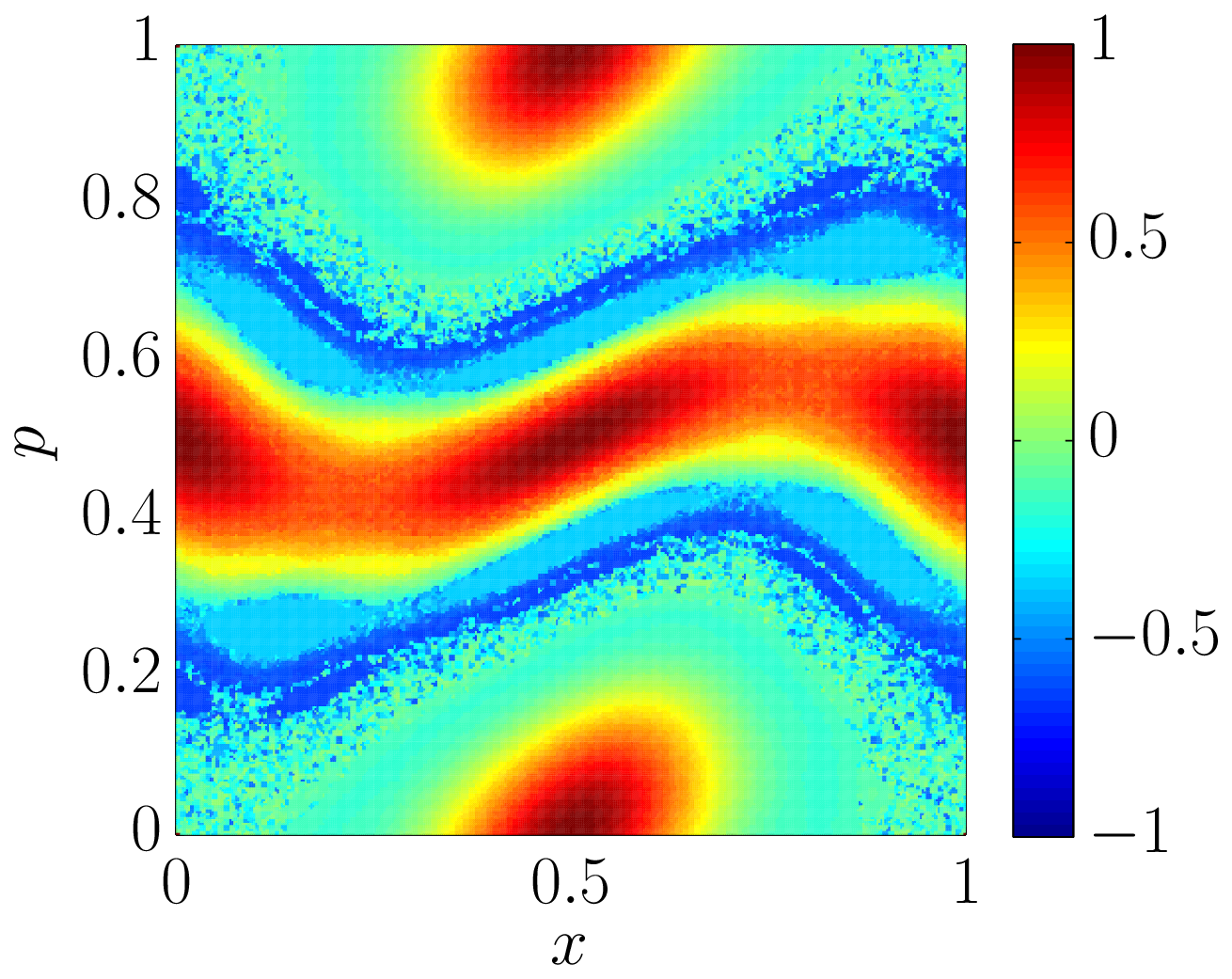} 
  \caption{Level sets of $\tilde f_3(x,p)$ }\label{fig:favg_obs_f3_avg}
  \end{subfigure}
  \caption[Observables before and after averaging.]{Three different observables \(f_{j}\)  (top row) and associated invariant functions \(\tilde f_{j}\) (bottom row), obtained by ergodic averages along trajectories of the Chirikov Standard Map \eqref{eq:std-map} for $\epsilon=0.15$. The pseudocolor is the value of each function.}
  \label{fig:favg-obs-comp}
\end{figure*}

\begin{rem}{}
  Starting from the same observable \(f\), two different eigenfunctions at \(\omega = 0\) can be computed as \(\phi(x)  := \favg[0]f(x)\) and for any \(\omega \not = 0\), \(\psi := \abs{\favg f(x)} \). An example can be found in Figure \ref{fig:stdmap-efunctions} where the first two images were computed by Fourier averages of $f(x,p) = \sin(\pi x - \pi/4) \cos(6\pi p)$ at \(\omega = 0\) and at \(\omega=1/3\), after which the modulus was taken. Both observables are invariant with respect to dynamics.

For a generic observable \(f\), the main difference is that \(\abs{\favg f}(x)\) is sensitive to \(\omega\)-periodic dynamics: even though the information about the phase of \(\omega\)-periodic sets is removed by taking the modulus, \(\abs{\favg f}(x)\) will be zero on all points that are not in any of \(\omega\)-resonant chains, which can be easily explained by DFT interpretation. There is no such restriction on \(\favg[0] f\), whose value can vary between different invariant sets (see Fig.~\ref{fig:stdmap-invariant}). Conversely,  \(\abs{\favg[1/3] f}\)  is zero everywhere except on the period-\(3\) island, where trajectories contain  \(\omega=1/3\) frequency components (Fig.~\ref{fig:stdmap-modulus}).
\end{rem}

The $\omega$ values corresponding to the $e^{i2\pi\omega}$ that are eigenvalues of the Koopman operator will result in non-zero eigenfunction \(f^{(\omega)}\) for at least some observable \(f\) . The $\omega$ values that are not frequencies of any of the eigenvalues will result in $\favg f \equiv 0$. The frequency $\omega = 0$, however, corresponds to eigenvalue $\lambda=1$ of $U$, which is always in the spectrum of the Koopman operator. This frequency is of a particular interest, as the eigenfunctions $\favg[0]f \equiv \tilde f$ are invariant functions for the system.

To illustrate the difference in eigenfunctions obtained by starting from different observables, we again used the Chirikov Standard Map \eqref{eq:std-map} at \(\epsilon=0.15\) . The averages were computed for $\omega = 0$, which projects observables to the invariant eigenspace of $U$. As Figure \ref{fig:favg-obs-comp} shows, the detail that level sets of averaged observables reveal about the state space is highly dependent on the starting observable, however, when a ``good'' observable is chosen, it can reveal a lot about the state space. The lack of intuition about how to select such a ``good'' starting observable led to development of the ergodic quotient analysis, presented in Section \ref{sec:quotient}. The diffusion coordinates \(\chi_{k}\) used therein can be interpreted as constructed observables that reveal detailed information about the state space.

The \emph{convergence set} \(\Sigma\) was established to be of full measure \(\mu\) by  Theorem \ref{thm:wienerwintner}. The measure \(\mu\) is a measure preserved by the system, which is also used to define the space of integrable observables \(L^{1}(M, \mu)\) for which the Wiener-Wintner theorem holds. Depending on the analyzed system, the existence of an invariant measure can be inferred through different arguments. For example, a system \(\dot x = F(x)\) governed by divergence-free vector field \(\nabla \cdot F \equiv 0\) conserves the volume-measure on the state space. In dissipative dynamics, the systems may conserve Sinai-Ruelle-Bowen measures on chaotic sets.\cite{Young:2002vc} Finally, on a compact state space \(M\) the only necessary assumption for existence of an invariant \(\mu\) is continuity of \(T\), by the theorem of Krylov-Bogolyubov.\cite{Katok:1995th}

The Wiener-Wintner theorem holds for \emph{any} invariant measure \(\mu\). Therefore, by choosing the measure used to formulate the space of observables \(\mathcal F = L^{1}(M, \mu)\), we influence the amount of information we can collect about the system by evaluating Fourier averages and the size of the set \(\Sigma\). In applied contexts, we would often want to include as many open sets as possible in the support of measure \(\mu\) chosen: for volume-preserving systems, volume of the state space is often a good choice. Caution is still needed, as there could be a \(\mu\)\nobreakdash-zero, yet dense, non-convergence set $\Sigma^{c}$. Nevertheless, in our experience, simulations of dynamical systems that model physical phenomena do not contain such extreme pathological cases.

\begin{figure}[htb]
  \centering
\includegraphics[width=.25\textwidth]{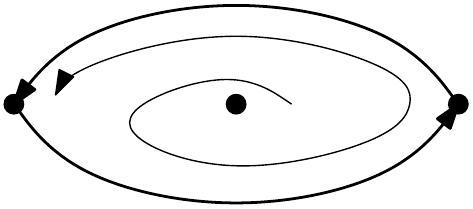}
  \caption[An attracting heteroclinic cycle.]{Attracting heteroclinic cycles may prevent convergence of ergodic averages.}
  \label{fig:attr-heterocl}
\end{figure}

For dissipative systems, one often requires a that the set of Birkhoff-regular initial conditions\cite{DAlessandro:2002vy}, for which Fourier averages of continuous observables  converge, is of positive volume, i.e., that the system preserves a \emph{physical measure}\cite{Young:2002vc} supported on the attractor. In those cases, computing quotient maps using Fourier averages will identify basins of attractions of the attractors, instead of attractors themselves, since the points in the basin will have the same averages as the points on the attractor.

A well-known example\cite{Gaunersdorfer:1992cg} of a dissipative system for which trajectory averages do not converge for almost all open sets of initial conditions contains an attractive heteroclinic cycle (see Fig.~\ref{fig:attr-heterocl}). Ergodic measures are \(\delta\)\nobreakdash-measures supported on each of the fixed points, consequently, any invariant measure will be singular with respect to Lebesgue. Trajectories that approach heteroclinics spend longer and longer times along each of the exterior fixed points, possibly in such a manner that the finite time averages do not converge. Notice that, while the conditions of Theorem \ref{thm:wienerwintner} are formally satisfied, they are almost vacuous for practical purposes, as only the subsets of heteroclinic orbits have well-defined averages, yet the convergence set is still of full measure, due to singular nature of the measure conserved. 

From the authors' experience in applying the averaging technique to dynamical systems that model physical phenomena for practical analysis, finding an appropriate convergence set \(\Sigma\) is rarely a problem. A more practically significant issue comes from the different rates of convergence of finite to infinite averages for points \emph{within} \(\Sigma\).

\begin{figure*}[thb]
  \centering
  \begin{subfigure}[t]{.48\linewidth}\centering
    \includegraphics[height=50mm]{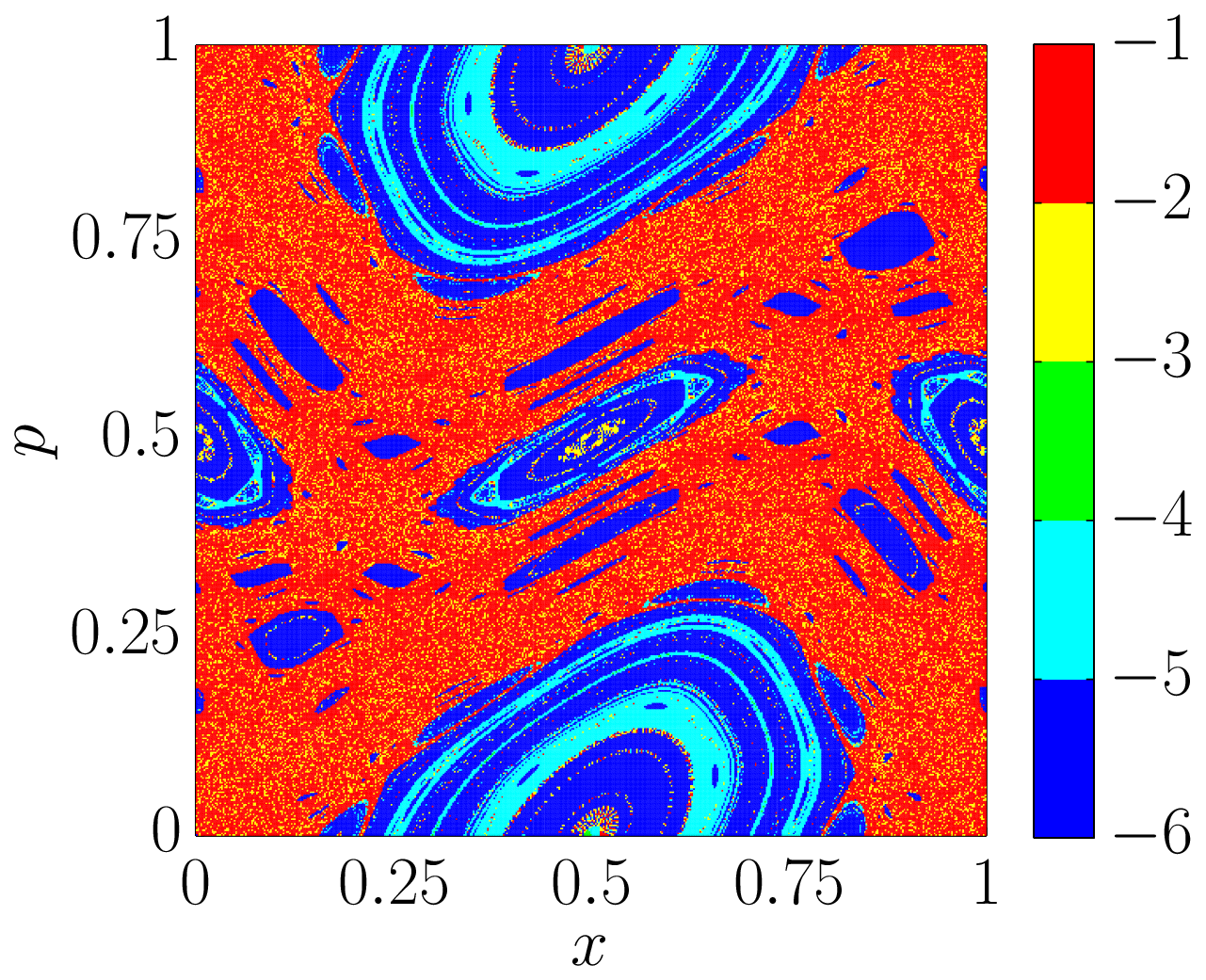}
    \caption{Difference in averages after \(N=10^{5}\) and \(N_{\infty} =  10^{6}\) iterates (\(\log_{10}\)-scale). The bottom of the color scale includes all the values below \(-6\).}
\label{fig:stdmap-conv-error}
  \end{subfigure}
  \begin{subfigure}[t]{.48\linewidth}\centering
    \includegraphics[height=50mm]{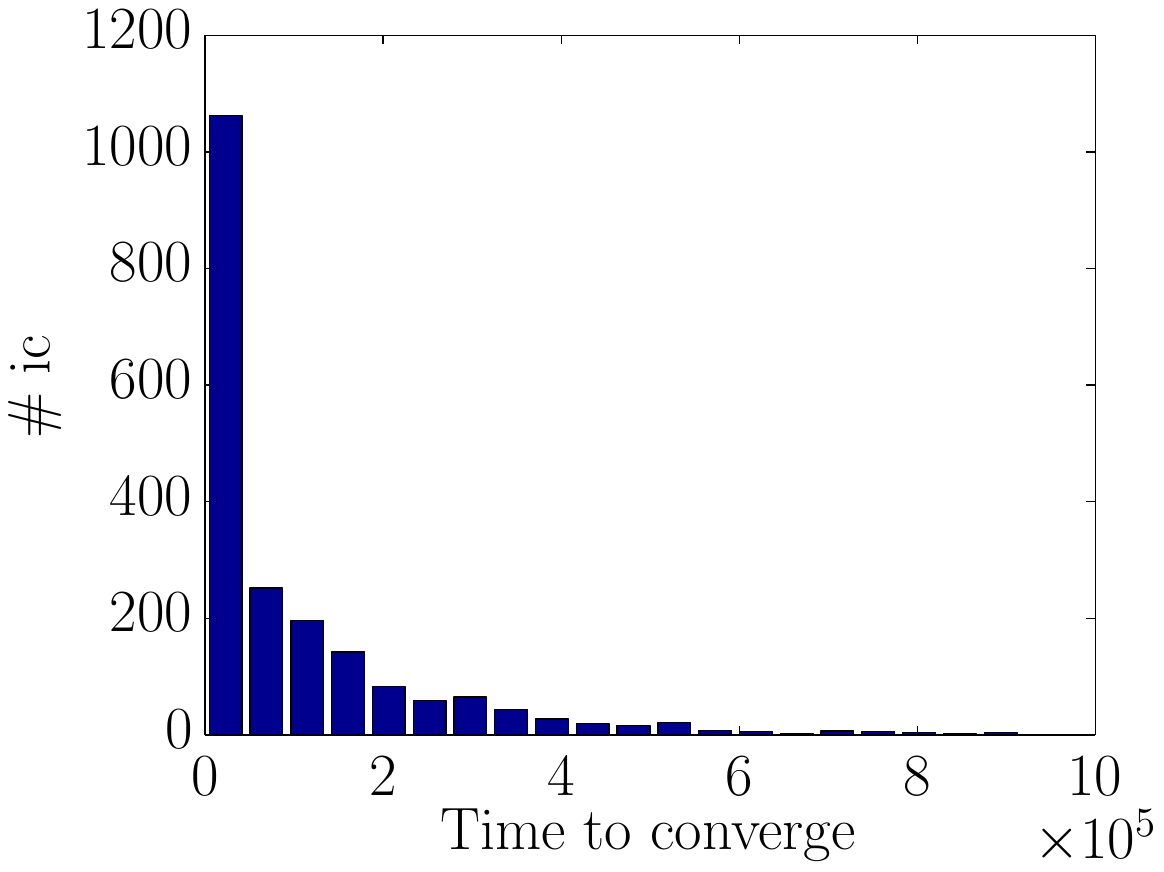}
    \caption{Distribution of number of initial conditions over iterations required for the difference in averages compared to \(N_{\infty} = 10^{6}\) iterates to reach \(10^{-3}\).}
\label{fig:stdmap-conv-times-hist}
  \end{subfigure}\\
  \begin{subfigure}[t]{.48\linewidth}\centering
    \includegraphics[height=50mm]{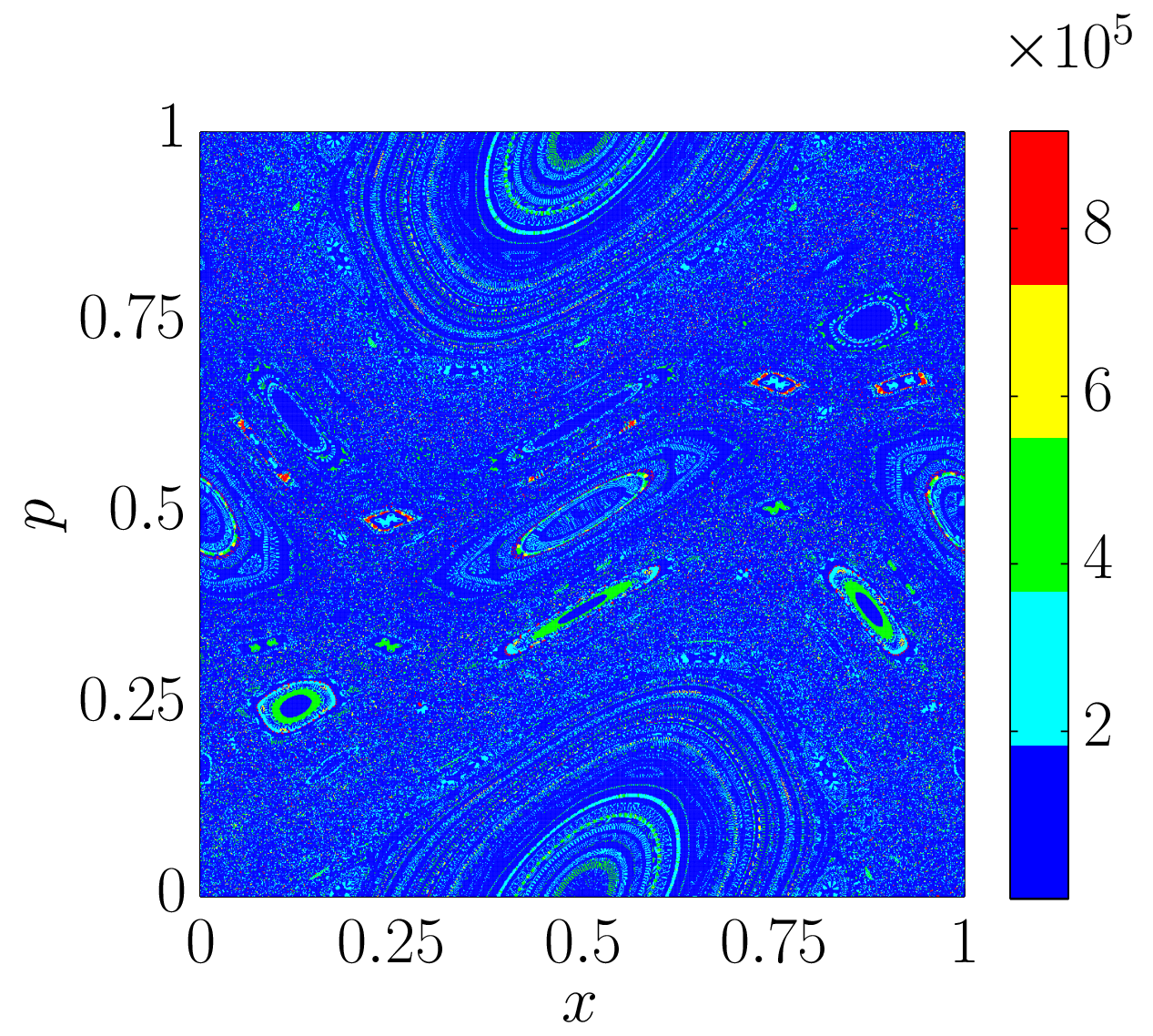}
    \caption{Number of steps required for the difference in averages compared to \(10^{6}\) iterates to reach \(10^{-3}\).}
\label{fig:stdmap-conv-times}
  \end{subfigure}
  \begin{subfigure}[t]{.48\linewidth}\centering
    \includegraphics[height=50mm]{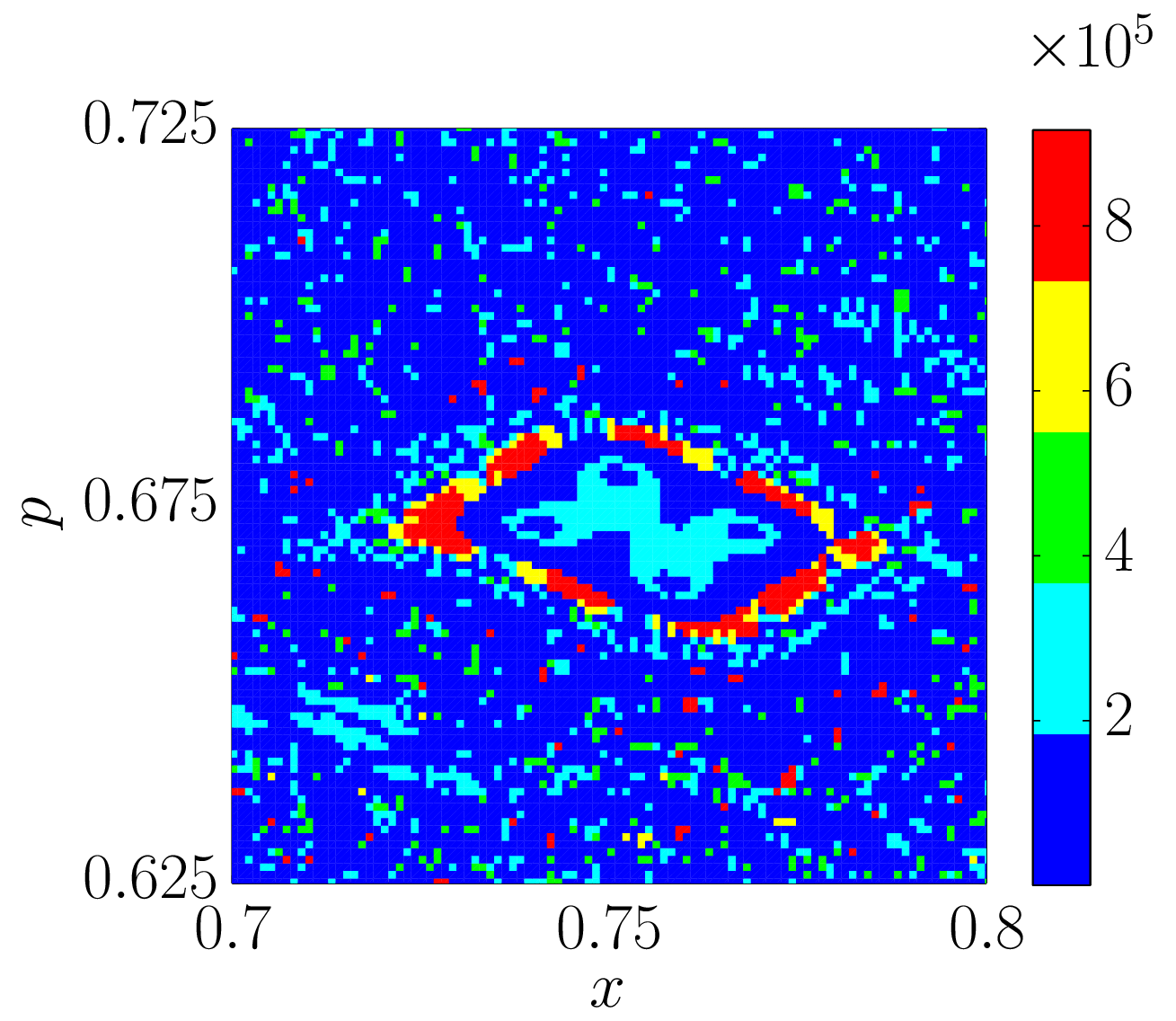}
    \caption{Number of steps required for the difference in averages compared to \(10^{6}\) iterates to reach \(10^{-3}\) (enlarged).}
\label{fig:stdmap-conv-times-zoomed}
  \end{subfigure}
  \caption[Convergence of ergodic averages.]{Convergence of ergodic averages. The difference in averages of Fourier harmonics for trajectory starting at \(x \in M\) is computed as \(\max_{k}\abs{A_{N} f_{k}(x) - A_{N_{\infty}}f_{k}(x)}\) over \(k \in  [-5,5]^{2} \subset \Z^{2}\) . Step \(N_{\infty} = 10^{6}\) was taken as ``true'' infinity. A total of \(2025\) initial conditions, seeded from a uniform rectangular grid, was used. Colors were interpolated from the first \(500\) iterates of each trajectory.} \label{fig:stdmap-conv}
\end{figure*}

The rate of convergence of finite limits 
\[
A_{N} f(x) := \frac{1}{N} \sum_{n=0}^{N-1} f \circ T^n(x), 
\]
and their Fourier counterparts \eqref{eq:finite-Fourier} on the set $\Sigma$ is not uniform. It is a classical result that the rates of convergence over periods of length $N$ can range from exponentially fast, i.e., $\mathcal O( e^{-\lambda N})$ for $\lambda > 0$, for trajectories approaching exponentially stable fixed points, to algebraic convergence $N^{-\alpha}$, where $\alpha > 0$ is arbitrarily small, see, e.g., \citet[, \S 3.2B][]{Petersen:1989uv}. From a practical perspective, in a lot of cases the situation does not look so bleak. Trajectories on periodic orbits and in strongly mixing regions achieve the rates of convergence of $\mathcal O(N^{-1})$ and $\mathcal O(N^{-1/2})$, respectively. The slopes $0 < \alpha < 1/2$ are to be expected near homoclinic and heteroclinic orbits, especially if such orbits are embedded within zones of intermittency, where trajectories get entrained around marginally stable fixed points for long times, before eventually moving away from them.\cite{ChaosBook} Such zones appear, for example, in perturbed hamiltonian and volume-preserving systems.\cite{Perry:1994wl} Studies of volume-preserving systems, in theoretical\cite{Treschev:1998uq} and computational\cite{Budisic:2012woa} contexts, have shown that such regions are small in area.

To illustrate, we plotted convergence errors for the Chirikov Standard Map \eqref{eq:std-map}, simulated for \(\epsilon=0.18\) where the state space contains both regular and mixing regions. This brief analysis is similar to those in literature.\cite{Levnajic:2010gq,Levnajic:2008vo,Budisic:2012woa}
As an indicator of convergence speed, we compared the ergodic average, i.e., Fourier average with \(\omega = 0\), after \(N_{\infty} = 10^{6}\) iterates with the average after (significantly) shorter number \(N \ll N_{\infty}\) of iterates. To avoid making conclusions based on a single observable, a truncated set of Fourier harmonics 
\[
f_{k}(x,p) = \frac{1}{2\pi}\exp\left[i 2\pi (k_{x} x + k_{p} p)\right],
\] for \(k = (k_{x},k_{p}) \in [-5,5]^{2}\subset \Z^{2}\), was used; this is the type of set used to practically approximate the quotient maps described in Section \ref{sec:quotient} and for each trajectory the largest absolute difference in averages over that set was taken as an indication of the convergence error. 

Figure \ref{fig:stdmap-conv-error} shows the error in averages after a fixed time \(N = 10^{5}\) iterates, while Figures \ref{fig:stdmap-conv-times-hist}, \ref{fig:stdmap-conv-times}, and \ref{fig:stdmap-conv-times-zoomed} show the time required for the error to stay within \(10^{-3}\) for \(100\) iterates. It is clear that the speed of convergence varies across the state space, with convergence in chaotic regions slower than in regular regions. We proposed\cite{Budisic:2012woa} that the simulated time is varied for individual trajectory based on the relative convergence error, which resulted in efficient simulation runs with only a small subset of initial conditions, in regions of intermittency, requiring long simulation times.

This section dealt with the problem of extending finite-time averages into the infinite limit, to be able to approximate the limiting measures of empirical distributions. It might be surprising that even explicitly finite time averages find their use in practical applications. In the next section, we demonstrate how to formulate continuous indicators of mixing and ergodicity, and use them in design of feedback control for technical systems.

%%% Local Variables: 
%%% mode: latex
%%% TeX-master: "../koopmanism"
%%% End: 

\section{Continuous indicators of ergodicity and mixing}
\label{sec:ctsinds}

When dynamical systems are analyzed as measure-preserving transformations, \emph{ergodicity} and \emph{mixing property} are among the first concepts discussed in introductory textbooks.\cite{Mane:1987wz,Lasota:1994vt} Let measure \(\mu\) be preserved by the system. In plain language, ergodicity with respect to \(\mu\)  means that sets left invariant by the transformation/flow are either full measure \(\mu\), or \(\mu\)-negligible. The mixing property, which implies ergodicity, means that any set of positive  \(\mu\)-measure will be distributed by the flow according to the mixing measure \(\mu\).

While the standard introductions to these properties\cite{Petersen:1989uv} often include equivalence theorems that provide alternative formulations of ergodicity and mixing, cf. \eqref{eq:erg-space-time} and \eqref{eq:erg-measure}, all the definitions treat them as binary indicators: either the system is ergodic/mixing, or it is not. From the perspective of applied dynamical systems, especially in the context of design of dynamics, having a binary indicator of a desired property is insufficient. Designers prefer to work with continuous indicators, e.g., an ergodicity indicator that takes values in $[0,1]$ where the extremum $0$ would imply classical ergodicity, and the higher values would indicate how far, in some sense, the system is from being ergodic. The continuous indicators of ergodicity and mixing can be constructed from averages of functions along trajectories. \cite{Mathew:2011ev,Scott:2009wh} The averaging process corresponds to an invariant measure, which is then compared to the \emph{a priori} measure $\mu$ to infer how close the system is to being $\mu$-ergodic or $\mu$-mixing.

Designing systems to be ergodic or mixing has a number of practical applications. Consider the case of micro-mixers, devices whose task is to mix two fluids, reactants, such that when the reaction is initiated, there are no ``pockets'' or unused reactants, and the reaction occurs uniformly in the vessel. At usual macro-scales, the fluids are easily mixed just by shaking them; on micro- and nano- scales this is not possible. Instead, micro-mixers use dynamical, advective transport, to reach the mixed state. An example where ergodicity is desirable comes from search-and-rescue missions, where helicopters or airplanes are used to scan a large area for survivors of airplane crashes and capsized boats. It is not trivial to design a flight path that ensures that the entire area is covered and that particular zones are searched more often or more thoroughly: such task can be phrased as design of an ergodic flight path. Both of these problems can be addressed by casting them into a framework of dynamical systems and requiring that the trajectories are mixing or ergodic.

\begin{figure}[htb]
  \centering
  \includegraphics[width=.45\textwidth]{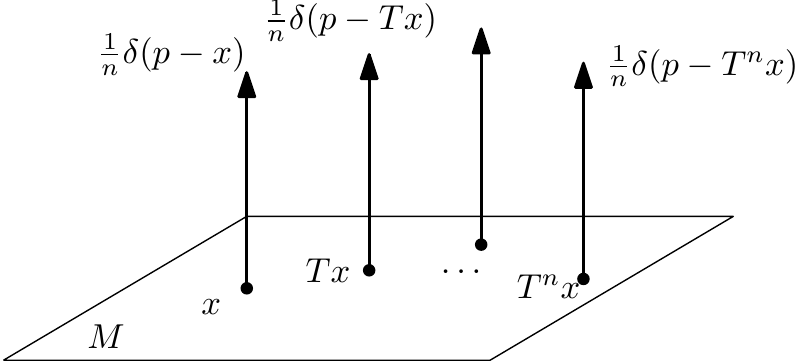}
  \caption[Empirical probability densities.]{For finite time $N$, averaging functionals can be represented by empirical probability measures, whose distributions $c_{x,N}$ can be thought of as sequences of $\delta$ distributions supported on trajectories.}\label{fig:emp-distr}
\end{figure}

While trajectory averages were treated as a technique for computation of projections onto eigenspaces of the Koopman operator, which required their extension into infinity, in this section we will show what can be extracted from finite-time averages \eqref{eq:finite-Fourier}, with \(\omega = 0\), i.e., the finite-time version of ergodic averages.  The average of a bounded observable $f:M \to \C$ over a finite-time trajectory segment, $f \mapsto \frac{1}{N}\sum_{n=0}^N f \circ T^n(x)$, is a linear continuous functional. By the Riesz representation theorem, the finite-average functionals are represented by \emph{empirical measures}, whose distributions, formally defined as 
\begin{align}
  c_{x,N}(p) := \frac{1}{N}\sum_{n=0}^{N-1} \delta[ p - T^n(x) ],
\label{eq:emp-measure}
\end{align}
can be thought of as strings (or ribbons) of \(\delta\)-distributions, supported along the orbits (see Fig.\ \ref{fig:emp-distr}). 

The consequence of defining the empirical distribution using the Riesz representation theorem is the well-known equality between spatial and time averages. With a bit of breadth in notation, the expression $\avg{f,g} = \int_{M} f(p) g(p) dp$ can be used to couple functions \(f\) and measures \(g(p)dp\). In the case of measures with density functions, i.e., absolutely continuous measures, both \(f\) and \(g\) can be taken as functions; when \(g(p)dp = d\mu(p)\) corresponds to a more general measure \(\mu\), \(f\) is taken as a test function, and \(g\) is in the class of distributions, or generalized functions, e.g., for a point mass \(\mu\), \(g\) is a Dirichlet \(\delta\) distribution. Coupling the empirical measure \eqref{eq:emp-measure} with a test function \(f\) yields
\begin{align*}
  \avg{c_{x,N},f} &= \int_{M} \left\{\frac{1}{N}\sum_{n=0}^{N-1} \delta[ p - T^n(x) ]\right\} f(p)dp \\
    &= \frac{1}{N}\sum_{n=0}^{N-1} \int_{M} \delta[ p - T^n(x) ] f(p) dp \\
    &= \frac{1}{N}\sum_{n=0}^{N-1} f[T^{n}(x)].
\end{align*}
If, as an observable, we take the characteristic function of a measurable set $A$, $r_{x,N}(A) = \avg{c_{x,N},\chi_A}$ is the residence time: the fraction of time $[0,N]$ that trajectory $T^n(x)$ spends inside set $A$. Therefore, the empirical measures $\avg{c_{x,N},\cdot}$ induced by averaging functionals are probability measures, for any initial condition $x \in M$ and $N \geq 0$. For any finite time, empirical distributions are not absolutely continuous with respect to the Lebesgue measure. In the limit $N \to \infty$, however, they converge in weak-$\ast$ topology to ergodic measures;\cite{Young:2002vc} their limits might be point masses, on equilibria, singular measures, e.g., on periodic orbits, or absolutely continuous measures, e.g., in mixing regions.

Ergodicity of the flow $T^n$ with respect to a flow-invariant measure $\mu$ can be stated as the condition that for $[\mu]$-a.e.\ $x \in M$, 
\begin{align}
  \lim_{N \to \infty} \avg{ c_{x,N}, \chi_A} = \mu(A),
\label{eq:erg-measure}
\end{align}
for all measurable sets $A$. The state-space points $x \in M$ whose trajectories satisfy the ergodicity condition are called $\mu$-generic. The first step towards a continuous indicator of ergodicity was presented by \citet{Scott:2009wh}, who used a Haar wavelet basis as observables: these wavelets directly relate to indicator functions of a basis for open sets on a rectangular state space. Since wavelet bases are parametrized by scale at which sets are resolved, the constructed ergodic indicator, termed \emph{ergodicity defect}, acted as a coarse-grained version of ergodicity, reflecting an engineering perspective: instruments have finite precisions, therefore, if the equalities \eqref{eq:erg-measure} are satisfied for all sets coarser than precision of the instrument, the system can be thought of as ergodic for the particular application.

Based on a similar idea, \citet{Mathew:2011ev} constructed a different continuous ergodicity indicator. For declaring ergodicity, it is sufficient to verify the condition \eqref{eq:erg-measure} on a basis for the Borel algebra, e.g., euclidean balls $B(x,r) = \{ p \in M : \norm{p - x}_2 < r\}$, whose indicator functions we label $\chi_{(p,r)}$. Integrating the deviations between trajectory averages and measures of sets, we obtain the empirical ergodicity $E_x(n)$ 
\begin{align}
  E_x(n)^{2} := \int_0^R \int_{M} \abs{ \avg{c_{x,n},\chi_{(p,r)}}
    - \mu[B(p,r)]}^2 dp\, dr,\label{eq:ergdef}
\end{align}
where $R$ is selected such that the largest ball includes the entire state space $M$ (cf. pseudodistance \eqref{eq:emp-pseudodistance}). It is almost immediate  that the ergodicity is equivalent to $\lim_{n \to \infty} E_x(n) = 0$ if $x$ is selected as a $\mu$-generic point. Strictly speaking, $E_x$ compares the \(N\to\infty\) limit of empirical measures to the prior $\mu$; to assure that the same behavior occurs almost everywhere, $E_x$ can be integrated along the state space. 

The practical value of $E_{x}(n)$ is in how its constituents are computed:
$\avg{c_{x,n},\chi_{(p,r)}}$ is computed as a time-average of an indicator function, which is very easily computed as the system is being simulated from an ODE. When $\mu$ is the volume measure, quantities $\mu[ B(x,r) ]$ are just volumes of euclidean balls, and are computed using well-known formulas. Consequently, evaluation of $E_{x}(n)$ is simple, as it requires only evaluation of finite-time averages. 

On the other hand, $\mu$ can be a more detailed measure, as in the mentioned case of probabilities of target detection. \cite{Mathew:2011ev,Mathew:2009et} In those cases, evaluating \eqref{eq:ergdef} might be a more complicated effort, requiring careful spatial gridding to control errors in the integral. Instead of measures of spherical sets, the metric $E_{x}(n)$ can be expressed using Fourier coefficients of $\mu$, facilitating the control of spatial scale resolution. The integrand in \eqref{eq:ergdef} can be interpreted as a difference between generalized expansion coefficients, where the basis set is the set of characteristic functions, whose supports are a base for Borel sets:
\begin{align*}
  \mu[B(p,r)] = \avg{ \partial \mu, \chi_{(p,r)} }.
\end{align*}
The \(\partial \mu\) stands for the (formal) density of the prior measure \(\mu\). 

If the basis $\chi_{(p,r)}$ is replaced by a harmonic basis $f_k$, then $\hat \mu(k) := \avg{\partial\mu, f_k}$ are just spatial Fourier coefficients, easily computed to very high orders by Fast Fourier Transform if we know the density \(\partial \mu\) explicitly. At the same time, practicality of evaluation of $\avg{c_{x,n},f_k}$ is not sacrificed. The resulting empirical ergodicity is given by a metric induced by the negative-index Sobolev norm $\norm{.}_{2,-s}$ on the space of distributions $W^{2,-s}$ (cf.\ \eqref{eq:neg-sob-norm}), if both $\partial\mu$ and $c_{x,n}$ are in it.
\begin{align}
  \label{eq:negsob}
  \begin{aligned}
    \norm{ c_{x,n} - \partial\mu }_{2,-s} \\
    &= \sum_{k \in \Z^D} \frac{\abs{ \avg{c_{x,n},f_k} - \avg{\partial\mu, f_k}   }^2}{ [1 + (2\pi\norm{k}_2)^2]^s},
  \end{aligned}
\end{align}
where the state space is, for simplicity, $M \simeq \T^D$, $k \in \Z^D$. Wavevectors $k \in \Z^D$ parametrize harmonic functions $f_k(x) = (2\pi)^{D/2}e^{i 2\pi k \cdot x }$. 

When the order of the Sobolev space $s$ is chosen as $s = (D+1)/2$, $\norm{ c_{x,n} - \partial\mu }_{2,-s}$ and $E_x(n)$ are equivalent\cite{Mathew:2011ev}: there exist constants $\alpha > 0$ and $\beta > 0$ such that 
\[
\alpha \norm{ c_{x,n} - \partial\mu }_{2,-s} \leq E_x(n) \leq \beta \norm{ c_{x,n} - \partial\mu }_{2,-s},
\] 
at all $n$. As a consequence, decay of $\norm{ c_{x,n} - \partial\mu }_{2,-s}$ can be used as the proxy for computing decay of  $E_{x}(n)$ to detect ergodicity, since  $\norm{ c_{x,n} - \partial\mu }_{2,-s}$ is easier to numerically evaluate for most measures $\mu$.

A practical application of $E_{x}(n)$ can be seen on a model search-and-rescue problem, where an Unmanned Aerial Vehicle (UAV), e.g., a small helicopter, would search an area containing a target whose position is estimated by a probability density \cite{Mathew:2011ev}. The search path for the vehicle can be planned using a dynamical system: the searcher is treated as a passive particle, possibly with a non-zero inertia, in a fluid-like flow. The probability density is a product of gaussians, modeling a priori estimates of the target, and discontinuous indicator functions, modeling foliage on the ground where the UAV has no visibility. The goal is to design the dynamics of the flow, such that the trajectory traced out by a particle avoids foliage and explores the feasible area according to probability of finding the target.

The devised control algorithm used $E_{x}(t)$ (in continuous-time setting) as the optimization function in a Hamilton-Jacobi-Bellman framework. Instead of optimizing for decay of $E_{x}(t)$ over a finite horizon $t \in [0,T]$, a greedy approach is chosen, where the optimization horizon is shrunk to the instance by taking $T \to 0$, resulting in a closed-form expression feedback law which relies on finite-time averages of harmonic observables, computed along the path that the searcher traveled. As a result, trajectories designed for a team of UAVs searched the area efficiently, with minimal crossings over the zones with high foliage (Fig.\ \ref{fig:smcsearch}), with further details presented in \cite{Mathew:2009et,Mathew:2011ev}.

\begin{figure*}[htb]
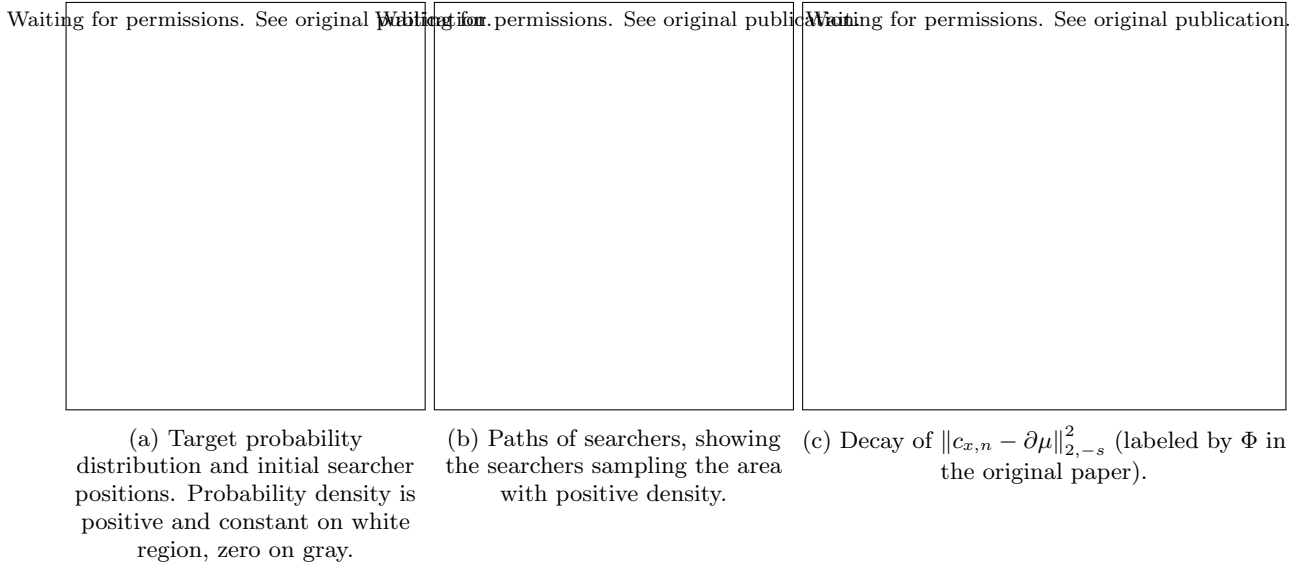

  \centering
  \begin{subfigure}[t]{.29\linewidth}\centering
\framebox[\linewidth][c]{ \raisebox{50mm}{\footnotesize Waiting for permissions. See original publication.}}
    \caption{Target probability distribution and initial searcher positions. Probability density is positive and constant on white region, zero on gray.}
  \end{subfigure}
  \begin{subfigure}[t]{.29\linewidth}\centering
\framebox[\linewidth][c]{ \raisebox{50mm}{\footnotesize Waiting for permissions. See original publication.}}
  \caption{Paths of searchers, showing the searchers sampling the area with positive density.}
  \end{subfigure}
  \begin{subfigure}[t]{.39\linewidth}\centering
\framebox[\linewidth][c]{ \raisebox{50mm}{\footnotesize Waiting for permissions. See original publication.}}
    \caption{Decay of $\norm{ c_{x,n} - \partial\mu }_{2,-s}^{2}$ (labeled by \(\Phi\) in the original paper).}
  \end{subfigure}
  \caption[UAV path planning by decay of empirical ergodicity. \citet{Mathew:2011ev}]{Searching for a target by optimizing the decay of empirical ergodicity \(E_{x}(n)\), through computation of \(\norm{ c_{x,n} - \partial\mu }_{2,-s}\) for several trajectories. (Original in \citet{Mathew:2011ev}, Physica D: Nonlinear phenomena by North-Holland. Reproduced with permission of North-Holland in the format reuse in a journal/magazine via Copyright Clearance Center.)\label{fig:smcsearch}}
\end{figure*}

In addition to quantifying ergodicity, the $W^{2,-s}$ norm can be used to quantify mixing with respect to a target density $\partial\mu$, in which context it was originally developed.\cite{Mathew:2005eq,Mathew:2007vq} In the final design of a  micro-mixer, an initial concentration of fluid reactants $\rho$ was advected by a dynamical system, whose effect was captured by the Perron-Frobenius operator $P^n \rho$. To quantify mixing, the metric $M(n) := \norm{P^n \rho - \partial\mu}_{2,-s}$ was computed with the order $s = 1/2$. By optimizing the dynamics of the fluid flow, a quick mixing of reactants was successfully achieved (Fig.~\ref{fig:sob-mixing}), both in numerical and experimental settings.

\begin{figure*}[htb]
  \centering
\framebox[\linewidth][c]{ \raisebox{50mm}{\footnotesize Waiting for permissions. See original publication.}}
  \caption[Mixing optimization. \citet{Mathew:2007vq}]{Simulation of mixing of reactants in a micro-chamber using dynamics designed by optimizing a $W^{-s,2}$ distance between advected density and uniform target density. (Original in \citet{Mathew:2007vq}, Journal of fluid mechanics by Cambridge University Press. Reproduced with permission of Cambridge University Press in the format reprint in a journal via Copyright Clearance Center.)}
  \label{fig:sob-mixing}
\end{figure*}

The main difference between $M(n)$ and $E_{x}(n)$ is that the mixing metric compares \emph{instantaneous} advected density $P^n \rho$ to the prior $\mu$, which, in the dual Koopman framework, is equivalent to using $U^n f$. Recall that the ergodicity metric used the temporal averages $\frac{1}{N}\sum_{n=0}^{N-1} [U^n f]$ (cf. \eqref{eq:emp-measure}). The other difference between mixing and ergodicity indicators can be interpreted as the requirement on the scales at which measure $\mu$ is sampled: the ergodicity indicator $E_{x}(n)$ gives more weight to larger spherical sets, while in $M(n)$, differences over all sets are weighted the same, regardless of the set size,  due to rescaling of the measure $dx dr$ by the measure of the spherical set to $dxdr / \mu(x,r)$. 
 Therefore, those trajectories that lead to uniform decay of $E_{x}(n)$ will first appear well distributed in $\mu$ over larger sets, and only then will they sample $\mu$ on smaller scales. Conversely, $M(n)$ does not make such a distinction, requiring that the trajectories distribute on all scales equally fast.

In practice, the choice between $E_{x}(n)$ and $M(n)$ is driven by the application: for mixing of fluids, it was important that two reactants were in contact at the instant of reaction initiation, e.g., initiation of burning in the chamber. In the search-and-rescue application, instantaneous behavior was not of the essence, as the goal was to find the target over the entire course of the search mission, not necessarily to have the same (high) probability to locate the target at any particular time instance. In future applications, a similar analysis of the problem would decide which property, ergodicity $E_{x}(n)$ or mixing $M(n)$, would be a more appropriate design criterion.

The indicators for ergodicity and mixing presented in this section are based on representations of ergodic measures in the space of Fourier coefficients. They therefore represent ergodic measures as accurately as it is possible through their Fourier coefficients: numerically, approximations will converge quicker the smoother  the ergodic measure studied. However, it is well known that the shape of the ergodic measure's density does not determine uniquely the dynamical system. In this sense, the indicators are not intrinsic to the dynamical system studied. For example, if one would use a unique ergodic measure \(\mu\) for a map \(T_{1}\) and seek to drive the ergodic measure of another system \(T_{2}\) to match \(\mu\), there would be no guarantee that trajectories of \(T_{1}\) and \(T_{2}\) would be conjugate, without further restrictions on properties of maps \(T_{1}\) and \(T_{2}\).

% this is configuration for my Emacs editor so please don't erase it.

%%% Local Variables: 
%%% mode: latex
%%% TeX-master: "../koopmanism"
%%% End: 

\section{Conclusions}
We have reviewed the theoretical aspects and applications of the spectral theory
of the Koopman operator in dynamical systems. The use of these
concepts holds promise to provide a theory that extends and
complements tools from geometrical dynamical systems theory that
enabled so much development in science and technology over the last
century. The presented material described three parallel branches of Koopman operator analysis:
\begin{inparaenum}[(i)]
  \item Koopman mode analysis,
  \item eigenquotient analysis, and
  \item indicators of ergodicity and mixing.
\end{inparaenum}
The Koopman modes generalize the notion of linear eigenmodes, known from, e.g., linear mechanical vibration theory, to the nonlinear context, without linearizing the dynamics first. It is interesting that projections onto eigenspaces of the Koopman operator are achieved via an extension of Laplace and Fourier
analysis, and that such extension works in the nonlinear case.
The Koopman mode analysis has proved especially useful in analysis of dynamics in infinite-dimensional state spaces, which were  observed using a high-dimensional measurements, e.g., thermal dynamics of a building system, with a distributed measurement of temperatures or fluids experiments.

The eigenquotient analysis generalizes the analysis of smooth integrals of motion, which is central in theoretical mechanics, to systems which do not have any smooth invariants. The non-smooth eigenfunctions are constructed using infinite-time averaging along trajectories of the dynamical system. The eigenfunctions are used to construct a geometry for families of invariant sets which can be used to extract invariant regions that locally resemble phase portraits of integrable Hamiltonian systems. The same analysis can not only be applied to the study of invariant sets, but also periodic and wandering sets.

Finally, we showed how the trajectory averages can be used even when only finite-time data is available. The values of finite-time averages of functions along trajectories can be used to evaluate quantitative indicators of how closely the system is to accurately sampling a prior measure. Such indicators can be used as optimization criteria in a feedback algorithm for trajectory planning of Unmanned Aerial Vehicles.

The theory reviewed here was introduced and developed mostly for
measure-preserving deterministic systems, including dynamics on the attractor of a dissipative
system. Much more work remains to be done in the case of
dissipative systems and extensions to non-smooth and hybrid
(deterministic/stochastic) case.

%%% Local Variables: 
%%% mode: latex
%%% TeX-master: "../koopmanism"
%%% End: 

\section*{Acknowledgments}

The authors would like to thank the Editor, Dr. Holmes, and anonymous reviewers for their extensive comments that helped improve the manuscript.
The authors were funded by the following grants: ARO W911NF-11-1-0511,
AFOSR FA9550-10-1-0143, AFOSR FA9550-09-1-014, ONR MURI N00014-11-1-0087, ONR N00014- 07-1-0587, ONR N00014-10-1-0611.

\section*{References}
\label{sec:ref}


\begin{thebibliography}{61}
\providecommand{\natexlab}[1]{#1}
\providecommand{\url}[1]{\texttt{#1}}
\expandafter\ifx\csname urlstyle\endcsname\relax
  \providecommand{\doi}[1]{doi: #1}\else
  \providecommand{\doi}{doi: \begingroup \urlstyle{rm}\Url}\fi

\bibitem[Adams and Fournier(2003)]{Adams:2003wi}
Robert~A Adams and John J~F Fournier.
\newblock \emph{{Sobolev spaces}}, volume 140 of \emph{Pure and Applied
  Mathematics (Amsterdam)}.
\newblock Elsevier - Academic Press, Amsterdam, second edition, 2003.
\newblock ISBN 0-12-044143-8.
\newblock URL \url{http://www.ams.org/mathscinet-getitem?mr=MR2424078}.

\bibitem[Arnoldi(1951)]{Arnoldi:1951ub}
W~E Arnoldi.
\newblock {The principle of minimized iteration in the solution of the matrix
  eigenvalue problem}.
\newblock \emph{Quarterly of Applied Mathematics}, 9:\penalty0 1, 1951.
\newblock URL
  \url{http://www.ams.org.proxy.library.ucsb.edu:2048/mathscinet/pdf/42792.pdf?pg1=MR&s1=13:163e&loc=fromreflist}.

\bibitem[Assani(2003)]{Assani:2003us}
Idris Assani.
\newblock \emph{{Wiener Wintner ergodic theorems}}.
\newblock World Scientific Publishing Co. Inc., River Edge, NJ, 2003.
\newblock ISBN 981-02-4439-8.
\newblock URL \url{http://www.ams.org/mathscinet-getitem?mr=MR1995517}.

\bibitem[Bolsinov and Fomenko(2004)]{Bolsinov:2004dw}
A~V Bolsinov and A~T Fomenko.
\newblock \emph{{Integrable Hamiltonian systems}}.
\newblock CRC Press, 2004.
\newblock ISBN 0-415-29805-9.
\newblock \doi{10.1201/9780203643426}.
\newblock URL \url{http://dx.doi.org/10.1201/9780203643426}.

\bibitem[Budi{\v s}i{\'c}(2012)]{Budisic:2012td}
Marko Budi{\v s}i{\'c}.
\newblock \emph{{Ergodic Quotients in Analysis of Dynamical Systems}}.
\newblock PhD thesis, UC Santa Barbara, August 2012.

\bibitem[Budi{\v s}i{\'c} and Mezi{\'c}(2009)]{Budisic:2009iy}
Marko Budi{\v s}i{\'c} and Igor Mezi{\'c}.
\newblock {An approximate parametrization of the ergodic partition using time
  averaged observables}.
\newblock In \emph{Proceedings of the 48th IEEE Conference on Decision and
  Control, held jointly with the 2009 28th Chinese Control Conference. CDC/CCC
  2009.}, pages 3162--3168, 2009.
\newblock \doi{10.1109/CDC.2009.5400512}.
\newblock URL
  \url{http://ieeexplore.ieee.org/xpl/articleDetails.jsp?arnumber=5400512}.

\bibitem[Budi{\v s}i{\'c} and Mezi{\'c}(2012)]{Budisic:2012woa}
Marko Budi{\v s}i{\'c} and Igor Mezi{\'c}.
\newblock {Geometry of the ergodic quotient reveals coherent structures in
  flows}.
\newblock \emph{Physica D. Nonlinear Phenomena}, 241\penalty0 (15):\penalty0
  1255--1269, August 2012.
\newblock \doi{10.1016/j.physd.2012.04.006}.
\newblock URL
  \url{http://www.sciencedirect.com/science/article/pii/S0167278912001108}.

\bibitem[Chen et~al.(2012)Chen, Tu, and Rowley]{Chen:2012jh}
Kevin~K Chen, Jonathan~H Tu, and Clarence~W Rowley.
\newblock {Variants of Dynamic Mode Decomposition: Boundary Condition, Koopman,
  and Fourier Analyses}.
\newblock \emph{Journal of Nonlinear Science}, April 2012.
\newblock \doi{10.1007/s00332-012-9130-9}.
\newblock URL
  \url{http://www.springerlink.com/index/10.1007/s00332-012-9130-9}.

\bibitem[Coifman and Lafon(2006)]{Coifman:2006cy}
Ronald Coifman and St{\'e}phane Lafon.
\newblock {Diffusion maps}.
\newblock \emph{Applied and Computational Harmonic Analysis}, 21:\penalty0
  5--30, January 2006.
\newblock \doi{10.1016/j.acha.2006.04.006}.
\newblock URL
  \url{http://www.cs.tau.ac.il/~shekler/Workshop_2007a/Papers/gm_elsevier.pdf}.

\bibitem[Coifman et~al.(2005)Coifman, Lafon, Lee, Maggioni, Nadler, Warner, and
  Zucker]{Coifman:2005bk}
Ronald Coifman, S~Lafon, AB~Lee, M~Maggioni, B~Nadler, F~Warner, and S~W
  Zucker.
\newblock {Geometric diffusions as a tool for harmonic analysis and structure
  definition of data: Diffusion maps}.
\newblock \emph{Proceedings Of The National Academy Of Sciences Of The United
  States Of America}, 102\penalty0 (21):\penalty0 7426--7431, 2005.
\newblock \doi{10.1073/pnas.0500334102}.
\newblock URL
  \url{http://gateway.webofknowledge.com/gateway/Gateway.cgi?GWVersion=2&SrcAuth=mekentosj&SrcApp=Papers&DestLinkType=FullRecord&DestApp=WOS&KeyUT=000229417500007}.

\bibitem[Crawley et~al.(2000)Crawley, Pedersen, Lawrie, and
  Winkelmann]{Crawley:2000tc}
Drury~B Crawley, Curtis~O Pedersen, Linda~K Lawrie, and Frederick~C Winkelmann.
\newblock {EnergyPlus: Energy Simulation Program}.
\newblock \emph{ASHRAE Journal}, 42:\penalty0 49--56, 2000.

\bibitem[Cvitanovi{\'c} et~al.(2010)Cvitanovi{\'c}, Artuso, Mainieri, Tanner,
  and Vattay]{ChaosBook}
Predrag Cvitanovi{\'c}, R~Artuso, R~Mainieri, G~Tanner, and G~Vattay.
\newblock \emph{{Chaos: Classical and Quantum}}.
\newblock Niels Bohr Institute, Copenhagen, stable 13th edition, 2010.
\newblock URL \url{http://ChaosBook.org}.

\bibitem[D'Alessandro et~al.(2002)D'Alessandro, Mezi{\'c}, and
  Dahleh]{DAlessandro:2002vy}
D~D'Alessandro, Igor Mezi{\'c}, and M~Dahleh.
\newblock {Statistical properties of controlled fluid flows with applications
  to control of mixing}.
\newblock \emph{Systems and Control Letters}, 45\penalty0 (4):\penalty0
  249--256, 2002.
\newblock \doi{10.1016/S0167-6911(01)00181-5}.
\newblock URL \url{http://dx.doi.org/10.1016/S0167-6911(01)00181-5}.

\bibitem[Dellnitz and Junge(1999)]{Dellnitz:1999tr}
Michael Dellnitz and Oliver Junge.
\newblock {On the approximation of complicated dynamical behavior}.
\newblock \emph{SIAM Journal on Numerical Analysis}, 36\penalty0 (2):\penalty0
  491--515, 1999.
\newblock URL \url{http://www.jstor.org/stable/2587207}.

\bibitem[Dragovi{\'c} and Radnovi{\'c}(2009)]{Dragovic:2009tf}
Vladimir Dragovi{\'c} and Milena Radnovi{\'c}.
\newblock {Bifurcations of Liouville tori in elliptical billiards}.
\newblock \emph{Regular and Chaotic Dynamics. International Scientific
  Journal}, 14\penalty0 (4-5):\penalty0 479--494, 2009.
\newblock \doi{10.1134/S1560354709040054}.
\newblock URL \url{http://dx.doi.org/10.1134/S1560354709040054}.

\bibitem[Eisenhower et~al.(2010)Eisenhower, Maile, Fischer, and
  Mezi{\'c}]{Eisenhower:2010tv}
Bryan Eisenhower, Tobias Maile, Martin Fischer, and Igor Mezi{\'c}.
\newblock {Decomposing Building System Data for Model Validation and Analysis
  Using the Koopman Operator}.
\newblock \emph{SimBuild 2010 - Fourth National Conference of IBPSA-USA},
  August 2010.
\newblock URL
  \url{http://www.engineering.ucsb.edu/~mgroup/wiki/images/a/a8/SB10-DOC-TS08B-03-Eisenhower.pdf}.

\bibitem[Gaunersdorfer(1992)]{Gaunersdorfer:1992cg}
Andrea Gaunersdorfer.
\newblock {Time averages for heteroclinic attractors}.
\newblock \emph{SIAM Journal on Applied Mathematics}, 52\penalty0 (5):\penalty0
  1476--1489, 1992.
\newblock \doi{10.1137/0152085}.
\newblock URL \url{http://dx.doi.org/10.1137/0152085}.

\bibitem[Georgescu et~al.(2012)Georgescu, Eisenhower, and
  Mezi{\'c}]{Georgescu:2012tt}
Michael Georgescu, Bryan Eisenhower, and Igor Mezi{\'c}.
\newblock {Creating zoning approximations to building energy models using the
  Koopman Operator}.
\newblock In \emph{SimBuild 2012 - Fifth National Conference of IBPSA-USA (to
  appear)}, August 2012.

\bibitem[Germano et~al.(1991)Germano, Piomelli, Moin, and
  Cabot]{Germano:1991te}
M~Germano, U~Piomelli, P~Moin, and W~H Cabot.
\newblock {A dynamic subgrid‐scale eddy viscosity model}.
\newblock \emph{Physics of Fluids A: Fluid Dynamics}, 3:\penalty0 6, 1991.
\newblock URL \url{http://link.aip.org/link/PFADEB/v3/i7/p1760/s1}.

\bibitem[Jones(2001)]{Jones:2001tj}
Christopher K. R.~T. Jones.
\newblock {Whither applied nonlinear dynamics?}
\newblock In B~Engquist and W~Schmid, editors, \emph{Mathematics
  unlimited---2001 and beyond}, pages 631--645. Springer, Berlin, 2001.
\newblock URL \url{http://www.ams.org/mathscinet-getitem?mr=MR1852181}.

\bibitem[Katok and Hasselblatt(1995)]{Katok:1995th}
Anatole Katok and Boris Hasselblatt.
\newblock \emph{{Introduction to the modern theory of dynamical systems}},
  volume~54 of \emph{Encyclopedia of Mathematics and its Applications}.
\newblock Cambridge University Press, Cambridge, 1995.
\newblock ISBN 0-521-34187-6.
\newblock URL \url{http://www.ams.org/mathscinet-getitem?mr=MR1326374}.

\bibitem[Koopman(1931)]{Koopman:1931ug}
B~O Koopman.
\newblock {Hamiltonian Systems and Transformations in Hilbert Space}.
\newblock \emph{Proceedings of National Academy of Sciences}, 17\penalty0
  (5):\penalty0 315--318, May 1931.
\newblock URL \url{http://www.jstor.org/stable/86114}.

\bibitem[Krengel(1985)]{Krengel:1985ew}
Ulrich Krengel.
\newblock \emph{{Ergodic theorems}}, volume~6 of \emph{de Gruyter Studies in
  Mathematics}.
\newblock Walter de Gruyter and Co., Berlin, 1985.
\newblock ISBN 3-11-008478-3.
\newblock \doi{10.1515/9783110844641}.
\newblock URL \url{http://dx.doi.org/10.1515/9783110844641}.

\bibitem[Lasota and Mackey(1994)]{Lasota:1994vt}
Andrzej Lasota and Michael~C Mackey.
\newblock \emph{{Chaos, Fractals, and Noise: Stochastic Aspects of Dynamics}},
  volume~97 of \emph{Applied Mathematical Sciences}.
\newblock Springer-Verlag, New York, second edition, 1994.
\newblock ISBN 0-387-94049-9.
\newblock URL \url{http://www.ams.org/mathscinet-getitem?mr=MR1244104}.

\bibitem[Levnaji{\'c} and Mezi{\'c}(2008)]{Levnajic:2008vo}
Zoran Levnaji{\'c} and Igor Mezi{\'c}.
\newblock {Ergodic Theory and Visualization II: Visualization of Resonances and
  Periodic Sets}.
\newblock \emph{arXiv.org}, nlin.CD, August 2008.
\newblock URL \url{http://arxiv.org/abs/0808.2182v1}.

\bibitem[Levnaji{\'c} and Mezi{\'c}(2010)]{Levnajic:2010gq}
Zoran Levnaji{\'c} and Igor Mezi{\'c}.
\newblock {Ergodic theory and visualization. I. Mesochronic plots for
  visualization of ergodic partition and invariant sets}.
\newblock \emph{Chaos: An Interdisciplinary Journal of Nonlinear Science},
  20\penalty0 (3):\penalty0 --, 2010.
\newblock \doi{10.1063/1.3458896}.
\newblock URL \url{http://dx.doi.org/10.1063/1.3458896}.

\bibitem[Ma{\~n}{\'e}(1987)]{Mane:1987wz}
Ricardo Ma{\~n}{\'e}.
\newblock \emph{{Ergodic theory and differentiable dynamics}}, volume~8 of
  \emph{Ergebnisse der Mathematik und ihrer Grenzgebiete (3) [Results in
  Mathematics and Related Areas (3)]}.
\newblock Springer-Verlag, Berlin, 1987.
\newblock ISBN 3-540-15278-4.
\newblock URL \url{http://www.ams.org/mathscinet-getitem?mr=MR889254}.

\bibitem[Mathew and Mezi{\'c}(2009)]{Mathew:2009et}
George Mathew and Igor Mezi{\'c}.
\newblock {Spectral multiscale coverage: A uniform coverage algorithm for
  mobile sensor networks}.
\newblock \emph{Decision and Control, 2009 held jointly with the 2009 28th
  Chinese Control Conference. CDC/CCC 2009. Proceedings of the 48th IEEE
  Conference on}, pages 7872--7877, 2009.
\newblock URL
  \url{http://ieeexplore.ieee.org/xpls/abs_all.jsp?arnumber=5400401}.

\bibitem[Mathew and Mezi{\'c}(2011)]{Mathew:2011ev}
George Mathew and Igor Mezi{\'c}.
\newblock {Metrics for ergodicity and design of ergodic dynamics for
  multi-agent systems}.
\newblock \emph{Physica D. Nonlinear Phenomena}, 240\penalty0 (4-5):\penalty0
  432--442, February 2011.
\newblock \doi{10.1016/j.physd.2010.10.010}.
\newblock URL
  \url{http://linkinghub.elsevier.com/retrieve/pii/S016727891000285X}.

\bibitem[Mathew et~al.(2005)Mathew, Mezi{\'c}, and Petzold]{Mathew:2005eq}
George Mathew, Igor Mezi{\'c}, and Linda Petzold.
\newblock {A multiscale measure for mixing}.
\newblock \emph{Physica D. Nonlinear Phenomena}, 211\penalty0 (1-2):\penalty0
  23--46, 2005.
\newblock \doi{10.1016/j.physd.2005.07.017}.

\bibitem[Mathew et~al.(2007)Mathew, Mezi{\'c}, Grivopoulos, Vaidya, and
  Petzold]{Mathew:2007vq}
George Mathew, Igor Mezi{\'c}, Symeon Grivopoulos, Umesh Vaidya, and Linda
  Petzold.
\newblock {Optimal control of mixing in Stokes fluid flows}.
\newblock \emph{Journal of Fluid Mechanics}, 2007.
\newblock URL \url{http://journals.cambridge.org/abstract_S0022112007005332}.

\bibitem[Meyer(2000)]{Meyer:2000tv}
Carl Meyer.
\newblock \emph{{Matrix analysis and applied linear algebra}}.
\newblock Society for Industrial and Applied Mathematics (SIAM), Philadelphia,
  PA, 2000.
\newblock ISBN 0-89871-454-0.
\newblock \doi{10.1137/1.9780898719512}.
\newblock URL \url{http://dx.doi.org/10.1137/1.9780898719512}.

\bibitem[Mezi{\'c}(1994)]{Mezic:1994tv}
Igor Mezi{\'c}.
\newblock \emph{{On geometrical and statistical properties of dynamical
  systems: theory and applications}}.
\newblock PhD thesis, California Institute of Technology, 1994.
\newblock URL \url{http://resolver.caltech.edu/CaltechETD:etd-07212005-131406}.

\bibitem[Mezi{\'c}(2005)]{Mezic:2005ji}
Igor Mezi{\'c}.
\newblock {Spectral properties of dynamical systems, model reduction and
  decompositions}.
\newblock \emph{Nonlinear Dynamics. An International Journal of Nonlinear
  Dynamics and Chaos in Engineering Systems}, 41\penalty0 (1-3):\penalty0
  309--325, 2005.
\newblock \doi{10.1007/s11071-005-2824-x}.
\newblock URL \url{http://dx.doi.org/10.1007/s11071-005-2824-x}.

\bibitem[Mezi{\'c} and Banaszuk(2000)]{Mezic:2000tm}
Igor Mezi{\'c} and Andrzej Banaszuk.
\newblock {Comparison of systems with complex behavior: spectral methods}.
\newblock \emph{Decision and Control, 2000. Proceedings of the 39th IEEE
  Conference on}, page~8, 2000.
\newblock URL
  \url{http://ieeexplore.ieee.org/xpls/abs_all.jsp?arnumber=912022}.

\bibitem[Mezi{\'c} and Banaszuk(2004)]{Mezic:2004is}
Igor Mezi{\'c} and Andrzej Banaszuk.
\newblock {Comparison of systems with complex behavior}.
\newblock \emph{Physica D. Nonlinear Phenomena}, 197\penalty0 (1-2):\penalty0
  101--133, January 2004.
\newblock \doi{10.1016/j.physd.2004.06.015}.
\newblock URL
  \url{http://linkinghub.elsevier.com/retrieve/pii/S0167278904002507}.

\bibitem[Mezi{\'c} and Wiggins(1999)]{Mezic:1999fu}
Igor Mezi{\'c} and Stephen Wiggins.
\newblock {A method for visualization of invariant sets of dynamical systems
  based on the ergodic partition}.
\newblock \emph{Chaos: An Interdisciplinary Journal of Nonlinear Science},
  9\penalty0 (1):\penalty0 213--218, January 1999.
\newblock \doi{10.1063/1.166399}.

\bibitem[Note1()]{Note1}
Note1.
\newblock \protect \url {http://apps1.eere.energy.gov/buildings/energyplus/}.

\bibitem[Perry and Wiggins(1994)]{Perry:1994wl}
Anthony~D Perry and Stephen Wiggins.
\newblock {KAM tori are very sticky: rigorous lower bounds on the time to move
  away from an invariant Lagrangian torus with linear flow}.
\newblock \emph{Physica D. Nonlinear Phenomena}, 71\penalty0 (1-2):\penalty0
  102--121, 1994.
\newblock URL
  \url{http://www.ams.org/mathscinet/search/publications.html?pg1=MR&s1=MR1264111}.

\bibitem[Petersen(1989)]{Petersen:1989uv}
Karl Petersen.
\newblock \emph{{Ergodic theory}}, volume~2 of \emph{Cambridge Studies in
  Advanced Mathematics}.
\newblock Cambridge University Press, Cambridge, UK, 1989.
\newblock ISBN 0-521-38997-6.
\newblock URL \url{http://www.ams.org/mathscinet-getitem?mr=MR1073173}.

\bibitem[Rokhlin(1949)]{Rokhlin:1949te}
Vladimir~Abramovich Rokhlin.
\newblock {Selected topics from the metric theory of dynamical systems}.
\newblock \emph{Rossi\u\i skaya Akademiya Nauk. Moskovskoe Matematicheskoe
  Obshchestvo. Uspekhi Matematicheskikh Nauk}, 4\penalty0 (2(30)):\penalty0
  57--128, 1949.
\newblock URL \url{http://www.ams.org/mathscinet-getitem?mr=MR0030710}.

\bibitem[Rokhlin(1966)]{Rokhlin:1966te}
Vladimir~Abramovich Rokhlin.
\newblock {Selected topics from the metric theory of dynamical systems}.
\newblock \emph{American Mathematical Society Translations, Series 2, Vol. 49,
  Ten Papers on Functional Analysis and Measure Theory}, 49:\penalty0 171--240,
  1966.

\bibitem[Rowley et~al.(2009)Rowley, Mezi{\'c}, Bagheri, Schlatter, and
  Henningson]{Rowley:2009ez}
Clarence~W Rowley, Igor Mezi{\'c}, Shervin Bagheri, Philipp Schlatter, and
  Dan~S Henningson.
\newblock {Spectral analysis of nonlinear flows}.
\newblock \emph{Journal of Fluid Mechanics}, 641:\penalty0 115--127, 2009.
\newblock \doi{10.1017/S0022112009992059}.
\newblock URL \url{http://dx.doi.org/10.1017/S0022112009992059}.

\bibitem[Ruhe(1984)]{Ruhe:1984tr}
Axel Ruhe.
\newblock {Rational Krylov sequence methods for eigenvalue computation}.
\newblock \emph{Linear Algebra and Its Applications}, 58:\penalty0 391--405,
  1984.
\newblock \doi{10.1016/0024-3795(84)90221-0}.
\newblock URL \url{http://dx.doi.org/10.1016/0024-3795(84)90221-0}.

\bibitem[Rypina et~al.(2011)Rypina, Scott, Pratt, and Brown]{Rypina:2011ec}
I~Rypina, S.~E. Scott, L.~J. Pratt, and M.~G. Brown.
\newblock {Investigating the connection between complexity of isolated
  trajectories and Lagrangian coherent structures}.
\newblock \emph{Nonlinear Processes in Geophysics}, 18\penalty0 (6):\penalty0
  977--987, 2011.
\newblock \doi{10.5194/npg-18-977-2011}.
\newblock URL \url{http://www.nonlin-processes-geophys.net/18/977/2011/}.

\bibitem[Saad and Schultz(1986)]{Saad:1986hx}
Youcef Saad and Martin~H Schultz.
\newblock {GMRES: A Generalized Minimal Residual Algorithm for Solving
  Nonsymmetric Linear Systems}.
\newblock \emph{SIAM Journal on Scientific and Statistical Computing},
  7\penalty0 (3):\penalty0 14, July 1986.
\newblock \doi{10.1137/0907058}.
\newblock URL \url{http://epubs.siam.org/doi/abs/10.1137/0907058}.

\bibitem[Schmid(2010)]{Schmid:2010ba}
Peter~J Schmid.
\newblock {Dynamic mode decomposition of numerical and experimental data}.
\newblock \emph{Journal of Fluid Mechanics}, 656:\penalty0 24, August 2010.
\newblock \doi{10.1017/S0022112010001217}.

\bibitem[Scott et~al.(2009)Scott, Redd, Kuznetsov, Mezi{\'c}, and
  Jones]{Scott:2009wh}
Sherry~E. Scott, Thomas~C. Redd, Leonid Kuznetsov, Igor Mezi{\'c}, and
  Christopher K. R.~T. Jones.
\newblock {Capturing deviation from ergodicity at different scales}.
\newblock \emph{Physica D. Nonlinear Phenomena}, 238\penalty0 (16):\penalty0
  1668--1679, 2009.
\newblock \doi{10.1016/j.physd.2009.05.003}.
\newblock URL \url{http://dx.doi.org/10.1016/j.physd.2009.05.003}.

\bibitem[Seena and Sung(2011)]{Seena:2011ft}
Abu Seena and Hyung~Jin Sung.
\newblock {Dynamic mode decomposition of turbulent cavity flows for
  self-sustained oscillations}.
\newblock \emph{International Journal of Heat and Fluid Flow}, 32\penalty0
  (6):\penalty0 13, December 2011.
\newblock \doi{10.1016/j.ijheatfluidflow.2011.09.008}.
\newblock URL
  \url{http://linkinghub.elsevier.com/retrieve/pii/S0142727X11001275}.

\bibitem[Shi et~al.(2008)Shi, Lai, Kern, Sicotte, Dinov, and Toga]{Shi:2008ef}
Yonggang Shi, Rongjie Lai, Kyle Kern, Nancy Sicotte, Ivo Dinov, and Arthur
  Toga.
\newblock {Harmonic Surface Mapping with Laplace-Beltrami Eigenmaps}.
\newblock In Dimitris Metaxas, Leon Axel, Gabor Fichtinger, and G{\'a}bor
  Sz{\'e}kely, editors, \emph{Medical Image Computing and Computer-Assisted
  Intervention -- MICCAI 2008}, pages 147--154. Springer Berlin / Heidelberg,
  2008.
\newblock ISBN 978-3-540-85989-5.
\newblock \doi{10.1007/978-3-540-85990-1{\_}18}.
\newblock URL \url{http://dx.doi.org/10.1007/978-3-540-85990-1_18}.

\bibitem[Sine(1968)]{Sine:1968vj}
Robert Sine.
\newblock {Geometric theory of a single Markov operator}.
\newblock \emph{Pacific Journal of Mathematics}, 27:\penalty0 155--166, 1968.
\newblock URL \url{http://www.ams.org/mathscinet-getitem?mr=MR0240281}.

\bibitem[Susuki and Mezi{\'c}(2009)]{Susuki:2009cw}
Yoshihiko Susuki and Igor Mezi{\'c}.
\newblock {Ergodic partition of phase space in continuous dynamical systems}.
\newblock In \emph{Decision and Control, 2009 held jointly with the 2009 28th
  Chinese Control Conference. CDC/CCC 2009. Proceedings of the 48th IEEE
  Conference on}, pages 7497--7502, 2009.
\newblock \doi{10.1109/CDC.2009.5400911}.
\newblock URL \url{http://dx.doi.org/10.1109/CDC.2009.5400911}.

\bibitem[Susuki and Mezi{\'c}(2010)]{Susuki:2010ei}
Yoshihiko Susuki and Igor Mezi{\'c}.
\newblock {Nonlinear Koopman modes of coupled swing dynamics and coherency
  identification}.
\newblock \emph{Power and Energy Society General Meeting, 2010 IEEE}, pages
  1--8, 2010.
\newblock \doi{10.1109/PES.2010.5589363}.
\newblock URL
  \url{http://ieeexplore.ieee.org/search/freesrchabstract.jsp?tp=&arnumber=5589363}.

\bibitem[Susuki and Mezi{\'c}(2011{\natexlab{a}})]{Susuki:2011ef}
Yoshihiko Susuki and Igor Mezi{\'c}.
\newblock {Nonlinear Koopman Modes and Coherency Identification of Coupled
  Swing Dynamics}.
\newblock \emph{Power Systems, IEEE Transactions on}, 26\penalty0 (4):\penalty0
  1894--1904, 2011{\natexlab{a}}.
\newblock \doi{10.1109/TPWRS.2010.2103369}.
\newblock URL
  \url{http://ieeexplore.ieee.org/search/freesrchabstract.jsp?tp=&arnumber=5713215}.

\bibitem[Susuki and Mezi{\'c}(2011{\natexlab{b}})]{Susuki:2011jq}
Yoshihiko Susuki and Igor Mezi{\'c}.
\newblock {Correction to ``Nonlinear Koopman Modes and Coherency Identification
  of Coupled Swing Dynamics'' [Nov 11 1894-1904]}.
\newblock \emph{Power Systems, IEEE Transactions on}, 26\penalty0 (4):\penalty0
  2584, 2011{\natexlab{b}}.
\newblock \doi{10.1109/TPWRS.2011.2165650}.
\newblock URL
  \url{http://ieeexplore.ieee.org/search/freesrchabstract.jsp?tp=&arnumber=6029304}.

\bibitem[Susuki and Mezi{\'c}(2012)]{Susuki:2012gk}
Yoshihiko Susuki and Igor Mezi{\'c}.
\newblock {Nonlinear Koopman Modes and a Precursor to Power System Swing
  Instabilities}.
\newblock \emph{Power Systems, IEEE Transactions on}, PP\penalty0
  (99):\penalty0 1, 2012.
\newblock \doi{10.1109/TPWRS.2012.2183625}.
\newblock URL
  \url{http://ieeexplore.ieee.org/xpl/articleDetails.jsp?tp=&arnumber=6165682&contentType=Early+Access+Articles&matchBoolean%3Dtrue%26rowsPerPage%3D30%26searchField%3DSearch_All%26queryText%3D%28p_DOI%3A10.1109%2FTPWRS.2012.2183625%29}.

\bibitem[Treschev(1998)]{Treschev:1998uq}
D~Treschev.
\newblock {Width of stochastic layers in near-integrable two-dimensional
  symplectic maps}.
\newblock \emph{Physica D. Nonlinear Phenomena}, 116\penalty0 (1-2):\penalty0
  21--43, 1998.
\newblock URL
  \url{http://www.ams.org/mathscinet/search/publications.html?pg1=MR&s1=MR1621888}.

\bibitem[Wichtrey(2010)]{Wichtrey:2010ta}
Tobias Wichtrey.
\newblock \emph{{Harmonic Limits of Dynamical and Control Systems}}.
\newblock PhD thesis, Logos Verlag, Berlin, 2010.
\newblock URL \url{http://d-nb.info/1009920588}.

\bibitem[Wiener and Wintner(1941)]{Wiener:1941wy}
Norbert Wiener and Aurel Wintner.
\newblock {Harmonic analysis and ergodic theory}.
\newblock \emph{American Journal of Mathematics}, 63:\penalty0 415--426, 1941.
\newblock URL \url{http://www.ams.org/mathscinet-getitem?mr=MR0004098}.

\bibitem[Yosida(1995)]{Yosida:1995ul}
Kosaku Yosida.
\newblock \emph{{Functional Analysis}}.
\newblock Springer-Verlag, 1995.
\newblock ISBN 3-540-58654-7.
\newblock URL
  \url{http://www.ams.org.proxy.library.ucsb.edu:2048/mathscinet/search/publications.html?fmt=bibtex&pg1=MR&s1=1336382}.

\bibitem[Young(2002)]{Young:2002vc}
Lai-Sang Young.
\newblock {What are SRB measures, and which dynamical systems have them?}
\newblock \emph{Journal of Statistical Physics}, 108\penalty0 (5-6):\penalty0
  733--754, 2002.
\newblock \doi{10.1023/A:1019762724717}.
\newblock URL \url{http://dx.doi.org/10.1023/A:1019762724717}.

\end{thebibliography}
\end{document}